\documentclass[11pt,a4paper]{amsart}

\usepackage{a4wide}

\usepackage{amsfonts,amsthm,amssymb,mathabx}
\usepackage[pdftex]{graphicx}
\usepackage{graphics}
\usepackage[matrix,arrow]{xy}
\usepackage[active]{srcltx}

 \usepackage[colorlinks=true]{hyperref}
 \hypersetup{
     colorlinks   = true,
     citecolor    = blue
}

\newtheorem{thm}{Theorem}
\numberwithin{thm}{section}
\numberwithin{equation}{section}
\newtheorem{theorem}[thm]{Theorem}
\newtheorem*{theorem*}{Theorem}

\newtheorem{corollary}[thm]{Corollary}
\newtheorem*{corollary*}{Corollary}
\newtheorem{lemma}[thm]{Lemma}

\newtheorem{prop}[thm]{Proposition}
\newtheorem{proposition}[thm]{Proposition}

\newtheorem*{conjecture*}{Conjecture}

\newtheorem*{question*}{Question}
\newtheorem{definition}[thm]{Definition}
\newtheorem*{definition*}{Definition}

\newtheorem*{definitions*}{Definitions}

\newtheorem*{rem*}{Remark}
\theoremstyle{remark}
\newtheorem{remark}[thm]{Remark}

\newtheorem*{remark*}{Remark}

\newtheorem*{remarks*}{Remarks}
\newtheorem*{example*}{Example}

\newtheorem*{examples*}{Examples}

\newcommand{\R}{\mathbb{R}}
\newcommand{\Z}{\mathbb{Z}}
\newcommand{\Q}{\mathbb{Q}}
\newcommand{\C}{\mathbb{C}}
\newcommand{\N}{\mathbb{N}}
\newcommand{\T}{\mathbb{T}}

\def\CP{{\mathbb C}{\mathbb P}}
\def\RP{{\mathbb R}P}
\def\PP{{\mathbb P}}

\newcommand{\ep}{\epsilon}
\newcommand{\ga}{\gamma}
\newcommand{\Ga}{\Gamma}

\newcommand{\Sp}{\text{Sp}}

\def\lg{\langle}
\def\rg{\rangle}

\def\cz{{\mu_\text{CZ}}}
\def\rcz{{\tilde\mu_\text{CZ}}}
\def\rs{{\mu_\text{RS}}}
\def\maslov{{\mu_\text{Maslov}}}
\def\czl{{\mu^-_\text{CZ}}}
\def\czu{{\mu^+_\text{CZ}}}
\def\hwz{{\mu_\text{HWZ}}}
\def\morse{{\mu_\text{Morse}}}
\def\vr{\varphi}
\def\om{\omega}
\def\wtl{\widetilde}
\def\ba{{\widebar a}}
\def\balpha{{\bar\alpha}}
\def\bbeta{{\bar\beta}}

\def\bxi{{\bar\xi}}
\def\bpi{{\bar\pi}}
\def\bvr{{\bar\varphi}}
\def\bM{{\widebar M}}
\def\bW{{\widebar W}}
\def\pr{\prime}
\def\P{{\mathcal P}}
\def\M{{\mathcal M}}

\def\g{{\mathfrak g}}
\def\J{{\mathcal J}}
\def\Jreg{{\mathcal J_{\text{reg}}}}
\def\mi{{\text{min}}}
\def\ma{{\text{max}}}
\def\u{{\widetilde u}}
\def\S{{\mathcal S}}
\def\B{{\mathcal B}}
\def\Om{\Omega}

\newcommand{\fk}{{\mathfrak{k}}}
\newcommand{\Xx}{\mathcal{X}}
\newcommand{\Nn}{\mathcal{N}}
\newcommand{\Pp}{\mathcal{P}}

\newcommand{\tgamma}{\tilde{\gamma}}
\newcommand{\rank}{rank\,}

\begin{document}
\title[Dynamical convexity and elliptic periodic orbits for Reeb flows]{Dynamical convexity and elliptic\\ periodic orbits for Reeb flows}

\author[M.~Abreu]{Miguel Abreu}
\address{Centro de An\'{a}lise Matem\'{a}tica, Geometria e Sistemas
Din\^{a}micos, Departamento de Matem\'atica, Instituto Superior T\'ecnico, Av.
Rovisco Pais, 1049-001 Lisboa, Portugal}
\email{mabreu@math.tecnico.ulisboa.pt}
 
\author[L.~Macarini]{Leonardo Macarini}
\address{Universidade Federal do Rio de Janeiro, Instituto de Matem\'atica,
Cidade Universit\'aria, CEP 21941-909 - Rio de Janeiro - Brazil}
\email{leomacarini@gmail.com}

\thanks{MA was partially funded by FCT/Portugal through projects 
PEst-OE/EEI/LA0009/2013, EXCL/MAT-GEO/0222/2012 and PTDC/MAT/117762/2010.
LM was partially supported by CNPq, Brazil.
The present work is part of the authors activities within BREUDS, 
a research partnership between European and Brazilian research groups in dynamical systems,
supported by an FP7 International Research Staff Exchange Scheme (IRSES) 
grant of the European Union.}


\begin{abstract}
A long-standing conjecture in Hamiltonian Dynamics states that the Reeb flow of any convex hypersurface in $\R^{2n}$ carries an elliptic closed orbit. Two important contributions toward its proof were given by Ekeland in 1986 and Dell'Antonio-D'Onofrio-Ekeland in 1995 proving this for convex hypersurfaces satisfying suitable pinching conditions and for antipodal invariant convex hypersurfaces respectively. In this work we present a generalization of these results using contact homology and a notion of dynamical convexity first introduced by Hofer-Wysocki-Zehnder for tight contact forms on $S^3$. Applications include geodesic flows under pinching conditions, magnetic flows and toric contact manifolds.
\end{abstract}

\maketitle

\section{Introduction}
\label{sec:intro}

Let $(M^{2n+1},\xi)$ be a closed (i.e. compact without boundary) co-oriented contact manifold. Let $\alpha$ be a contact form supporting $\xi$ and denote by $R_\alpha$ the corresponding Reeb vector field uniquely characterized by the equations $\iota_{R_\alpha} d\alpha = 0$ and $\alpha (R_\alpha) = 1$. We say that a periodic orbit $\gamma$ of $R_\alpha$ is \emph{elliptic}\footnote{Some authors rather call this an {\it elliptic-parabolic} periodic orbit .} if every eigenvalue of its linearized Poincar\'e map has modulus one.

The existence of elliptic orbits has important consequences for the Reeb flow. As a matter of fact, fundamental contributions due to Arnold, Birkhoff, Kolmogorov, Moser, Newhouse, Smale and Zehnder establish that, under generic conditions, the presence of elliptic orbits implies a rich dynamics: it forces the existence of a subset of positive measure filled by invariant tori (this implies, in particular, that the flow is not ergodic with respect to the volume measure), existence of transversal homoclinic connections and positivity of the topological entropy; see \cite{AbMar} and references therein.

The search for elliptic closed orbits goes back to Poincar\'e \cite{Po} and a long-standing conjecture in Hamiltonian Dynamics states that the Reeb flow of every convex hypersurface in $\R^{2n}$ carries an elliptic periodic orbit \cite{Eke,Lon02}. By the dynamical consequences of the existence of elliptic orbits mentioned before, this conjecture leads to deep implications: if it holds then Boltzmann's ergodic hypothesis would fail for a large class of important and natural Hamiltonian systems.

Unfortunately this is still an open question, but there are some important partial results. The first one, proved by Ekeland \cite{Eke86}, asserts the existence of an elliptic closed orbit on a regular energy level of a convex Hamiltonian in $\R^{2n}$ satisfying suitable \emph{pinching conditions} on the Hessian. The second one, due to Dell'Antonio-D'Onofrio-Ekeland \cite{DDE}, establishes that if an hypersurface in $\R^{2n}$ is convex and \emph{invariant by the antipodal map} then its Reeb flow carries an elliptic closed orbit. The proofs are based on classical variational methods and use in a strong way the hypothesis of convexity.

The goal of this work is to generalize these results for contact manifolds using \emph{contact homology} and a neck-stretching argument. In order to do this, we introduce a generalization of the notion of dynamical convexity originally defined by Hofer, Wysocki and Zender \cite{HWZ} for tight contact forms on $S^3$. Several applications are given, like the existence of elliptic closed geodesics for pinched Finsler metrics and the existence of elliptic closed orbits on sufficiently small energy levels of magnetic flows given by a non-degenerate magnetic field on a closed orientable surface of genus different from one.

Finally, it should be mentioned that, although the general machinery of contact homology is yet to be fully put on a rigorous basis \cite{HWZ10, HWZ11, HN}, the results in this paper are rigorously established, with no transversality issues. Indeed, due to special properties of the periodic orbits involved, we have to deal only with somewhere injective holomorphic curves; see Remark \ref{rmk:transversality1}.

\subsection*{Acknowledgements} 
We thank IMPA and IST for the warm hospitality during the preparation of this work. We are grateful to Gabriele Benedetti, Viktor Ginzburg and Umberto Hryniewicz for useful conversations regarding this paper. We are also grateful to the anonymous referees for a careful reading of the manuscript and many useful suggestions. Part of these results were presented by the first author at the Workshop on Conservative Dynamics and Symplectic Geometry, IMPA, Rio de Janeiro, Brazil, September 2--6, 2013 and by the second author at the Mathematical Congress of Americas, Guanajuato, Mexico, August 5-9, 2013 and the VIII Workshop on Symplectic Geometry, Contact Geometry, and Interactions, IST, Lisbon, Portugal, January 30 - February 1, 2014. They thank the organizers for the opportunity to participate in such wonderful events.

\section{Statement of the results}
\label{sec:results}

\subsection{Dynamical convexity}

In order to understand the results of this paper, it is important to define a notion of dynamical convexity for general contact manifolds in terms of the underlying contact homology. We will use contact homology in the definition although this is yet to be fully put on a rigorous basis \cite{HWZ10, HWZ11, HN}. However, our results do not have transversality issues and in many situations can be stated in terms of other invariants, like equivariant symplectic homology. (As proven in \cite{BO2}, equivariant symplectic homology is isomorphic to linearized contact homology whenever the late is well defined.)

Let us recall the basic definitions. Cylindrical contact homology $HC_*(M,\alpha)$ is an invariant of the contact structure 
$\xi:=\ker \alpha$, 
whose chain complex is generated by good periodic orbits of $R_\alpha$. This is defined whenever the periodic orbits of $R_\alpha$ satisfy 
suitable assumptions; see Section \ref{sec:CH} for details. Throughout this paper, 
\emph{we will assume that the first Chern class $c_1 (\xi)$ vanishes on $H_2 (M,\Z)$,
take the grading given by the Conley-Zehnder index and use rational coefficients}.

The Conley-Zehnder index is defined for non-degenerate periodic orbits and there are different extensions for degenerate closed orbits in the literature; see, for instance, \cite{HWZ, Lon02, RS}. We will consider here the Robbin-Salamon index and lower and upper semicontinuous extensions of the Conley-Zehnder index. We will denote them by $\rs$, $\czl$ and $\czu$ respectively (see Section \ref{sec:CZ} for precise definitions; the lower semicontinuous index coincides with those defined in \cite{HWZ,Lon99,Lon00,Lon02,LZ}).

In their celebrated paper \cite{HWZ}, Hofer, Wysocki and Zehnder introduced the notion of \emph{dynamical convexity} for tight contact forms on the three-sphere: a contact form on $S^{3}$ is called dynamically convex if every periodic orbit $\gamma$ satisfies $\czl(\gamma) \geq 3$. They proved that a convex hypersurface in $\R^{4}$ (that is, a smooth hypersurface bounding a compact convex set with non-empty interior) with the induced contact form is dynamically convex. Notice that, in contrast to convexity, dynamical convexity \emph{is invariant by contactomorphisms}. Their main result establishes that the Reeb flow of a dynamically convex contact form has global sections given by pages of an open book decomposition. The existence of such sections has several dynamical consequences like, for instance, the dichotomy between two or infinitely many periodic orbits.

In this work, we propose a generalization of the notion of dynamical convexity for general contact manifolds and show how it can be used to obtain the existence of elliptic closed orbits. To motivate our definition, let us consider initially the case of higher dimensional spheres. As observed in \cite{HWZ}, the proof that a convex hypersurface in $\R^4$ is dynamically convex works in any dimension and shows that a periodic orbit $\gamma$ on a convex hypersurface in $\R^{2n+2}$ satisfies $\czl(\gamma) \geq n+2$ (see Section \ref{sec:index_estimates} for a proof). It turns out that the term $n+2$ has an important meaning: \emph{it corresponds to the lowest degree with non-trivial contact homology for the standard contact structure $\xi_{\rm st}$ on $S^{2n+1}$}. Indeed, an easy computation shows that
\[
HC_\ast (S^{2n+1}, \xi_{\rm st}) \cong
\begin{cases}
\Q & \text{if $\ast = n+2+2k$ and $k \in \{0,1,2,3,\dots\}$} \\
0 & \text{otherwise.}
\end{cases} 
\]
In Theorem \ref{thm:convex} we show that this is a particular instance of a much more general phenomenon for toric contact manifolds. Based on this variational property, we introduce the following definition of dynamical convexity. In order to deal with contact manifolds that are not simply connected, we will make use of the fact that cylindrical contact homology has a filtration given by the free homotopy class of the periodic orbits: $HC_*(M,\xi) = \oplus_{a \in [S^1,M]} HC^a_*(M,\xi)$.

\begin{definition*}
Let $k_- = \inf\{k \in \Z; HC_k(\xi) \neq 0\}$ and $k_+ = \sup\{k \in \Z; HC_k(\xi) \neq 0\}$. A contact form $\alpha$ is positively (resp. negatively) dynamically convex if $k_-$ is an integer and $\czl(\gamma) \geq k_-$ (resp. $k_+$ is an integer and $\czu(\gamma) \leq k_+$) for every periodic orbit $\gamma$ of $R_\alpha$. Similarly, let $a$ be a free homotopy class in $M$, $k^a_- = \inf\{k \in \Z; HC^a_k(\xi) \neq 0\}$ and $k^a_+ = \sup\{k \in \Z; HC^a_k(\xi) \neq 0\}$. A contact form $\alpha$ is positively (resp. negatively) $a$-dynamically convex if $k^a_-$ is an integer and $\czl(\gamma) \geq k^a_-$ (resp. $k^a_+$ is an integer and $\czu(\gamma) \leq k^a_+$) for every periodic orbit $\gamma$ of $R_\alpha$ with free homotopy class $a$.
\end{definition*}

\begin{remark}
The motivation to consider \emph{negative} dynamical convexity comes from applications to magnetic flows; see Section \ref{sec:applications}.
\end{remark}

\subsection{Statement of the main results}
\label{sec:main_results}

Our first main result establishes that, under some hypotheses, dynamically convex contact forms have elliptic orbits with extremal Conley-Zehnder indexes. In order to state it, we need some preliminary definitions. We say that a Morse function is \emph{even} if every critical point has even Morse index. A contact manifold $(M,\xi)$ is called \emph{Boothby-Wang} if it supports a contact form $\beta$ whose Reeb flow generates a free circle action. If $M$ is Boothby-Wang then the quotient $N:=M/S^1$ is endowed with a symplectic form $\om$ given by the curvature of $\beta$ and we say that $M$ is a \emph{prequantization} of $(N,\om)$. Given a finite subgroup $G \subset S^1$ and a contact form $\alpha$ supported by a Boothby-Wang contact structure $\xi$, we say that $\alpha$ is \emph{$G$-invariant} if $(\phi_\beta^{t_0})^*\alpha=\alpha$, where $\phi_\beta^t$ is the flow of $R_\beta$ and $t_0 \in S^1$ is the generator of $G$. A periodic orbit $\ga$ of a $G$-invariant contact form $\alpha$ is called \emph{$G$-symmetric} if $\phi_\beta^{t_0}(\ga(\R))=\ga(\R)$.

\vskip .2cm
\noindent
{\bf Theorem A.}
{\it Let $(M^{2n+1},\xi)$ be a Boothby-Wang contact manifold over a closed symplectic manifold $(N,\om)$ such that $c_1(TN) = \lambda\omega$ for some constant $\lambda \in \R$. Let $\vr$ be a simple orbit of $R_\beta$ and denote by $a$ the free homotopy class of $\vr$. Assume that one of the following two conditions holds:
\begin{itemize}
\item[(H1)] $N$ admits an even Morse function, $a=0$ and $\lambda>1$.
\item[(H2)] $\om|_{\pi_2(N)}=0$ and $\pi_1(N)$ is torsion free.
\end{itemize}
Let $G$ be a finite subgroup of $S^1$ and $\alpha$ a $G$-invariant contact form supporting $\xi$. Then $\alpha$ carries a $G$-symmetric closed orbit $\tilde\gamma$ with free homotopy class $a$. Moreover, if $G$ is non-trivial and every periodic orbit $\ga$ of $\alpha$ with free homotopy class $a$ satisfies $\czl(\ga)\geq \rs(\vr)-n$ in case (H1) or $\czu(\ga) \leq \rs(\vr)+n$ in case (H2) then $\tilde\ga$ is elliptic and its Conley-Zehnder index satisfies $\czl(\tilde\ga) = \rs(\vr)-n$ in case (H1) or $\czu(\tilde\ga) = \rs(\vr)+n$ in case (H2).}
\vskip .2cm

The statement above can be nicely rephrased using dynamical convexity in the following way. Under the assumptions of Theorem A, one can construct a non-degenerate contact form $\beta^\pr$ (given by a perturbation of $\beta$) supporting $\xi$ for which cylindrical contact homology is well defined (that is, its differential $\partial$ is well defined and satisfies $\partial^2=0$) and $k^a_- = \rs(\vr)-n$ in case (H1) and $k^a_+ = \rs(\vr)+n$ in case (H2) (see Section \ref{sec:computationCH}). Thus, if $G$ is non-trivial and $\alpha$ is positively $a$-dynamically convex in case (H1) or negatively $a$-dynamically convex in case (H2) then $\alpha$ carries an elliptic $G$-symmetric closed orbit. We did not state Theorem A in this way because, as was already noticed, contact homology is yet to be fully put on a rigorous basis.

\begin{remark}
\label{rmk:strictly ellipticA}
It can be shown that $\tilde\ga$ is actually strictly elliptic, that is, $\tilde\ga$ is elliptic and every eigenvalue of its linearized Poincar\'e map different from one is Krein-definite; see Remark \ref{rmk:strictly elliptic - proof}.
\end{remark}

\begin{remark}
\label{rmk:monotone}
The hypothesis that $c_1(TN) = \lambda\omega$ for some constant $\lambda \in \R$ is equivalent to the condition that $c_1(\xi)|_{H_2(M,\Z)}=0$; see Section \ref{sec:proof main theorem}. The hypothesis that $\pi_1(N)$ is torsion free is used only to ensure that, under the hypothesis (H2), the free homotopy class of the fiber is primitive; see Section \ref{sec:Proof Thm 1}. If $\pi_r(N)=0$ for every $r\geq 2$ then (H2) is automatically satisfied; see \cite{GGM}.
\end{remark}

\begin{remark}
\label{rmk:trivialization Thm A}
As usual, the indexes of $\ga$ and $\vr$ are computed using a trivialization of $\xi$ over a closed curve representing the free homotopy class $a$ (see Section \ref{sec:CZ}). It is easy to see that the assumptions on $\alpha$ do not depend on the choice of this trivialization: the difference between $\rs(\vr)$ and $\mu_{\text{CZ}}^\pm(\tilde\ga)$ does not depend on this, see \eqref{eq:change of trivialization}.
\end{remark}

\begin{remark}
Theorem A implies, in particular, that the Weinstein conjecture holds for contact manifolds satisfying the hypotheses (H1) or (H2) using techniques of contact homology without transversality issues. Notice that $M$ does not need to admit an aspherical symplectic filling. In fact, in the proof we do not need a filling of $M$ at all. This is important because in the proof we consider the contact homology of the quotient $M/G$ which, in general, does not admit an aspherical symplectic filling.
\end{remark}

\begin{remark}
\label{rmk:transversality1}
In (H1), the hypotheses on $\lambda$ can be replaced by the weaker assumption that $\lambda>0$ if we assume that the relevant moduli spaces of pseudo-holomorphic curves used to define and prove the invariance of the contact homology can be cut out transversally. (This condition is used only to assure that in case (H1) a positively $a$-dynamically convex contact form has the property that every contractible periodic orbit has non-negative reduced Conley-Zehnder index; see Remark \ref{rmk:transversality3}.) Moreover, the argument is easier in this case, as will be explained along the proof; see Remarks \ref{rmk:transversality2} and \ref{rmk:transversality3}. However, with the hypotheses of the main theorem the proof is rigorously established, with no transversality issues. A crucial ingredient for this is the fact that the orbit used in the neck-stretching argument developed in Section \ref{sec:neck-stretching} has minimal index and action in case (H1) and maximal index and action in case (H2). This ensures that we have to deal only with somewhere injective holomorphic curves.
\end{remark}

Now, notice that $S^{2n+1}$ with the standard contact structure is Boothby-Wang and its first Chern class vanishes. Moreover, the quotient is $\CP^n$ which clearly admits an even Morse function (as any toric symplectic manifold) and the Robbin-Salamon index of the (simple) orbits of the circle action is equal to $2n+2$ (for a trivialization given by a capping disk). The action of $\Z_2 \subset S^1$ is generated by the antipodal map. Consequently we obtain the following generalization of the result proved by Dell'Antonio, D'Onofrio and Ekeland \cite{DDE} mentioned in the introduction.

\begin{corollary}
\label{cor1}
If a contact form on $S^{2n+1}$ is dynamically convex (that is, every periodic orbit $\ga$ satisfies $\czl(\ga)\geq n+2$) and invariant by the antipodal map then its Reeb flow carries an elliptic $\Z_2$-symmetric closed orbit $\ga$. Moreover, $\czl(\ga)=n+2$.
\end{corollary}

\begin{remark}
Clearly the $\Z_2$-invariance of the contact form in the previous corollary can be generalized to $\Z_m$-invariance for any $m \geq 2$. This means that if a starshaped hypersurface $S$ in $\R^{2n+2}$ is dynamically convex and invariant by the map $z \mapsto e^{2\pi i/m}z$ then it carries an elliptic $\Z_m$-symmetric closed characteristic. If $S$ is convex then this result follows from a generalization of \cite{DDE} due to Arnaud \cite{Arn}. Moreover, Long and Zhu \cite{LZ} proved the existence of an elliptic closed characteristic on a convex hypersurface $S$ in $\R^{2n+2}$, without any hypothesis of invariance, under the restrictive assumption that its Reeb flow carries only finitely many simple periodic orbits. Under more restrictive assumptions, several results have been obtained concerning the multiplicity of the elliptic periodic orbits; see, for instance, \cite{Lon00,Lon02,Wan13,Wan14} and references therein.
\end{remark}

\begin{remark}
It can be shown that if $S$ is a convex hypersurface in $\R^{2n+2}$ then the elliptic periodic orbit given by the previous corollary is simple.
\end{remark}

Our second main result shows that the hypothesis of non-triviality of $G$ used to ensure that the periodic orbit given by Theorem A is elliptic can be relaxed once we impose certain bounds on the Conley-Zehnder indexes of some high iterate of the periodic orbits.

\vskip .2cm
\noindent
{\bf Theorem B.}
{\it Let $(M^{2n+1},\xi)$ be a contact manifold satisfying the hypotheses of Theorem A and $\alpha$ a contact form supporting $\xi$. Suppose that there exists an integer $k>1$ such that every periodic orbit $\ga$ of $\alpha$ with free homotopy class $a$ satisfies $\czl(\ga)\geq \rs(\vr)-n$ and $\czl(\ga^k)\geq k\rs(\vr)-n$ in case (H1) or $\czu(\ga) \leq \rs(\vr)+n$ and $\czu(\ga^k) \leq k\rs(\vr)+n$ in case (H2). Then $\alpha$ carries a elliptic periodic orbit $\tilde\ga$. Moreover, $\czl(\tilde\gamma)=\rs(\vr)-n$ in case (H1) and $\czu(\tilde\gamma)=\rs(\vr)+n$ in case (H2).}
\vskip .2cm

\begin{remark}
\label{rmk:strictly ellipticB}
As in Remark \ref{rmk:strictly ellipticA}, we actually have that $\tilde\ga$ is strictly elliptic.
\end{remark}

The hypothesis in Theorem B has the following subtle homological meaning. In dynamical convexity, the inequalities $\cz(\ga) \geq k_-$ and $\cz(\ga) \leq k_+$ are obvious necessary conditions for a non-degenerate closed orbit $\ga$ being homologically essential: clearly $\ga$ cannot contribute to the contact homology if one of these inequalities is violated due to the very definition of $k_-$ and $k_+$. On the other hand, as explained in Section \ref{sec:computationCH}, under the assumptions of Theorem A, the contact homology of a prequantization $M$ of a closed symplectic manifold $(N,\om)$ is given by copies of the singular homology of $N$ with a shift in the degree. More precisely, we have the isomorphism
\[
HC_{*}(M) \cong \oplus_{k \in \N} H_{*+ n - k\rs(\vr)}(N),
\]
see \eqref{eq:HC-1}. Thus, the inequalities $\cz(\ga^k)\geq k\rs(\vr)-n$ and $\cz(\ga^k) \leq k\rs(\vr)+n$ are obvious necessary conditions for $\ga^k$ being homologically essential \emph{in the $k$-th copy of $H_*(N)$ in the contact homology}.

Now, let $M$ be an hypersurface in $\R^{2n+2}$ such that $M=H^{-1}(1/2)$, where $H: \R^{2n+2} \to \R$ is a convex Hamiltonian homogeneous of degree two. As defined in \cite{Eke,Eke86},  we say that $M$ is \emph{$(r,R)$-pinched}, with $0<r\leq R$, if
\[
\|v\|^2R^{-2} \leq \lg d^2H(x)v,v \rg \leq \|v\|^2r^{-2}
\]
for every $x \in M$ and $v \in \R^{2n+2}$, where $d^2H(x)$ is the Hessian of $H$ at $x$. Using results due to Croke and Weinstein \cite{CW} and lower estimates for the Conley-Zehnder index, we can show that if $M$ is $(r,R)$-pinched with $\frac{R}{r} < \sqrt{\frac{k}{k-1}}$ for some real number $k>1$ then every periodic orbit $\ga$ of $M$ satisfies $\czl(\ga^{\lfloor k \rfloor}) \geq \lfloor k \rfloor(2n+2)-n$, where $\lfloor k \rfloor = \max\{j \in \Z \mid j \leq k\}$ (note that when $k\to 1$ we conclude that every convex hypersurface is dynamically convex); see Section \ref{sec:index_estimates}. Thus, we derive the following result mentioned in the introduction due to Ekeland \cite{Eke86,Eke}.

\begin{corollary}
\label{cor2}
Let $M$ be a $(r,R)$-pinched convex hypersurface in $\R^{2n+2}$ with $R/r < \sqrt{2}$. Then $M$ carries an elliptic closed Reeb orbit $\ga$ such that $\czl(\ga)=n+2$.
\end{corollary}

\begin{remark}
It can be shown that the elliptic orbit given by the previous corollary is simple.
\end{remark}

\subsection{Applications}
\label{sec:applications}

The hypothesis of non-triviality of $G$ in Theorem A is probably just technical, but it is crucial in our argument. However, in many applications this condition is automatically achieved due to topological reasons. Let us start illustrating this with an application to geodesic flows on $S^2$. Let $F$ be a Finsler metric on $S^2$ with reversibility $r := \max\{F(-v) \mid v \in TM,\ F(v)=1\}$. It is well known that the geodesic flow of any Finsler metric on $S^2$ lifts to a Hamiltonian flow on a suitable star-shaped hypersurface in $\R^4$ via a double cover, see \cite{CO} (the sphere bundle $SS^2$ is contactomorphic to $\RP^3$ with the contact structure induced by $S^3$ tight). It is proved in \cite{HP} that, if the flag curvature $K$ of $F$ satisfies the pinching condition $(r/(r+1))^2 < K \leq 1$, then the lift of the geodesic flow to $S^3$ is dynamically convex. We refer to \cite{Rad04} for the precise definitions; we just want to mention here that if $F$ is a Riemannian metric then $r=1$, $K$ is the Gaussian curvature and the pinching condition reads as the classical assumption $1/4 < K \leq 1$. Thus, we have the following theorem. It was proved before by Rademacher \cite{Rad07} using different methods and assuming that the pinching condition is strict.

\begin{theorem}
\label{thm:elliptic geodesic S^2}
Let $F$ be a Finsler metric on $S^2$ with reversibility $r$ and flag curvature $K$ satisfying $(r/(r+1))^2 \leq K \leq 1$. Then $F$ carries an elliptic closed geodesic $\gamma$ whose Morse index satisfies $\morse(\ga)\leq 1$. Moreover, if $(r/(r+1))^2 < K \leq 1$ then $\ga$ is prime and $\morse(\ga)=1$.
\end{theorem}

The case where the inequality $(r/(r+1))^2 \leq K$ is not strict can be dealt with using a perturbation argument; see Section \ref{sec:proof applications} for details.

Applications to geodesic flows in higher dimensions are given by the next two results. Consider a Finsler metric $F$ on $S^{n+1}$ with $n\geq 1$, reversibility $r$ and flag curvature $K$. An easy computation of the contact homology of the unit sphere bundle $SS^{n+1}$ shows that $k_-=\rs(\vr)-n=n$, where $\vr$ is a prime closed geodesic of the round metric and $\rs(\vr)$ is computed using a suitable trivialization. Results due to Rademacher \cite{Rad04} establish that if $F$ satisfies the pinching condition $(r/(r+1))^2 < K \leq 1$ then every closed geodesic $\ga$ of $F$ satisfies $\czl(\ga)\geq n$. In other words, this pinching condition implies that the contact form defining the geodesic flow is positively dynamically convex. Moreover, $SS^{n+1}$ is a prequantization of the Grassmannian of oriented two-planes $G^+_2(\R^{n+2})$ which admits an even Morse function and is monotone with constant of monotonicity $\lambda>1$; see Section \ref{sec:proof applications} for details. Therefore, we get the following result which generalizes \cite[Theorem A(ii)]{BTZ}. It was also proved by Rademacher \cite{Rad07} using different methods and assuming that the pinching condition is strict.  In the statement, $K_-$ and $K_+$ are continuous functions defined on the unit sphere bundle of $\RP^{n+1}$ defined as $K_-(x,v) = \min_{\sigma \subset T_x\RP^{n+1}} K(\sigma,v)$ and $K_+(x,v) = \max_{\sigma \subset T_x\RP^{n+1}} K(\sigma,v)$, where $\sigma$ runs over all the planes in $T_x\RP^{n+1}$ and $K(\sigma,v)$ is the flag curvature; see Section \ref{sec:proof elliptic geodesic S^2} for details. Notice that $K_-=K_+$ if $n=1$.

\begin{theorem}
\label{thm:elliptic geodesic RP^n}
Let $F$ be a Finsler metric on $\RP^{n+1}$ with $n\geq 1$, reversibility $r$ and flag curvature $K$ satisfying $(r/(r+1))^2 \leq K \leq 1$. If $K$ is not strictly bigger than $(r/(r+1))^2$, suppose that $K_-(x,v)/K_+(x,v) > (r/(r+1))^2$ for every $(x,v)$ in the unit sphere bundle of $\RP^{n+1}$. Then $F$ carries an elliptic closed geodesic $\gamma$. Moreover, $\ga$ is prime and $\morse(\gamma) = 0$.
\end{theorem}

\begin{remark}
The hypothesis that $K_-(x,v)/K_+(x,v) > (r/(r+1))^2$ allows us to make a $C^\infty$-perturbation of $F$ such that the flag curvature $K'$ of the new metric $F'$ satisfies $(r/(r+1))^2 < K' \leq 1$; see Proposition \ref{prop:perturbation}. This, together with a uniform bound for the period of the closed orbit given by Theorem A, enables us to deal with the case that the pinching condition is not strict. An analogous assumption appears in the next theorem by the same reason.
\end{remark}

Under the stronger pinching condition $\frac{9}{4}(r/(r+1))^2 < K \leq 1$, Rademacher \cite{Rad04} proved that every closed geodesic $\ga$ of $F$ satisfies $\czl(\ga^2)\geq 3n = 2\rs(\vr)-n$. Hence, we have the following generalization of \cite[Theorem A(i)]{BTZ} proved before by Rademacher \cite{Rad07} using different methods and assuming that the pinching condition is strict. Similarly to the previous theorem, in the statement $K_-$ and $K_+$ are defined as $K_-(x,v) = \min_{\sigma \subset T_xS^{n+1}} K(\sigma,v)$ and $K_+(x,v) = \max_{\sigma \subset T_xS^{n+1}} K(\sigma,v)$, where $\sigma$ runs over all the planes in $T_xS^{n+1}$.

\begin{theorem}
\label{thm:elliptic geodesic S^n}
Let $F$ be a Finsler metric on $S^{n+1}$ with $n\geq 1$, reversibility $r$ and flag curvature $K$ satisfying $\frac{9}{4}(r/(r+1))^2 \leq K \leq 1$. If $K$ is not strictly bigger than $\frac{9}{4}(r/(r+1))^2$, suppose that $K_-(x,v)/K_+(x,v) > \frac{9}{4}(r/(r+1))^2$ for every $(x,v)$ in the unit sphere bundle of $S^{n+1}$. Then $F$ carries an elliptic closed geodesic $\gamma$ such that $\morse(\gamma) \leq n$. Moreover, if $\frac{9}{4}(r/(r+1))^2 < K \leq 1$ then $\ga$ is prime and $\morse(\gamma) = n$.
\end{theorem}

Our next result applies to magnetic flows and motivates our definition of \emph{negative} dynamical convexity. Let $N$ be a Riemannian manifold and $\Omega$ a closed 2-form on $N$. Consider on $T^*N$ the twisted symplectic form $\omega = \omega_0 + \pi^*\Omega$, where $\omega_0$ is the canonical symplectic form and $\pi: T^*N \to N$ is the projection. Let $H: T^*N \to \R$ be the Hamiltonian given by the kinetic energy. The Hamiltonian flow of $H$ with respect to $\omega$ is called the magnetic flow generated by the Riemannian metric and the magnetic field $\Omega$.

It is well known that if $N$ is a closed orientable surface with genus $\g\neq 1$ and $\Omega$ is a symplectic form then there exists $\ep>0$ such that $H^{-1}(k)$ is of contact type for every $k<\ep$, see \cite{Ben,GGM} (an easy argument using the Gysin sequence shows that it cannot be of contact type if $N=T^2$). Benedetti proved in \cite{Ben} that if $N=S^2$ then every periodic orbit $\ga$ of the lifted Reeb flow on $S^3$ satisfies $\czl(\ga)\geq 3$ if $\ep$ is chosen sufficiently small.

Moreover, it follows from his estimates that, for surfaces of genus $\g>1$, a homologically trivial periodic orbit $\gamma$ on $H^{-1}(k)$ satisfies $\czu(\gamma) \leq 2\chi(N)+1$, where $\chi(N)$ stands for the Euler characteristic of $N$ and the index is computed using a suitable global trivialization of the contact structure; see Section \ref{sec:proof magnetic} for details. Now, fix $k< \ep$ and consider $S := H^{-1}(k)$ with the contact form $\alpha$. One can prove that $S$ is Boothby-Wang and there is a $|\chi(N)|$-covering $\tau: \wtl S \to S$ such that $\wtl\beta := \tau^*\beta$ (where $\beta$ is the contact form whose Reeb flow is periodic) generates a free circle action and the deck transformations are given by the induced action of $\Z_{|\chi(N)|} \subset S^1$. Moreover, the orbits of this circle action in $\wtl S$ are homologous to zero. Let $\wtl\vr$ be a (simple) orbit of the Reeb flow of  $\wtl\beta$. An easy computation shows that $\rs(\wtl\vr)=2\chi(N)$. Consequently, the lift of $\alpha$ to $\wtl S$ furnishes a $\Z_{|\chi(N)|}$-invariant contact form satisfying the hypothesis (H2) in Theorem A.

In other words, the magnetic flow on a sufficiently small energy level has a lift to a $|\chi(N)|$-covering that is \emph{positively} dynamically convex if $\g=0$ and \emph{$a$-negatively} dynamically convex if $\g > 1$, where $a$ is the free homotopy class of $\wtl\vr$. Therefore, we obtain the following result.

\begin{theorem}
\label{thm:magnetic}
Let $(N,g)$ be a closed orientable Riemannian surface of genus $\g \neq 1$ and $\Omega$ a symplectic magnetic field on $N$. There exists $\ep>0$ such that the magnetic flow has an elliptic closed orbit $\gamma$ on $H^{-1}(k)$ for every $k<\ep$. Moreover, $\gamma$ is prime, freely homotopic to a fiber of $SN$ and satisfies $\czl(\gamma)=1$ if $\g=0$ and $\czu(\gamma)=-1$ if $\g>1$, where the indexes are computed using a suitable global trivialization of $\xi$ over $H^{-1}(k)$.
\end{theorem}

\begin{remark}
There are in \cite{Ben} explicit lower bounds for $\ep$ in terms of the Riemannian metric and the magnetic field.
\end{remark}

\subsection{Toric contact manifolds}

As pointed out previously, it is proved in \cite{HWZ} that a convex hypersurface in $\R^{2n}$ is dynamically convex. Our next result shows that this is a particular case of a more general phenomenon for toric contact manifolds. Details are given in Sections \ref{sec:toric} and \ref{sec:proof convex}. Toric contact manifolds can be defined as contact manifolds of dimension $2n+1$ equipped with an effective Hamiltonian action of a torus of dimension $n+1$. A \emph{good} toric contact manifold has the property that its symplectization can be obtained by symplectic reduction of $\C^d\setminus\{0\}$, where $d$ is the number of facets of the corresponding momentum cone, by the action of a subtorus $K \subset \T^d$, with the action of $\T^d$ given by the standard linear one. The sphere $S^{2n+1}$ is an example of a good toric contact manifold and its symplectization is obtained from $\C^{n+1}$ with $K$ being trivial (that is, there is no reduction at all; the symplectization of $S^{2n+1}$ can be identified with $\C^{n+1}\setminus\{0\}$).

Consequently, given a contact form $\alpha$ on a good toric contact manifold $M$ we can always find a Hamiltonian $H_\alpha: \C^d \to \R$ invariant by $K$ such that the reduced Hamiltonian flow of $H$ is the Reeb flow of $\alpha$. Notice that $H_\alpha$ is not unique. We say that a contact form $\alpha$ on $M$ is \emph{convex} if such $H_\alpha$ can be chosen convex on $Z:=F^{-1}(0)$, where $F$ is the momentum map associated to the action of the subtorus $K$. By a technical reason, we will also suppose that the linearized Hamiltonian flow of $H_\alpha$ along every periodic orbit in $Z$ satisfies $d\Phi^t_{H_\alpha}(\nabla F_\kappa)=\nabla F_\kappa$ for all $t \in \R$ and every component  $F_\kappa$ of the momentum map, where the gradient is taken with respect to the Euclidean metric; see Section \ref{sec:proof convex}. Clearly, a contact form on the standard contact sphere $S^{2n+1}$ is convex in this sense if and only if the corresponding hypersurface in $\C^{n+1}$ is convex. Moreover, every toric contact form on any good toric contact manifold is convex; see Section \ref{sec:toric}.

It is shown in \cite{AM} that good toric contact manifolds admit suitable non-degenerate contact forms whose cylindrical contact homology is well defined (in fact, its differential vanishes identically). Moreover, it can be computed in a purely combinatorial way in terms of the associated momentum cone. Using this computation in our definition of dynamical convexity, we achieve the following result.

\begin{theorem}
\label{thm:convex}
A convex contact form $\alpha$ on a good toric simply connected contact manifold $(M,\xi)$ is positively dynamically convex. More precisely, for every periodic orbit $\ga$ of $\alpha$ there exists a non-degenerate contact form $\tilde\alpha$ on $(M,\xi)$ with a well defined cylindrical contact homology such that $\czl(\ga) \geq \inf\{k \in \Z\mid HC_k(\tilde\alpha) \neq 0\}$.
\end{theorem}

\begin{remark}
It is proved in Proposition \ref{prop:finite} that $(M,\xi)$ always supports a non-degenerate contact form $\tilde\alpha$ with a well defined cylindrical contact homology such that $k_-:=\inf\{k \in \Z\mid HC_k(\tilde\alpha) \neq 0\}$ is an integer. Recall that well defined cylindrical contact homology means that its differential $\partial$ is well defined and satisfies $\partial^2=0$.
\end{remark}

Notice that dynamical convexity for general toric contact manifolds is a property sensitive to the choice of the contact structure. For instance, it is shown in \cite{AM} a family of inequivalent toric contact structures $\xi_k$ on $S^2 \times S^3$, 
$k \in \{0,1,2,\ldots\}$, for which $k_-=2$ for $\xi_0$ and $k_-=0$ for $\xi_k$ with $k>0$.

Now, let $M$ be a prequantization of a closed toric symplectic manifold $(N,\omega)$ such that $[\omega] \in H^2(N,\Z)$. It is easy to see that $M$ is a good toric contact manifold. Moreover, $N$ admits an even Morse function. We have then the following generalization of the aforementioned result from \cite{DDE}.

\begin{corollary}
Let $M$ be a simply connected prequantization of a closed toric symplectic manifold $(N,\omega)$ such that $[\omega] \in H^2(N,\Z)$ and $G$ a non-trivial finite subgroup of $S^1$. Suppose that $N$ is monotone with constant of monotonicity $\lambda$ bigger than one. Then every $G$-invariant convex contact form on $M$ has an elliptic $G$-symmetric closed orbit.
\end{corollary}

\begin{remark}
As mentioned in Remark \ref{rmk:transversality1}, the hypotheses on $N$ can be relaxed if we assume suitable transversality assumptions for the relevant moduli spaces of pseudo-holomorphic curves.
\end{remark}

\vskip .2cm
\noindent {\bf Organization of the paper.} In Section \ref{sec:CZ} we introduce the basic material and give some estimates for the indexes of periodic orbits used throughout this work. In Section \ref{sec:CH} we review the background concerning almost complex structures and holomorphic curves necessary for the proofs of Theorems A and B, established in Section \ref{sec:proof main theorem}. The applications are proved in Section \ref{sec:proof applications}. Finally, Sections \ref{sec:toric} and \ref{sec:proof convex} are devoted to introduce the basic material about toric contact manifolds and prove Theorem \ref{thm:convex} respectively.

\section{The Conley-Zehnder index}
\label{sec:CZ}

\subsection{The Conley-Zehnder index for paths of symplectic matrices}

Let $\P(2n)$ be the set of paths of symplectic matrices $\Gamma:[0,1]\to Sp(2n)$ such that $\Gamma(0)=\text{Id}$, endowed with the $C^1$-topology. Consider the subset $\P^*(2n) \subset \P(2n)$ given by the non-degenerate paths, that is, paths $\Ga \in \P(2n)$ satisfying the additional property that $\Gamma(1)$ does not have $1$ as an eigenvalue. Following \cite{SZ}, one can associate to $\Gamma \in \P^*(2n)$ its Conley-Zehnder index $\cz(\Gamma) \in \Z$ uniquely characterized by the following properties:
\begin{itemize}
 \item {\bf Homotopy:} If $\Gamma_s$ is a homotopy of arcs in $\P^*(2n)$ then $\cz(\Gamma_s)$ is constant.
 \item {\bf Loop:} If $\phi: \R/\Z \rightarrow Sp(2n)$ is a loop at the identity and $\Gamma \in \P^*(2n)$ then $\cz(\phi\Gamma) = \cz(\Gamma) + 2\maslov(\phi)$, where $\maslov(\phi)$ is the Maslov index of $\phi$.
  \item {\bf Signature:} If $A \in \R^{2n\times 2n}$ is a symmetric non-degenerate matrix with all eigenvalues of absolute value less than $2\pi$ and $\Gamma(t)=\exp(J_0At)$, where $J_0$ is the canonical complex structure in $\R^{2n}$, then $\cz(\Ga)=\frac{1}{2}\text{Sign}(A)$.
\end{itemize}

There are different extensions of $\cz: \P^*(2n) \to \Z$ to degenerate paths in the literature. In this work we will use the Robbin-Salamon index $\rs: \P(2n) \to \frac{1}{2}\Z$ defined in \cite{RS} and lower and upper semicontinuous extensions denoted by $\czl$ and $\czu$ respectively. The later are defined in the following way: given $\Ga \in \P(2n)$ we set
\begin{equation}
\label{lower_CZ}
\czl(\Ga) = \sup_{U} \inf\{\cz(\Ga^\pr) \mid \Ga^\pr \in U \cap \P^*(2n)\} 
\end{equation}
and
\begin{equation}
\label{upper_CZ}
\czu(\Ga) = \inf_{U} \sup\{\cz(\Ga^\pr) \mid \Ga^\pr \in U \cap \P^*(2n)\}, 
\end{equation}
where $U$ runs over the set of neighborhoods of $\Ga$. The index $\czl$ coincides with \cite[Definition 6.1.10]{Lon02}. We have the relation
\begin{equation}
\label{eq:czl x czu}
\czl(\Ga^{-1}) = -\czu(\Ga).
\end{equation}
for every $\Ga \in \P(2n)$. This follows immediately from the fact that if $\Ga \in \P^*(2n)$ then $\cz(\Ga^{-1})=-\cz(\Ga)$.

\subsection{An analytical definition of $\czl$}

Following \cite{HWZ}, we will give an alternative analytical definition of $\czl$ which will be useful for us later. Denote by $\S$ the set of continuous paths of symmetric matrices $A(t)$ in $\R^{2n \times 2n}$, $0\leq t\leq 1$, equipped with the metric $d(A,B)=\int_0^1\|A(t)-B(t)\|\,dt$. Associated to a path $A \in \S$ we have the symplectic path $\Ga \in \P(2n)$ given by the solution of the initial value problem
\begin{equation}
\label{eq:symm_path}
\dot\Ga(t) = J_0A(t)\Ga(t),\ \Ga(0)=\text{Id},\ 0 \leq t \leq 1.
\end{equation}
Conversely, every symplectic path $\Ga \in \P(2n)$ is associated to a unique path of symmetric matrices $A$ satisfying the previous equation. Consider the self-adjoint operator $L_A: H^1(S^1,\R^{2n}) \to L^2([0,1],\R^{2n})$ given by
\[
L_A(v)=-J_0\dot v - A(t)v.
\]
This operator is a compact perturbation of the self-adjoint operator $-J_0\dot{v}$ whose spectrum is $2\pi\Z$. Thus, the spectrum of $L_A$, denoted by $\sigma(A)$, is given by real eigenvalues having multiplicities at most $2n$ and is unbounded from above and below. The spectrum bundle $\B \to \S$ is defined as
\[
\B = \bigcup_{A\in\S} (\{A\} \times \sigma(A)).
\]
The spectrum bundle has a unique bi-infinite sequence $(\lambda_k)_{k\in\Z}$ of continuous sections characterized by the following properties:
\begin{itemize}
\item $A \mapsto \lambda_k(A)$ is continuous.
\item $\lambda_k(A) \leq \lambda_{k+1}(A)$ for all $k \in \Z$, $A \in \S$, and $\sigma(A)=(\lambda_k(A))_{k\in\Z}$.
\item For $\tau \in \sigma(A)$ the number of $k$'s satisfying $\lambda_k(A)=\tau$ is the multiplicity of $\tau$.
\item The sequence is normalized at $A=0$ when $\lambda_j(0)=0$ for all $j$ such that $-n \leq j \leq n$.
\end{itemize}
The maps $\lambda_k$ have the following important monotonicity property. If $B \in \S$ satisfies $B(t)\geq 0$ then
\begin{equation}
\label{eq:index_monotonicity1}
\lambda_k(A+B) \leq \lambda_k(A)
\end{equation}
and if additionally we have $B(t_0) \geq \ep\text{Id}$ for some $t_0 \in [0,1]$ and $\ep>0$ then
\begin{equation}
\label{eq:index_monotonicity2}
\lambda_k(A+B) < \lambda_k(A)
\end{equation}
for every $A \in \S$ and $k \in \Z$. Consider the map $\hwz: \P(2n) \to \Z$ defined as
\[
\hwz(\Ga) = \max\{k \in \Z \mid \lambda_k(A)<0\}.
\]
It was proved in \cite{HWZ} that if $\Ga$ is non-degenerate then $\hwz(\Ga)=\cz(\Ga)$. We claim that for any $\Ga \in \P(2n)$ we have the identity
\begin{equation}
\label{eq:hwz=czl}
\hwz(\Ga) = \czl(\Ga).
\end{equation}
Indeed, given a non-degenerate perturbation $\Ga^\pr$ of $\Ga$, it follows from the continuity of the map $A \mapsto \lambda_k(A)$ that $\cz(\Ga^\pr) = \hwz(\Ga^\pr) \geq \hwz(\Ga)$ which implies that $\hwz(\Ga) \leq \czl(\Ga)$. On the other hand, let $A$ be the path of symmetric matrices associated to $\Ga$ via \eqref{eq:symm_path} and take $\ep>0$ arbitrarily small such that the path $\Ga^\pr$ generated by $A^\pr := A - \ep\text{Id}$ is non-degenerate. By \eqref{eq:index_monotonicity1} and the properties of $\lambda_k$ we have that
\[
\cz(\Ga^\pr) = \hwz(\Ga^\pr) = \hwz(\Ga),
\]
which implies that $\hwz(\Ga) \geq \czl(\Ga)$.

\subsection{The Conley-Zehnder index of periodic Reeb orbits}
\label{sec:index_orbits}

Let $a$ be a free homotopy class of a contact manifold $M^{2n+1}$. Fix a closed curve $\vr$ representing the free homotopy class $a$ and let $\Phi_t: \xi(\vr(t)) \to \R^{2n}$ be a symplectic trivialization of $\xi$ over $\vr$. Given a periodic orbit $\ga$ of $R_\alpha$ with free homotopy class $a$, consider a homotopy between $\vr$ and $\ga$. The trivialization $\Phi$ can be extended over the whole homotopy, defining a trivialization of $\xi$ over $\ga$. The homotopy class of this trivialization does not depend on the extension of $\Phi$. Moreover, if $c_1(\xi)|_{H_2(M,\Z)}=0$ then this homotopy class does not depend on the choice of the homotopy as well. Therefore, the index of the path
\begin{equation}
\Gamma(t) = \Phi(\gamma(t)) \circ d\phi_\alpha^t(\gamma(0))|_\xi \circ \Phi^{-1}(\gamma(0))
\end{equation}
is well defined, where $\phi_\alpha$ is the Reeb flow of $\alpha$. In this way, we have the indexes $\rs(\ga;\Phi)$, $\czl(\ga;\Phi)$ and $\czu(\ga;\Phi)$ which coincide with $\cz(\ga;\Phi)$ if $\ga$ is non-degenerate, that is, if its linearized Poincar\'e map does not have one as eigenvalue. The mean index of $\ga$ is defined as
\[
\Delta(\ga;\Phi) = \lim_{k\to\infty}\frac{1}{k} \czl(\ga^k;\Phi^k).
\]
It turns out that this limit exists. Moreover, the mean index is continuous with respect to the $C^2$-topology in the following sense: if $\alpha_j$ is a sequence of contact forms converging to $\alpha$ in the $C^2$-topology and $\ga_j$ is a sequence of periodic orbits of $\alpha_j$ converging to $\ga$ then $\Delta(\ga_j) \xrightarrow{j\to\infty} \Delta(\ga)$ \cite{SZ}. It is well known that the mean index satisfies the inequality
\[
|\cz(\tilde\ga;\Phi) - \Delta(\ga;\Phi)| \leq n
\]
for every closed orbit $\ga$ and non-degenerate perturbation $\tilde\ga$ of $\ga$. By the definition of $\mu_{\text{CZ}}^\pm$ this implies that
\begin{equation}
\label{eq:mean index}
|\mu_{\text{CZ}}^\pm(\ga;\Phi) - \Delta(\ga;\Phi)| \leq n
\end{equation}
for every periodic orbit $\ga$.

The trivialization can be omitted in the notation if it is clear in the context. If $\ga$ is contractible, we say that a trivialization of $\xi$ over $\ga$ is induced by a \emph{capping disk} if $\vr$ is the constant path and $\Phi_t: \xi(\vr(t)) \to \R^{2n}$ does not depend on $t$.

If we choose another trivialization $\Psi_t: \xi(\vr(t)) \to \R^{2n}$ over $\vr$ then we have the relations
\begin{equation}
\label{eq:change of trivialization}
\mu_{\text{CZ},\text{RS}}^\pm(\ga;\Psi) = \mu_{\text{CZ},\text{RS}}^\pm(\ga;\Phi) + 2\maslov(\Psi_t \circ \Phi_t^{-1}).
\end{equation}
The \emph{reduced} Conley-Zehnder index of a non-degenerate periodic orbit $\ga$ is defined as
\[
\rcz(\ga) = \cz(\ga) + n - 2.
\]
\emph{If $\ga$ is contractible we always take a trivialization over $\ga$ induced by a capping disk in the definition of the {\bf reduced} Conley-Zehnder index}. The reason for this is that the reduced Conley-Zehnder index is used in the computation of the Fredholm index of holomorphic curves with finite energy in symplectic cobordisms, see Section \ref{sec:CH}.

\subsection{Some estimates for the Conley-Zehnder index}
\label{sec:index_estimates}

In this section we will provide some estimates for the Conley-Zehnder index used in the proof of Corollary \ref{cor2}. Consider a Hamiltonian $H: \R^{2n+2} \to \R$  homogeneous of degree two. In what follows, we will use the convention that the Hamiltonian vector field $X_H$ is given by $\om(X_H,\cdot) = dH$. Suppose that $M:=H^{-1}(1/2)$ is $(r,R)$-pinched, that is,
\begin{equation}
\label{eq:pinching}
\|v\|^2R^{-2} \leq \lg d^2H(x)v,v \rg \leq \|v\|^2r^{-2}
\end{equation}
for every $x \in M$ and $v \in \R^{2n+2}$. Let $\xi$ be the contact structure on $M$ and $\xi^\om$ its symplectic orthogonal with respect to the canonical symplectic form $\om$. Clearly both $\xi$ and $\xi^\om$ are invariant by the linearized Hamiltonian flow of $H$. Let $\ga$ be a periodic orbit in $M$ with period $T$ and consider a capping disk $\sigma: D^2 \to M$ such that $\sigma|_{\partial D^2}=\ga$. Denote by $\Phi^\xi: \sigma^*\xi \to D^2 \times \R^{2n}$ and $\Phi^{\xi^\om}: \sigma^*\xi^\om \to D^2 \times \R^{2}$ the unique (up to homotopy) trivializations of the pullbacks of $\xi$ and $\xi^\om$ by $\sigma$. Fix a symplectic basis $\{e,f\}$ of $\R^2$. Note that $\Phi^{\xi^\om}$ can be chosen such that $\Phi^{\xi^\om}(X_H)=e$ and $\Phi^{\xi^\om}(Y)=f$, where $Y(x)=x$ (note that $\{X_H(\sigma(x)),Y(\sigma(x))\}$ is a symplectic basis of $\sigma^*\xi^\om(x)$ for every $x \in D^2$). Indeed, let $A: D^2 \to \Sp(2)$ be the map that associates to $x \in D^2$ the unique symplectic map that sends $\Phi^{\xi^\om}(X_H(\sigma(x)))$ to $e$ and $\Phi^{\xi^\om}(Y(\sigma(x)))$ to $f$. Then $\bar A \circ \Phi^{\xi^\om}$ gives the desired trivialization, where $\bar A: D^2 \times \R^2 \to D^2 \times \R^2$ is given by $\bar A(x,v)=(x,A(x)v)$.

We have that $\Phi:=\Phi^\xi \oplus \Phi^{\xi^\om}$ gives a trivialization of $\sigma^*T\R^{2n+2}$. Let $\Ga: [0,T] \to \Sp(2n+2)$ be the symplectic path given by the linearized Hamiltonian flow of $H$ along $\ga$ using $\Phi$. By construction, we can write $\Ga=\Ga^{\xi} \oplus \Ga^{\xi^\om}$, where $\Ga^{\xi}$ and $\Ga^{\xi^\om}$ are given by the linearized Hamiltonian flow of $H$ restricted to $\xi$ and $\xi^\om$ respectively. By the construction of $\Phi^{\xi^\om}$, the last one is given by the identity ($H$ is homogenous of degree two and therefore its linearized Hamiltonian flow preserves both $X_H$ and $Y$) and consequently $\czl(\Ga^{\xi^\om})=-1$. Thus,
\begin{equation}
\label{eq:index_split}
\czl(\Ga) = \czl(\Ga^{\xi})+\czl(\Ga^{\xi^\om}) = \czl(\Ga^{\xi})-1,
\end{equation}
where the first equality follows from the additivity property of the Conley-Zehnder index.

Now, note that $\Phi$ is homotopic to the usual (global) trivialization $\Psi$ of $\R^{2n+2}$ because both are defined over the whole capping disk. Thus, we will keep denoting by $\Ga$, without fear of ambiguity, the symplectic path given by the linearized Hamiltonian flow of $H$ along $\ga$ using $\Psi$. Let $H_r: \R^{2n+2} \to \R$ and $H_R: \R^{2n+2} \to \R$ be the Hamiltonians given by
\[
H_r(x)=\frac{1}{2r^2}\|x\|^2\ \text{ and }\ H_R(x)=\frac{1}{2R^2}\|x\|^2.
\]
Given a (not necessarily closed) trajectory $\ga_R: [0,S] \to \R^{2n+2}$ of $H_R$ the index of the symplectic path $\Ga_R: [0,S] \to \Sp(2n+2)$ induced by the corresponding linearized flow in $\R^{2n+2}$ (using the global trivialization $\Psi$) is given by
\begin{equation}
\label{eq:ind H_R}
\czl(\Ga_R) =
\begin{cases}
(2n+2)\frac{S}{2\pi R^2}-n-1 \text{ if } \frac{S}{2\pi R^2} \in \Z\\
(2n+2)\lfloor \frac{S}{2\pi R^2}\rfloor+n+1\text{ otherwise.}
\end{cases}
\end{equation}
Using the monotonicity property \eqref{eq:index_monotonicity1} and the first inequality in \eqref{eq:pinching} we conclude that if $S=T$ then
\begin{equation}
\label{eq:comparison}
\czl(\Ga) \geq \czl(\Ga_R).
\end{equation}
On the other hand, the second inequality in \eqref{eq:pinching} applied to $v=x$ furnishes
\[
H(x) \leq H_r(x).
\]
(Indeed, note that, by homogeneity, $H(x)= \frac 12\lg d^2H(x)x,x \rg$.) A theorem due to Croke and Weinstein \cite[Theorem A]{CW} establishes that if $H: \R^{2n+2} \to \R$ is a convex Hamiltonian homogeneous of degree two such that $H(x) \leq H_r(x)$ then every non-constant periodic solution of $H$ has period at least $2\pi r^2$.

Now suppose that $R/r<\sqrt\frac{k}{k-1}$ for some real number $k>1$. Then we have that $\lfloor k \rfloor \frac{T}{2\pi R^2} \geq \lfloor k \rfloor\frac{r^2}{R^2} >  \lfloor k \rfloor - 1$. Using this inequality,  \eqref{eq:ind H_R} and \eqref{eq:comparison} we arrive at
\begin{align*}
\czl(\Ga^{\lfloor k\rfloor}) & \geq (2n+2)(\lfloor k\rfloor-1)+n+1 \\
& = (2n+2)\lfloor k\rfloor-n-1.
\end{align*}
Consequently, by \eqref{eq:index_split} we infer that
\begin{equation}
\label{eq:index_estimate}
\czl(\ga^{\lfloor k\rfloor};\Phi^\xi) = \czl((\Ga^\xi)^{\lfloor k\rfloor}) \geq (2n+2)\lfloor k\rfloor-n.
\end{equation}
(Recall here that $\Ga^\xi$ is given by the restriction of $\Ga$ to the contact structure $\xi$; clearly, $(\Ga^\xi)^{\lfloor k\rfloor}$ denotes the $\lfloor k\rfloor$-th iterate of $\Ga^\xi$ and, by  \eqref{eq:index_split}, $\czl((\Ga^\xi)^{\lfloor k\rfloor})=\czl(\Ga^{\lfloor k\rfloor})+1$.) In particular, as mentioned in Section \ref{sec:main_results}, taking $k\to 1$ we conclude that every convex hypersurface in $\R^{2n}$ is dynamically convex.

\section{Holomorphic curves and cylindrical contact homology} 
\label{sec:CH}

\subsection{Cylindrical almost complex structures in symplectizations}
\label{sec:acs in symplectizations}

In what follows we will provide a geometric description of almost complex structures on topologically trivial cobordisms. Although we will follow a more coordinate free approach, this is equivalent to the standard coordinate dependent perspective usually found in the literature (see, for instance, \cite{Bo2}). Let $\xi^\bot\setminus 0$ be the annihilator of $\xi$ in $T^*M$ minus the zero section. It is naturally endowed with the symplectic form $\om_\xi:=d\lambda$, where $\lambda$ is the Liouville 1-form on $T^*M$. The co-orientation of $\xi$ orients the line bundle $TM/\xi \to M$ and consequently also $(TM/\xi)^* \simeq  \xi^\bot$. The symplectization $W$ of $(M,\xi)$ is the connected component of $\xi^\bot\setminus 0$ given by the positive covectors with respect to this orientation.

A choice of a contact form $\alpha$ representing $\xi$ induces the symplectomorphism
\begin{equation}
\label{eq:symp}
\begin{split}
\Psi_\alpha: (W,\om_\xi) & \to (\R \times M,d(e^r\alpha)) \\
\theta & \mapsto (\ln\theta/\alpha,\tau(\theta)),
\end{split}
\end{equation}
where $r$ denotes the $\R$-coordinate, $\tau: T^*M \to M$ is the projection and $\theta/\alpha$ is the unique function on $M$ such that $\theta=(\theta/\alpha)\alpha$. The free additive $\R$-action $c\cdot (r,x) \mapsto (r+c,x)$ on the right side corresponds to $c\cdot\theta \mapsto e^c\theta$ on the left side. An almost complex structure $J$ on $W$ is called cylindrical if $J$ is invariant with respect to this $\R$-action. It is called $\om_\xi$-compatible if $\om_\xi(J\cdot,\cdot)$ defines a Riemannian metric. We say that $J$ is compatible with $\alpha$ if it is cylindrical, $\om_\xi$-compatible and $\tilde J := (\Psi_\alpha)_*J$ satisfies  $\tilde J\partial_r = R_\alpha$ and $\tilde J(\xi)=\xi$. Denote the set of almost complex structures compatible with $\alpha$ by $\J(\alpha)$. It is well known that $\J(\alpha)$ is non-empty and contractible.

\subsection{Almost complex structures with cylindrical ends}
\label{sec:acs in cobordisms}

The fibers of $\tau: W \to M$ can be ordered in the following way: given $\theta_0, \theta_1 \in \tau^{-1}(x)$ we write $\theta_0 < \theta_1$ (resp. $\theta_0 \leq  \theta_1$) when $\theta_1 / \theta_0 > 1$ (resp. $\theta_1 / \theta_0 \geq 1$). Given two contact forms $\alpha_-,\alpha_+$ for $\xi$, we define $\alpha_- < \alpha_+$ if $\alpha_-|_x < \alpha_+|_{x}$ pointwise and, in this case, we set
\[
 \overline{W}(\alpha_-,\alpha_+) = \left \{ \theta \in W \mid \alpha_-|_{\tau(\theta)} \leq \theta \leq \alpha_+|_{\tau(\theta)} \right\}
\]
which is an exact symplectic cobordism between $(M,\alpha_-)$ and $(M, \alpha_+)$. Recall that this means that the oriented boundary of $\overline{W}(\alpha_-,\alpha_+)$ equals $M_+ \cup \overline{M}_-$, where $M_\pm$ are the sections given by $\alpha_\pm$ and $\overline{M}_-$ stands for $M_-$ with the reversed orientation, there is a Liouville vector field $Y$ for $\om_\xi$ (given by the radial one) transverse to the boundary and pointing outwards along $M_+$ and inwards along $M_-$ and the 1-form $i_Y\om_\xi$ restricted to $M_\pm$ equals $\alpha_\pm$. Define
\[
  \begin{split}
    W^-(\alpha_-) &= \left \{ \theta \in W \mid \theta \leq \alpha_-|_{\tau(\theta)} \right\}, \\
    W^+(\alpha_+) &= \left \{ \theta \in W \mid \alpha_+|_{\tau(\theta)} \leq \theta \right\}.
  \end{split}
\]
It follows that
\[
  W= W^-(\alpha_-) \bigcup_{\substack{\partial^+ W^-(\alpha_-) = \\ \partial^- \overline{W}(\alpha_-,\alpha_+)}} \overline{W}(\alpha_-,\alpha_+) \bigcup_{\substack{\partial^+ \overline{W}(\alpha_-,\alpha_+) \\ =  \partial^- W^+(\alpha_+)}} W^+(\alpha_+).
\]
An almost-complex structure $J$ on $W$ satisfying
\begin{itemize}
 \item $J$ coincides with $J_+ \in \J(\alpha_+)$ on a neighborhood of $W^+(\alpha_+)$,
 \item $J$ coincides with $J_- \in \J(\alpha_-)$ on a neighborhood of $W^-(\alpha_-)$,
 \item $J$ is $\omega_{\xi}$-compatible
\end{itemize}
is an almost-complex structure with cylindrical ends. The set of such almost-complex structures will be denoted by $\J(J_-,J_+)$.  It is well known that this is a non-empty contractible set.

\subsection{Splitting almost complex structures}
\label{sec:splittingacs}

Suppose we are given contact forms $\alpha_- < \alpha < \alpha_+$ supporting $\xi$. Let $J_- \in \J(\alpha_-)$, $J \in \J(\alpha)$ and $J_+ \in \J(\alpha_+)$ and consider almost complex structures $J_1 \in \J(J_-,J)$, $J_2 \in \J(J,J_+)$. Let us denote by $g_c(\theta) = e^c\theta$ the $\R$-action on $W$. Given $R\geq 0$, define the almost complex structure
\begin{equation}
\label{eq:splitcs}
J_1 \circ_R J_2 = \left\{ \begin{aligned} &(g_{-R})^*J_2 \ \mbox{ on } W^+(\alpha) \\ &(g_R)^*J_1 \ \mbox{ on } W^-(\alpha) \end{aligned} \right.
\end{equation}
which is smooth since $J$ is $\R$-invariant. We call $J_1 \circ_R J_2$ a splitting almost complex structure. Note that if $\epsilon_0>0$ is small enough then $J_1 \circ_R J_2 \in \J(J_-,J_+)$ for all $0 < R \leq \epsilon_0$.

We will show that $J_1 \circ_R J_2$ is biholomorphic to a complex structure in $\J(J_-,J_+)$ for every $R$ sufficiently large. For this, given $R>0$ take a function $\varphi_R:\R\to \R$ satisfying $\varphi_R(r) = r+R$ if $r\leq -R-\epsilon_0$, $\varphi_R(r) = r-R$ if $r\geq R+\epsilon_0$ and $\varphi_R' > 0$ everywhere. This function can be chosen so that $\sup_{R,r} |\varphi_R'(r)| \leq 1$ and 
\[
\inf_{R>0} \inf \{ \varphi_R'(r) \mid r\in (-\infty,-R] \cup [R,+\infty)\} \geq \frac{1}{2}.
\]
In particular, $\varphi_R^{-1}$ has derivative bounded in the intervals $(-\infty,\varphi_R(-R)]$ and $[\varphi_R(R),+\infty)$ uniformly in $R$. Consider the diffeomorphisms $\psi_R : \R\times M \to \R\times M$ given by $\psi_R(r,x) = (\varphi_R(r),x)$ and define
\[
\Phi_R = \Psi_\alpha^{-1} \circ \psi_R \circ \Psi_\alpha : W \to W,
\]
where $\Psi_\alpha$ is defined in \eqref{eq:symp}. It is straightforward to check that
\begin{equation}
\label{eq:jsplit}
J^\pr_R := (\Phi_R)_*(J_1 \circ_R J_2)
\end{equation}
belongs to $\J(J_-,J_+)$, for every $R$ sufficiently large.

\subsection{Finite energy curves in symplectizations}
\label{sec:curves_symplectizations}

Let  $\Sigma = S^2\setminus\Gamma$ be a punctured rational curve, where $S^2$ is endowed with a complex structure $j$ and $\Gamma = \{x,y_1,...,y_s\}$ is the set of  (ordered) punctures of $\Sigma$. Fix a non-degenerate contact form $\alpha$ for $\xi$ and a cylindrical almost complex structure $J \in \J(\alpha)$. A holomorphic curve from $\Sigma$ to $W$ is a smooth map $u: (\Sigma,j) \to (W,J)$ satisfying $du \circ j = J \circ du$. Its Hofer energy is defined as
\[
E(u) = \sup_{\phi \in \Lambda} \int_\Sigma u^*d\alpha^\phi,
\]
where $\Lambda = \{\phi: \R \to [0,1]; \phi^\pr \geq 0\}$ and $\alpha^\phi := \Psi_\alpha^*(\phi(r)\alpha)$. Let $r=\pi_1\circ \Psi_\alpha\circ u$ and $\tilde u=\pi_2 \circ \Psi_\alpha\circ u$, where $\pi_1: \R \times M \to \R$ and $\pi_2: \R \times M \to M$ are the projections onto the first and second factor respectively. Consider a holomorphic curve $u$ with finite Hofer energy such that $r(z) \to \infty$ as $z\to x$ and $r(z) \to -\infty$ as $z\to y_i$ for $i=1,\dots,s$ (we say that $x$ is the positive puncture and $y_1,\dots,y_s$ are the negative punctures of $u$). Then it can be proved that there are  polar coordinates $(\rho,\theta)$ centered at each puncture $p$ of $\Sigma$ such that
\begin{equation}
\label{eq:punctures}
\lim_{\rho\to 0} \u(\rho,\theta) = 
\begin{cases}
\gamma(-T\theta/2\pi) \text{ if $p=x$} \\
\gamma_i(T_i\theta/2\pi) \text{ if $p=y_i$ for some $i=1,\ldots,s$}.
\end{cases}
\end{equation}
where $\gamma$ and $\gamma_i$ are periodic orbits of $R_\alpha$ of periods $T$ and $T_i$ respectively. Denote the set of such holomorphic curves by $\widehat\M(\gamma,\gamma_1,...,\gamma_s;J)$ and notice that $j$ is not fixed. Define an equivalence relation $\simeq$ on $\widehat\M(\gamma,\gamma_1,...,\gamma_s;J)$ by saying that $u$ and $u^\pr$ are equivalent if there is a biholomorphism $\vr: (S^2,j) \to (S^2,j^\pr)$ such that $\vr(p)=p$ for every $p \in \Gamma$ and $u = u^\pr \circ \vr$. Define the moduli space $\M(\gamma,\gamma_1,...,\gamma_s;J)$ as $\widehat\M(\gamma,\gamma_1,...,\gamma_s;J)/\simeq$.

A crucial ingredient in order to understand the moduli space $\M(\gamma,\gamma_1,...,\gamma_s;J)$ is the linearized Cauchy-Riemann operator $D_{(u,j)}$ associated to each holomorphic curve $(u,j) \in \widehat\M(\gamma,\gamma_1,...,\gamma_s;J)$. If $D_{(u,j)}$ is surjective for every $(u,j) \in \widehat\M(\gamma,\gamma_1,...,\gamma_s;J)$ then $\M(\gamma,\gamma_1,...,\gamma_s;J)$ is either empty or a smooth manifold with dimension given by
\[
\rcz(\gamma) - \sum_{i=1}^s \rcz(\gamma_i).
\]
Unfortunately, this surjectivity cannot be achieved in general and suitable multi-valued perturbations of the Cauchy-Riemann equations are needed in general to furnish a nice structure for $\M(\gamma,\gamma_1,...,\gamma_s;J)$ \cite{HWZ10, HWZ11, HN}. However, these problems can be avoided if we restrict ourselves to the space of \emph{somewhere injective} holomorphic curves. More precisely, $u$ is somewhere injective if there exists $z \in S^2\setminus\Ga$ such that $\u^{-1}(\u(z))=\{z\}$ and $\pi\circ d\u(z) \neq 0$, where $\pi: T_{\u(z)}M \to \xi_{\u(z)}$ is the projection along the Reeb vector field $R_\alpha$ in $\u(z)$. Consider the moduli space $\M_{\text{si}}(\gamma,\gamma_1,...,\gamma_s;J)$ given by somewhere injective curves in $\widehat\M(\gamma,\gamma_1,...,\gamma_s;J)$.

\begin{theorem}[Dragnev \cite{Dra}]
There exists a residual subset $\Jreg(\gamma,\gamma_1,...,\gamma_s) \subset \J(\alpha)$ such that the moduli space $\M_{\text{si}}(\gamma,\gamma_1,...,\gamma_s;J)$ is either empty or a smooth manifold with dimension given by $\rcz(\gamma) - \sum_{i=1}^s \rcz(\gamma_i)$ for every $J \in \Jreg(\gamma,\gamma_1,...,\gamma_s)$.
\end{theorem}

Now notice that, since $\alpha$ is non-degenerate, the set of closed orbits of $R_\alpha$ is countable (here we are tacitly identifying a periodic orbit $\ga(t)$ with $\ga(t+c)$ where $c$ is any real number). It follows from the previous theorem that there exists a residual subset $\Jreg(\alpha) \subset \J(\alpha)$ such that $\M_{\text{si}}(\gamma,\gamma_1,...,\gamma_s;J)$ is a smooth manifold for every periodic orbits $\gamma,\gamma_1,...,\gamma_s$ and all $s \in \N$.

The following proposition gives sufficient conditions to guarantee that every holomorphic curve in $\widehat\M(\gamma,\gamma_1,...,\gamma_s;J)$ is somewhere injective. This will be crucial for us in the proof of Theorems A and B.

\begin{proposition}
\label{prop:somewhere_injective}
If $\gamma$ is simple then every holomorphic curve in $\widehat\M(\gamma,\gamma_1,...,\gamma_s;J)$ is somewhere injective. If $s=1$ and $\gamma_1$ is simple then every holomorphic cylinder in $\widehat\M(\gamma,\gamma_1;J)$ is somewhere injective.
\end{proposition}

\begin{proof}
Suppose that the first claim is not true and let $u \in \widehat\M(\gamma,\gamma_1,...,\gamma_s;J)$ be a curve that is not somewhere injective. Then there exist a holomorphic branched covering $\phi: S^2 \to S^2$ with degree bigger than one and a somewhere injective holomorphic curve $u^\pr: \Sigma^\pr \to W$ such that $u=u^\pr \circ \phi$, where $\Sigma^\pr=S^2\setminus\Ga^\pr:=S^2\setminus\phi(\Ga)$, see \cite[Theorem 1.3]{Dra}. Let $x^\pr=\phi(x)$, where $x$ is the positive puncture of $u$, and notice that $x^\pr$ is a positive puncture of $u^\pr$. Since $u$ has only one positive puncture, $\phi^{-1}(x^\pr)=\{x\}$. Thus, the branching index of $x$ is equal to the degree of $\phi$ which is bigger than one. But this implies that $\gamma$ is not simple, contradicting our hypothesis. The proof of the second claim is absolutely analogous.
\end{proof}

\subsection{Finite energy curves in symplectic cobordisms}
\label{sec:curves in cobordisms}

Consider non-degenerate contact forms $\alpha_- < \alpha_+$ for $\xi$, $J_\pm \in \J(\alpha_\pm)$ and $J \in \J(J_-,J_+)$. The energy of a holomorphic curve $u: (\Sigma,j) \to (W,J)$ is defined as
\[
E(u) = E_-(\u) + E_+(\u) + E_0(\u),
\]
where
\begin{equation*}
\begin{aligned}
& E_-(u) = \sup_{\phi\in\Lambda} \int_{u^{-1}(W^-(\alpha_-))} u^*d\alpha_{-}^\phi \\
& E_+(u) = \sup_{\phi\in\Lambda} \int_{u^{-1}(W^+(\alpha_+))} u^*d\alpha_{+}^\phi \\
& E_0(u) = \int_{u^{-1}(\overline W(\alpha_-,\alpha_+))} u^*\omega_\xi.
\end{aligned}
\end{equation*}
Holomorphic curves with finite energy in cobordisms are asymptotic to periodic orbits of $\alpha_+$ at the positive punctures and to periodic orbits of $\alpha_-$ at the negative punctures in the sense described in \eqref{eq:punctures}. Similar results to those described in the previous section hold for complex structures with cylindrical ends. In particular, given $J_\pm \in \Jreg(\alpha_\pm)$ we have a residual subset $\Jreg(\gamma,\gamma_1,...,\gamma_s;J_-,J_+) \subset \J(J_-,J_+)$ such that if $J \in \Jreg(\gamma,\gamma_1,...,\gamma_s;J_-,J_+)$ then the moduli spaces $\M_{\text{si}}(\gamma,\gamma_1,...,\gamma_s;J)$ is either empty or a smooth manifold with dimension
\[
\rcz(\gamma) - \sum_{i=1}^s \rcz(\gamma_i)
\]
for every periodic orbit $\gamma$ of $\alpha_+$ and periodic orbits $\gamma_1,...,\gamma_s$ of $\alpha_-$. Denote by $\Jreg(J_-,J_+)$ the residual subset of $\J(J_-,J_+)$ given by the intersection of $\Jreg(\gamma,\gamma_1,...,\gamma_s;J_-,J_+)$ running over all the periodic orbits of $\alpha_+$ and $\alpha_-$ and $s \in \N$.

\subsection{Finite energy curves in splitting cobordisms}
\label{sec:cobordisms}

Consider non-degenerate contact forms $\alpha_- < \alpha < \alpha_+$ for $\xi$, $J_\pm \in \J(\alpha_\pm)$, $J \in \J(\alpha)$ and $J_1 \in \J(J_-,J)$, $J_2 \in \J(J,J_+)$. Given $R\geq 0$ we have the splitting almost complex structure $J_1 \circ_R J_2$ defined in Section \ref{sec:splittingacs}. The energy of a holomorphic curve $u: (\Sigma,j) \to (W,J_1 \circ_R J_2)$ is defined as
\begin{equation}
\label{eq:energy_splitting}
E(u) = E_{\alpha_-}(u) + E_{\alpha_+}(u) + E_\alpha(u) + E_{(\alpha,\alpha_+)}(u) + E_{(\alpha_-,\alpha)}(u)
\end{equation}
where
\begin{equation*}
\begin{aligned}
& E_{\alpha_+}(u) = \sup_{\phi\in\Lambda} \int_{u^{-1}(W^+(e^R\alpha_+))} u^*d\alpha_+^\phi,\ \ E_{\alpha}(u) = \sup_{\phi\in\Lambda} \int_{u^{-1}(\overline W(e^{-R}\alpha,e^R\alpha))} u^*d\alpha^\phi, \\
& E_{\alpha_-}(u) = \sup_{\phi\in\Lambda} \int_{u^{-1}(W^-(e^{-R}\alpha_-))} u^*d\alpha_-^\phi,\ \ E_{(\alpha,\alpha_+)}(u) = \int_{u^{-1}(\overline W(e^R\alpha,e^R\alpha_+))} u^*(e^{-R}\omega_\xi), \\
&\text{and}\ E_{(\alpha_-,\alpha)}(u) = \int_{u^{-1}(\overline W(e^{-R}\alpha_-,e^{-R}\alpha))} u^*(e^R\omega_\xi).
\end{aligned}
\end{equation*}
Finite energy holomorphic curves for splitting almost complex structures share all the properties of finite energy curves for almost complex structures with cylindrical ends. These energies are important for the SFT-compactness results used in the proofs of Theorems A and B as explained in Section \ref{sec:proof main theorem}.

\subsection{Cylindrical contact homology}

Cylindrical contact homology is a powerful invariant of contact structures introduced by Eliashberg, Givental and Hofer in their seminal work \cite{EGH}. Let $\alpha$ be a non-degenerate contact form supporting $\xi$ and denote by $\P$ the set of periodic orbits of $R_\alpha$. A periodic orbit of $R_\alpha$ is called \emph{bad} if it is an even iterate of a prime periodic orbit whose parities of the Conley-Zehnder indexes of odd and even iterates disagree. An orbit that is not bad is called \emph{good}. Denote the set of good periodic orbits by $\P^0$.

Consider the chain complex $CC_*(\alpha)$ given by the graded group with coefficients in $\Q$ generated by good periodic orbits of $R_\alpha$ graded 
by their Conley-Zehnder indexes (throughout this paper we are not using the standard convention where the grading of contact homology is given by the reduced Conley-Zehnder index).

The boundary operator $\partial$ is given by counting rigid holomorphic cylinders in the symplectization $W$ of $M$. More precisely, let $J \in \J(\alpha)$ and consider the moduli space $\M(\ga,\bar\ga;J)$ of holomorphic cylinders asymptotic to periodic orbits $\ga$ and $\bar\ga$ of $\alpha$. Since $J$ is cylindrical, if $\ga \neq \bar\ga$ then $\M(\ga,\bar\ga;J)$ carries a free $\R$-action induced by the $\R$-action $c\cdot\theta \mapsto e^c\theta$ on $W$. Define
\[
\partial\gamma = m(\gamma)\sum_{\bar\gamma \in \P^0;\ \cz(\bar\gamma) = \cz(\gamma)-1} \sum_{F \in \M(\gamma,\bar\gamma)/\R}\frac{\text{sign}(F)}{m(F)} \bar\gamma\,,
\]
where $m(\ga)$ is the multiplicity of $\ga$, $\text{sign}(F)$ is the sign of $F$ determined by a coherent orientation of $\M(\gamma,\bar\gamma)$ \cite{BM} and $m(F)$ is the cardinality of the group of automorphisms of $F$ \cite{Bo2}. Suppose that $HC_*(\alpha)$ is well defined, that is, $\partial$ is well defined and $\partial^2=0$ (in particular, we are not addressing the question of invariance). Then clearly cylindrical contact homology has natural filtrations in terms of the action and the free homotopy classes of the periodic orbits. More precisely, let $A_\alpha(\ga) = \int_\gamma \alpha$ be the action of a periodic orbit $\ga \in \P$ and take a real number $T>0$. Consider the chain complex $CC^{a,T}_*(\alpha)$ generated by good periodic orbits of $R_\alpha$ with free homotopy class $a$ and action less than $T$. This is a subcomplex of $CC_*(\alpha)$ and its corresponding homology is denoted by $HC^{a,T}_*(\alpha)$.

Under suitable transversality assumptions, it can be proved that $\partial$ is well defined and $\partial^2=0$ if $\alpha$ does not have contractible periodic orbits with reduced Conley-Zehnder index $1$. Although in general these transversality issues still have to be fixed, contact homology can be rigorously defined in some particular instances, see \cite{BO1,BO2,HMS,HN}. Two such situations are particularly relevant in this work. The first one is when the chain complex is lacunary (that is, $CC_k(\alpha) = 0$ whenever $CC_{k-1}(\alpha) \neq  0$). Then the boundary vanishes identically and consequently we obviously achieve the relation $\partial^2=0$. The second situation is when we consider a free homotopy class with only simple periodic orbits and there is no contractible periodic orbit (in the relevant action window). In this case, every holomorphic curve with one positive puncture asymptotic to one of these orbits is somewhere injective and therefore the contact homology is well defined by standard arguments.

\section{Proof of Theorems A and B}
\label{sec:proof main theorem}

Firstly, note that the hypothesis that $c_1(TN) = \lambda\om$ for some constant $\lambda \in \R$ is equivalent to the condition that $c_1(\xi)|_{H_2(M,\Z)}=0$. In fact, $c_1(\xi)|_{H_2(M,\Z)}=0$ if and only if $c_1(\xi)$ vanishes as an element of $H^2(M,\R)$. But $c_1(\xi)=\pi^*c_1(TN)$, where $\pi: M \to N$ is the quotient projection, and the kernel of $\pi^*: H^2(N,\R) \to H^2(M,\R)$ is generated by the cohomology class of $\om$.

Consider the quotient $\bM=M/G$ and let $\tau: M \to \bM$ be the projection. Denote the cardinality of $G$ by $m$ so that $G \simeq \Z_m$. By the $G$-invariance of $\alpha$, it induces a contact form $\balpha$  on $\bM$ such that $\tau_*R_{\alpha}=R_\balpha$. Analogously, $\beta$ induces a contact form $\bbeta$ on $\bM$ whose Reeb flow $R_\bbeta$ generates a free $S^1$-action (fix the period of the orbits of $R_\bbeta$ equal to one). Let $\ba$ be the free homotopy class of the (simple) orbits of $R_\bbeta$. Note that $\tau_*a = \ba^m$. Let $\bxi$ be the kernel of $\bbeta$. Clearly, if $m=1$ then $\bM=M$, $\bxi=\xi$, $\balpha=\alpha$, $\bbeta=\beta$ and $\ba=a$. Since $\bM$ is the prequantization of $(N,m\omega)$ it follows from the monotonicity of $N$ that $c_1(\bxi)|_{H_2(\bM,\Z)}=0$.

Now, choose a simple orbit $\bvr$ of $R_\bbeta$ as a reference loop for the homotopy class $\ba$. It corresponds to a fiber of the circle bundle $\bpi: \bM \to N$ over some point $q \in N:=M/S^1$. Take the pullback of a (symplectic) frame in $T_qN$ to define a symplectic trivialization $\Psi_t: \bxi(\bvr(t)) \to \R^{2n}$ over $\bvr$. Given a  periodic orbit $\gamma$ in the free homotopy class $\ba$, consider a cylinder connecting $\gamma$ and $\bvr$ and use this to define a trivialization $\Psi$ of $\bxi$ over $\gamma$; see Section \ref{sec:index_orbits}. Since $c_1(\bxi)|_{H_2(\bM,\Z)}=0$, this trivialization is well defined up to homotopy. Moreover, the Robbin-Salamon index of $\bvr$ clearly vanishes with respect to this trivialization.

Similarly, choose a simple orbit $\vr$ of $R_\beta$ as a reference loop for the free homotopy class $a$. Since the conditions on $\alpha$ in Theorems A and B do not depend on the choice of the trivialization of $\xi$ over $\vr$ (see Remark \ref{rmk:trivialization Thm A}), we can take the trivialization induced by $\Psi^m$ so that $\rs(\vr)=0$ with respect to this trivialization. Theorems A and B then clearly follow from the next two results.

\begin{theorem}
\label{Thm 1}
Under the assumptions of Theorem A, there is in case (H1) a periodic orbit $\gamma_\mi$ of $\balpha$ with free homotopy class $\ba$ such that if every contractible periodic orbit $\ga$ of $\alpha$ satisfies $\czl(\ga;\Psi^m)\geq -n$ then $\czl(\ga_\mi;\Psi) \leq -n$. In case (H2), there is a periodic orbit $\gamma_\ma$ of $\alpha$ with free homotopy class $\ba$ satisfying $\czu(\ga_\ma;\Psi) \geq n$.
\end{theorem}

\begin{theorem}
\label{Thm 2}
Let $\ga$ be a periodic orbit of $\balpha$ such that $\czl(\ga;\Psi) \leq -n$ (resp. $\czu(\ga;\Psi) \geq n$). Suppose that there exists an integer $j>1$ such that  $\czl(\ga^j;\Psi^j) \geq -n$ (resp. $\czu(\ga^j;\Psi^j) \leq n$). Then $\ga$ is elliptic and the equality holds in the previous inequalities, that is, $\czl(\ga;\Psi) = \czl(\ga^j;\Psi^j) = -n$ (resp. $\czu(\ga;\Psi) = \czu(\ga^j;\Psi^j) = n$).
\end{theorem}

The next sections are devoted to the proofs of these theorems. The proof of Theorem \ref{Thm 1} is harder since it involves all the machinery of holomorphic curves. With the transversality assumption (see Remark \ref{rmk:transversality1}) the argument can be considerably simplified and some assumptions of Theorem A can be relaxed, as will be explained in the proof; see Remarks \ref{rmk:transversality1}, \ref{rmk:transversality2} and \ref{rmk:transversality3}. The proof of Theorem \ref{Thm 2}, in turn, does not rely on holomorphic curves at all and is based on Bott's formula for symplectic paths developed by Y. Long \cite{Lon02,Lon99}.

\subsection{Preliminaries for the proof of Theorem \ref{Thm 1}}

We start with the following topological fact.

\begin{proposition}
\label{prop:homotopy}
Let $N_\ba=\{k \in \N; \ba^k=\ba\}$. Then $N_\ba=\{1\}$ if and only if $\omega|_{\pi_2(N)}=0$.
\end{proposition}

\begin{proof}
Since the Euler class of the $S^1$-principal bundle $\bpi: \bM \to N$ is given by $m[\omega] \in H^2(N;\Z)$, the corresponding long exact homotopy sequence is
\[
\cdots \to \pi_2(N) \overset{\partial}{\to} \pi_1(S^1) \overset{i_*}{\to} \pi_1(\bM) \overset{\bpi_*}{\to} \pi_1(N) \to \cdots
\]
where $\partial(S) = m\omega(S)[S^1]$ with $[S^1]$ being the generator of $\pi_1(S^1)$ and $i_*$ is the map induced by the inclusion. It follows from this sequence that the map $i_*$ is injective if and only if $\omega|_{\pi_2(N)}=0$. The image of $i_*$ is a normal subgroup because it is the kernel of $\bpi_*$. Since $i_*(\pi_1(S^1))$ is cyclic and normal we infer that $gfg^{-1}=f^{\pm 1}$ for every $g \in \pi_1(\bM)$, where $f=i_*[S^1]$ is the homotopy class of the fiber. So suppose that $\omega|_{\pi_2(N)}=0$. Then $f \neq f^k$ for every $k > 1$ and $f$ is freely homotopic only to $f^{\pm 1}$. Consequently $N_\ba=\{1\}$. On the other hand, if $\omega|_{\pi_2(N)}\neq 0$ then there exists $k>1$ such that $f=f^k$ and therefore $N_\ba \neq \{1\}$.
\end{proof}

Now, assume the hypothesis (H1) in Theorem A. Since $a=0$, $\omega|_{\pi_2(N)}\neq 0$. Take $k \in N_\ba$ such that $k\neq 1$. Let $\Phi_t: \bxi(\bvr^k(t)) \to \R^{2n}$ be the trivialization induced by a constant frame in $T_qN$ and let $\Psi_t: \bxi(\bvr^k(t)) \to \R^{2n}$ be the trivialization over $\bvr^k$ induced by a constant trivialization over $\bvr$ via a homotopy between $\bvr$ and $\bvr^k$ as explained above. We want to compare the Robbin-Salamon indexes of $\bvr^k$ with respect to the trivializations $\Phi$ and $\Psi$. In order to do this, consider a smooth  homotopy $F: [0,1] \times S^1 \to \bM$ such that $F(0,t)=\bvr(t)$ and $F(1,t)=\bvr(kt)$. We have that $\pi \circ F$ passes to the quotient $[0,1] \times S^1/((\{0\} \times S^1) \cup (\{1\} \times S^1)) \simeq S^2$ and therefore it induces a continuous map $f: S^2 \to N$ such that $f$ outside the poles is given by $\pi \circ F$. Clearly $\rs(\bvr^k;\Phi)=0$ and we have that
\[
\rs(\bvr^k;\Psi) = \rs(\bvr^k;\Phi) + 2\lg c_1(TN),S \rg,
\]
where $c_1(TN)$ is the first Chern class of $TN$ and $S$ is the homology class of the 2-cycle $f(S^2)$. On the other hand, we have that
\begin{equation}
\label{eq:shift}
\begin{split}
\lg [\omega],S \rg & = \int_{[0,1]\times S^1} (\pi \circ F)^*\omega = \frac{1}{m}\int_{[0,1]\times S^1} F^*(d\bbeta) \\
& = \frac{1}{m}\bigg(\int_{\bvr^k} \bbeta - \int_{\bvr} \bbeta\bigg) = \frac{1}{m}(k-1) > 0.
\end{split}
\end{equation}
(Recall here that the simple orbits of $R_\bbeta$ have period one.) Thus, it follows from our assumption on the constant of monotonicity $\lambda$ that
\begin{equation}
\label{eq:bound_c_1(TN)}
\rs(\bvr^k;\Psi) = 2\lg c_1(TN),S \rg \geq 4.
\end{equation}

A similar argument shows the following. Consider a capping disk $F: D^2 \to \bM$ such that $F|_{\partial D^2} = \bvr^m$ (note that $\bvr^m$ is contractible since $a=0$) and let $\Upsilon_t: \bxi(\bvr^m(t)) \to \R^{2n}$ be the trivialization induced by $F$. As before, $\pi \circ F$ induces a continuous map $f: S^2 \to N$ and
\[
\rs(\bvr^m;\Upsilon) = 2\lg c_1(TN),S \rg,
\]
where $S$ is the homology class of the 2-cycle $f(S^2)$. Moreover,
\[
\lg \omega,S \rg = \frac{1}{m}\int_{\bvr^m} \bbeta = 1 > 0.
\]
Hence, $\lg c_1(TN),S \rg > 0$ by our hypotheses. Consequently, if the constant of monotonicity $\lambda$ of $N$ is bigger than one we have that
\begin{equation}
\label{eq:lower_bound_index}
\rs(\bvr^m;\Upsilon) = 2\lg c_1(TN),S \rg \geq 4.
\end{equation}

\subsection{Proof of Theorem \ref{Thm 1}}
\label{sec:Proof Thm 1}

For the sake of clarity, we will split the proof in two main steps: computation of the filtered cylindrical contact homology of a suitable non-degenerate perturbation $\bbeta^\pr$ of $\bbeta$ and a neck-stretching argument for $\balpha$.

\subsubsection{Computation of $HC_*^{a,T}(\bbeta^\pr)$}
\label{sec:computationCH}

We will show that the hypotheses of the theorem imply that for a conveniently chosen constant $T>1$ the filtered cylindrical contact  homology $HC^{\ba,T}_*(\bbeta^\pr)$ of a suitable non-degenerate perturbation of $\bbeta^\pr$ of $\bbeta$ is well defined and can be computed. Well defined here means that the moduli space of holomorphic cylinders with virtual dimension zero is finite and that the square of the differential vanishes (in particular, we are not addressing the issue of the invariance of the contact homology); see Section \ref{sec:CH}.

Fix a Morse function $f: N \to \R$ and let $\bar f = f \circ \bar\pi$. The next lemma is extracted from \cite[Lemmas 12 and 13] {Bo2} and establishes that there exists a non-degenerate perturbation $\bbeta^\pr$ of $\bbeta$ such that the periodic orbits of $\bbeta^\pr$ with action less than $T$ correspond to the critical points of $f$.

\begin{lemma} \cite{Bo2}
\label{perturbation}
Given $T>1$ we can choose $\epsilon>0$ small enough such that the periodic orbits of the contact form $\bbeta^\pr:=(1+\epsilon\bar f)\bbeta$ with action less than $T$ are non-degenerate and given by the fibers over the critical points of $f$. Moreover, given a critical point $p$ of $f$ we have $\cz(\gamma_p^k) = \rs(\bvr^k) - n + \text{ind}(p)$ and $A_{\bbeta^\pr}(\gamma_p^k)=k(1+\epsilon f(p))$ for every $k$, where $\gamma_p$ is the corresponding (simple) periodic orbit and $\text{ind}(p)$ is the Morse index of $p$.\footnote{As usual, the indexes of $\gamma_p^k$ and $\bvr^k$ are computed with respect to a fixed homotopy class of trivializations induced by a trivialization over a reference loop of the free homotopy class.}
\end{lemma}

Now, we consider the two cases of Theorem A. 

\noindent {\bf Case (H1).} We can choose $\bar f$ even so that every periodic orbit $\gamma$ of $\bbeta^\pr$ with action less than $T$ has the property that the parities of the Conley-Zehnder indexes of the orbits are the same (observe that $\rs(\bvr^k)$ is always even). Consequently, the chain complex is lacunary. Thus, the filtered cylindrical contact homology is well defined because its differential vanishes identically. Moreover, we conclude that
\begin{equation}
\label{eq:HC-1}
HC^{\ba,T}_{*}(\bbeta^\pr) \cong \oplus_{k \in N_\ba^T} H_{*+ n - \rs(\bvr^k;\Psi)}(N),
\end{equation}
where $N_\ba^T=\{k \in N_\ba;\ k < T\}$ (recall that the period of the simple orbits of $\bbeta$ is fixed equal to one). 

\noindent {\bf Case (H2).} By Proposition \ref{prop:homotopy}, the hypothesis that $\om|_{\pi_2(N)}=0$ implies that $\ba^k \neq 0$ for every $k \in \N$ and $N_\ba=\{1\}$. Moreover, the hypothesis that $\pi_1(N)$ is torsion free implies that the free homotopy class $\ba$ is primitive, that is, every closed curve $\ga$ with free homotopy class $\ba$ is simple; see \cite[Lemmas 4.1 and 4.2]{GGM}. As a matter of fact, suppose that $\ga=\psi^k$ for some $k \in \N$ where $\psi$ is a simple closed curve with homotopy class $[\psi]$. We claim that $k=1$. In fact, the proof of Proposition \ref{prop:homotopy} shows that $[\ga]=[\bvr]^{\pm 1}$ which implies that $\pi_*[\ga]=0 \implies \pi_*[\psi]^k=0 \implies \pi_*[\psi]=0$ since $\pi_1(N)$ is torsion free. Thus, $[\psi]$ is in the image of the map $i_*: \pi_1(S^1) \to \pi_1(\bM)$ induced by the inclusion of the fiber. But $\ba^j\neq \ba$ for every $j\neq 1$, implying that $k=1$, as claimed.

Consequently, every periodic orbit of $\bbeta^\pr$ with free homotopy class $\ba$ is simple and $\bbeta^\pr$ has no contractible periodic orbit with action less than $T$. Thus, every holomorphic cylinder connecting periodic orbits of $\bbeta^\pr$ with homotopy class $\ba$ and action less than $T$ is somewhere injective (see Proposition \ref{prop:somewhere_injective})  and the moduli spaces used to define $HC_*^{\ba,T}(\bbeta^\pr)$ can be cut out transversally by choosing a generic almost complex structure on the symplectization of $\bM$; see Section \ref{sec:curves_symplectizations}. To compute $HC_*^{\ba,T}(\bbeta^\pr)$, we can employ the Morse-Bott techniques from \cite{Bo1} and it turns out that
\begin{equation}
\label{eq:HC-2}
HC^{\ba,T}_*(\bbeta^\pr) \cong H_{*+ n}(N).
\end{equation}

\subsubsection{A neck-stretching argument for $\balpha$}
\label{sec:neck-stretching}

In this section, we will use a neck-stretching argument to obtain periodic orbits for $\balpha$ using cobordisms between suitable multiples of $\bbeta$. This argument was nicely used by Hryniewicz-Momin-Salom\~ao \cite{HMS} in a similar context. It remounts to a sandwich trick often used in Floer homology \cite{BPS} but is more involved since the cylindrical contact homology of $\balpha$ does not need to be well defined. Moreover, we have to take care of possible holomorphic curves with negative index that can appear in the boundary of the SFT-compactification of the space of holomorphic curves with fixed asymptotes. 

More precisely, choose positive constants $c_- < c_+$ such that $c_-\bbeta < \balpha < c_+\bbeta$ and define $\bbeta_\pm = c_\pm\bbeta$ for notational convenience. Multiplying $\balpha$ by a constant if necessary, we can assume that $c_+=1$. In our argument, we do not need to assume that the cylindrical contact homology of $\balpha$ is well defined. However, in order to merely illustrate the idea of the sandwich trick, suppose that $HC^{\ba,T}_*(\balpha)$ is well defined and that all the relevant moduli spaces can be cut out transversally (we stress the fact that we are not assuming this in this work). Then one can show that there are maps $\psi^{\bbeta_+}_{\bbeta_-}$, $\psi^{\bbeta_+}_\balpha$ and $\psi^\balpha_{\bbeta_-}$ that fit in the commutative diagram
\begin{equation}
\label{eq:triangle}
\xymatrix{HC^{\ba,T}_*(\bbeta_+) \ar[rr]^{\psi^{\bbeta_+}_{\bbeta_-}}
\ar[dr]_{\psi^{\bbeta_+}_\balpha} && HC^{\ba,T}_*(\bbeta_-)\\
& HC^{\ba,T}_*(\balpha) \ar[ur]_{\psi^\balpha_{\bbeta_-}}&}
\end{equation}
such that the map $\psi^{\bbeta_+}_{\bbeta_-}$ is injective. The argument to prove this is standard. The only issue that is not completely standard is the fact that $\bbeta^\pr$ is allowed to have a contractible periodic orbit $\gamma$ with action less than $T$ such that the reduced Conley-Zehnder  $\rcz(\gamma)$ vanishes. (This can happen only if we relax the assumption on the constant of monotonicity $\lambda$ in the hypothesis (H1) of Theorem A; see Remark \ref{rmk:transversality2}.) However, this difficulty is overcome using the fact that, in this case, the symplectization of $\bbeta^\pr$ admits no rigid non-trivial holomorphic curve with one positive puncture; see \cite[Proposition 4.2]{AM} for details. Using the computation of $HC^{\ba,T}_*(\bbeta^\pr)$ one can easily conclude Theorem \ref{Thm 1} under the assumptions that $HC^{\ba,T}_*(\balpha)$ is well defined and that all the relevant moduli spaces can be cut out transversally. In fact, since $HC^{\ba,T}_{-n}(\bbeta_+) \neq 0$ and $HC^{\ba,T}_{n}(\bbeta_+) \neq 0$ we conclude that any non-degenerate perturbation $\balpha^\pr$ of $\balpha$ has periodic orbits $\gamma^\pr_\mi$ and $\gamma^\pr_\ma$ with action bounded by $T$, free homotopy class $\ba$ and such that $\cz(\ga^\pr_\mi) = -n$ and $\cz(\ga^\pr_\ma) = n$. Thus, these periodic orbits must come from bifurcations of periodic orbits $\gamma_\mi$ and $\gamma_\ma$ of $\balpha$ such that $\czl(\ga_\mi) \leq -n$ and $\czu(\ga_\ma) \geq n$.

The proof without these hypotheses is much more involved and is based on a neck-stretching argument and a careful analysis of the holomorphic curves that can appear in the boundary of the SFT-compactification of the relevant moduli spaces. In what follows, we refer to Section \ref{sec:CH} for the definitions. Let $\bW$ be the symplectization of $\bM$ and fix from now on a non-degenerate perturbation $\balpha^\pr$ of $\balpha$ and almost complex structures $J_\pm \in \Jreg(\bbeta^\pr_\pm)$, $J \in \Jreg(\balpha^\pr)$, $\bar J_1 \in \Jreg(J_-,J)$ and $\bar J_2 \in \Jreg(J,J_+)$ on $\bW$. Recall that these complex structures are regular only for \emph{somewhere injective} curves.

Take $T>1$ and choose a Morse function $f$ on $N$ with only one local minimum and only one local maximum and such that $f$ is even in case (H1). The existence of such function in case (H2) follows from an argument using cancellation of critical points and it is a classical result due to Morse \cite{Mor}: every closed manifold admits a polar function, that is, a Morse function with only one local minimum and one local maximum. Notice that in case (H1), an even Morse function is obviously perfect and therefore has only one local minimum and one local maximum.

Let $\bbeta^\pr$ be the non-degenerate perturbation of $\bbeta$ given by Lemma \ref{perturbation} and define $\bbeta^\pr_\pm = c_\pm\bbeta^\pr$. 
By our choice of $f$ we have that:

\noindent {\bf Case (H1).} There is only one orbit $\ga^\pm_\mi$  of $\bbeta_\pm^\pr$ that generates $HC_{-n}^{\ba,T}(\bbeta^\pr_\pm)$. In fact,
by \eqref{eq:bound_c_1(TN)} and \eqref{eq:HC-1}, we have that  $HC_{-n}^{\ba,T}(\bbeta^\pr_\pm) \simeq \Q$. Thus there is only one 
orbit $\ga^\pm_\mi$  of $\bbeta_\pm^\pr$ that generates $HC_{-n}^{\ba,T}(\bbeta^\pr_\pm)$ since the differential vanishes.

\noindent {\bf Case (H2).} Since $N_\ba^T=\{1\}$ and $f$ has only one local maximum, there is only one orbit $\ga^\pm_\ma$ that generates $HC_{n}^{\ba,T}(\bbeta^\pr_\pm)$.

Given a sequence of real numbers $R_n\to \infty$ consider the splitting almost complex structures $J_{R_n} = \bar J_1 \circ_{R_n} \bar J_2$ defined in \eqref{eq:splitcs}. As explained in Section \ref{sec:splittingacs}, taking a subsequence if necessary, we have that $J_{R_n}$ is biholomorphic to an almost complex structure $J^\pr_{R_n} \in \J(J_-,J_+)$ for every $n$. 

\begin{proposition}\label{cylinder}
\ 

\noindent {\bf Case (H1).} There exists a $J_{R_n}$-holomorphic cylinder $u_n: \R \times \R/\Z \to \bW$ positively asymptotic to $\ga^+_\mi$ and negatively asymptotic to $\ga^-_\mi$ for every $n$. 

\noindent {\bf Case (H2).} There exists a $J_{R_n}$-holomorphic cylinder $v_n: \R \times \R/\Z \to \bW$ positively asymptotic to $\ga^+_\ma$ and negatively asymptotic to $\ga^-_\ma$ for every $n$.
\end{proposition}

\begin{remark}
\label{rmk:transversality2}
If we assume that all the relevant moduli spaces can be cut out transversally then the proof of the previous proposition is easier and does not need the hypothesis on the constant of monotonicity $\lambda$ in case (H1) of Theorem A. In fact, without this hypothesis we may have possibly several periodic orbits $\ga^{\pm,i}_\mi$ ($i=1,\dots,l_\pm$) of $\bbeta_\pm^\pr$ that generate $HC_{-n}^{\ba,T}(\bbeta_\pm)$. Given a periodic orbit $\ga^{+,i}_\mi$ there is, for every $n$, a $J_{R_n}$-holomorphic cylinder $u_n: \R \times \R/\Z \to \bW$ such that $u_n$ is positively asymptotic to $\ga^{+,i}_\mi$ and negatively asymptotic to $\ga^{-,j}_\mi$ for some $j \in \{1,\dots,l_-\}$. Indeed, if $u_n$ does not exist we would conclude that the injective maps $\psi^{\bbeta_+}_{\bbeta_-}$ from the commutative triangle \eqref{eq:triangle} satisfy $\psi^{\bbeta_+}_{\bbeta_-}([\ga^{+,i}_\mi]) = 0$, a contradiction. The argument for the existence of $v_n$ is analogous.
\end{remark}

Proposition \ref{cylinder} is a consequence of the next three lemmas.

\begin{lemma} \label{lemma1}
\ 

\noindent {\bf Case (H1).} Given $J \in \Jreg(J_-,J_+)$, the moduli space $\M(\ga^+_\mi,\ga^-_\mi;J)$ is finite.

\noindent {\bf Case (H2).} Given $J \in \Jreg(J_-,J_+)$, the moduli space $\M(\ga^+_\ma,\ga^-_\ma;J)$ is finite. 
\end{lemma}

\begin{proof}
Since $\ga^+_\mi$ (resp. $\ga^+_\ma$) is simple, every holomorphic cylinder in $\widehat\M(\ga^+_\mi,\ga^-_\mi;J)$ (resp. $\widehat\M(\ga^+_\ma,\ga^-_\ma;J)$) is somewhere injective and consequently $\M(\ga^+_\mi,\ga^-_\mi;J)$ (resp. $\M(\ga^+_\ma,\ga^-_\ma;J)$) is a smooth manifold of dimension zero. Consider a sequence $u_n$ in $\M(\ga^+_\mi,\ga^-_\mi;J)$ (resp. $\M(\ga^+_\ma,\ga^-_\ma;J)$). We will show that $u_n$ has a subsequence converging to an element $u_\infty$ in $\M(\ga^+_\mi,\ga^-_\mi;J)$ (resp. $\M(\ga^+_\ma,\ga^-_\ma;J)$) and therefore $\M(\ga^+_\mi,\ga^-_\mi;J)$ (resp. $\M(\ga^+_\ma,\ga^-_\ma;J)$) is a finite set.

Since the curves $u_n$ have fixed asymptotes, there exists $C>0$ such that
\[
E(u_n) < C
\]
for every $n$. Therefore, we can employ the SFT-compactness results from \cite{BEHWZ} to study the limit of $u_n$ in the compactification of the space of holomorphic curves. It follows from this that there are integer numbers $1\leq l \leq m$ and a subsequence $u_n$ converging in a suitable sense to a holomorphic building given by
\begin{itemize}
\item (possibly several) $J_+$-holomorphic curves $u^1,\dots,u^{l-1}$;
\item one $J$-holomorphic curve $u^l$;
\item (possibly several) $J_-$-holomorphic curves $u^{l+1},\dots,u^m$.
\end{itemize}
Moreover, these holomorphic curves satisfy the following properties:
\begin{itemize}
\item[(a)] Each connected component of $u^i$ has only one positive puncture. Moreover, $u^i$ has at least one connected component not given by a trivial vertical cylinder over a periodic orbit.
\item[(b)] $u^1$ is connected and positively asymptotic to $\ga^+_\mi$ (resp. $\ga^+_\ma$).
\item[(c)] $u^m$ has only one negative puncture and it is asymptotic to $\ga^-_\mi$ (resp. $\ga^-_\ma$).
\item[(d)] The boundary data of the curve $u^i$ at the negative end match the boundary data of $u^{i+1}$ on the positive end.
\item[(e)] The sum of the indexes of $u^i$ equals the index of $u_n$ and therefore vanishes.
\end{itemize}
Now we split the argument in cases (H1) and (H2). 

\noindent {\bf Case (H1).} Since $\ga^+_\mi$ has minimal action (that is, every periodic orbit of $\bbeta^\pr_+$ has action bigger than or equal to the action of $\ga^+_\mi$) the $J_+$-holomorphic curves $u^1,\dots,u^{l-1}$ can be only vertical cylinders over $\ga^+_\mi$. But,  by property (a), at least one connected component  of each level of the building is not given by a trivial cylinder and therefore $l=1$; see \cite{BEHWZ}. Thus, $u^l=u^1$ is positively asymptotic to $\ga^+_\mi$ and consequently is somewhere injective. Let $\ga^-_1,\dots,\ga^-_l$ be the negative asymptotes of $u^l$. We have that
\[
\text{index}(u^l) = \rcz(\ga^+_\mi) - \sum_{i=1}^l \rcz(\ga^-_i) \geq 0,
\]
where $\rcz$ is the reduced Conley-Zehnder index. It follows from the properties above that only one of the periodic orbits $\ga^-_1,\dots,\ga^-_l$ has free homotopy class $\ba$ and all the remaining orbits are contractible. But in case (H1) every contractible periodic orbit $\ga$ of $\bbeta^\pr_-$ with action less than $T$ has positive reduced Conley-Zehnder index with respect to the trivialization given by a capping disk. Indeed, these orbits must come from bifurcations of contractible periodic orbits of $\bbeta$ given by $\bvr^{km}$ for some $k \in \N$. But this implies that
\[
\rcz(\ga) = \cz(\ga) + n - 2 \geq \rs(\bvr^{km}) - n + n - 2 \geq 2,
\]
where the last inequality follows from our hypothesis that $\lambda>1 \implies \rs(\bvr^m)=\rs(\vr)>2 \implies \rs(\bvr^m)\geq4$, where the index is computed using a capping disk (note that $\rs(\bvr^{km})=k\rs(\bvr^m)$ because $R_\bbeta$ defines a circle action). Moreover, every periodic orbit of $\bbeta^\pr_-$ with free homotopy class $\ba$ and action less than $T$ has Conley-Zehnder index bigger than or equal to $\cz(\ga^+_\mi)$. We conclude then that $u^l$ has to be a cylinder negatively asymptotic to $\ga^-_\mi$ because $\ga^-_\mi$ is the only periodic orbit of $\bbeta^\pr_-$ with index equal to $\cz(\ga^+_\mi)$ and action less than $T$. Since $\ga^-_\mi$ has minimal action we conclude that the $J_-$-holomorphic curves $u^{l+1},\dots,u^m$ can be only vertical cylinders over $\ga^-_\mi$. Therefore, the only non-trivial component of the holomorphic building is the cylinder $u_\infty:=u^l=u^1$.

\noindent {\bf Case (H2).} Since every periodic orbit of $\bbeta_\pm^\pr$ with free homotopy class $\ba$ is simple and $\bbeta_\pm^\pr$ do not have contractible periodic orbits with action less than $T$, then every holomorphic curve in this building must be a somewhere injective cylinder. Consequently, every non-trivial cylinder $u^1,\dots,u^{l-1},u^{l+1},\dots,u^m$ has positive index. But
\[
\sum_{i=1}^m \text{index}(u^i)=0.
\]
Consequently we get only one non-trivial holomorphic cylinder $u_\infty:=u^l = u^1 \in \M(\ga^+_\ma,\ga^-_\ma;J)$, as desired.
\end{proof}

Based on the previous lemma, we can define the maps $N_\mi: \Jreg(\ga^+_\mi,\ga^-_\mi;J_-,J_+) \to \Z_2$ and $N_\ma: \Jreg(\ga^+_\ma,\ga^-_\ma;J_-,J_+) \to \Z_2$ given by
\[
N_\mi(J) = \#\M(\ga^+_\mi,\ga^-_\mi;J) \text{(mod 2)}\text{ and }N_\ma(J) = \#\M(\ga^+_\ma,\ga^-_\ma;J) \text{(mod 2)},
\]
in cases (H1) and (H2) respectively.

\begin{lemma}\label{lemma2}
\ 

\noindent {\bf Case (H1).} There exist $J_\pm \in \Jreg(\bbeta^\pr_\pm)$ and $J \in \Jreg(J_-,J_+)$ such that $N_\mi(J)=1$. 

\noindent {\bf Case (H2).} There exist $J_\pm \in \Jreg(\bbeta^\pr_\pm)$ and $J \in \Jreg(J_-,J_+)$ such that $N_\ma(J)=1$.
\end{lemma}

\begin{proof}
We will follow an argument from \cite{HMS}. Let $J_+ \in \Jreg(\bbeta^\pr_+)$ (we know that this set is not empty). Identify the symplectization $\bW$ of $\bM$ with $\R \times \bM$ through the diffeomorphism  $\Psi_{\bbeta_+^\pr}$ defined in \eqref{eq:symp}. Consider the diffeomorphism $F: \bW \to \bW$ given by $F(r,x)=(r/c_-,x)$ and define $J_- := F^*J_+ \in \J(\bbeta^\pr_-)$ (recall that $c_+=1$). Since every $J_-$-holomorphic curve is of the form $F^{-1}\circ u$ for some $J_+$-holomorphic curve $u$, we infer that $J_- \in  \Jreg(\bbeta^\pr_-)$.

Let $f: \R \to \R$ be a smooth non-increasing function such that $f \equiv 1/c_-$ near $(-\infty,\ln (c_-)]$ and $f\equiv 1$ near $[0,\infty)$. Consider the almost complex structure $J$ defined by
\[
J(\partial_r) = fR_{\bbeta^\pr_+},\ J(R_{\bbeta^\pr_+}) = -\frac{1}{f}\partial_r \text{ and } J|_\xi=J_+|_\xi=J_-|_\xi.
\]
By construction, $J \in \J(J_-,J_+)$. We claim that there exists a diffeomorphism $G: \bW \to \bW$ such that $G^*J=J_+$. Indeed, consider the function $g: \R \to \R$ given by the unique solution of the initial value problem
\[
g^\pr(r)=1/f(g(r)),\ g(0)=0
\]
and define $G(r,x)=(g(r),x)$. We have that
\begin{align*}
G^*J|_{(r,x)}\partial_r & = dG^{-1}|_{(g(r),x)}J|_{(g(r),x)}g^\pr(r)\partial_r \\
& = dG^{-1}|_{(g(r),x)}f(g(r))g^\pr(r)R_{\bbeta^\pr_+} \\
& = dG^{-1}|_{(g(r),x)}R_{\bbeta^\pr_+} = R_{\bbeta^\pr_+},
\end{align*}
proving the claim. Therefore, since $J_+ \in \Jreg(\bbeta^\pr_+)$ we conclude that $J \in \Jreg(J_-,J_+)$. It is easy to see that the biholomorphism $G$ induces a bijective correspondence between the moduli spaces $\M(\ga^+_\mi,\ga^-_\mi;J)$ and $\M(\ga^+_\mi,\ga^+_\mi;J_+)$. But $\M(\ga^+_\mi,\ga^+_\mi;J_+)$ contains only the vertical cylinder over $\ga^+_\mi$. In particular, $N_\mi(J)=1$. The argument for $N_\ma$ is analogous and left to the reader.
\end{proof}

\begin{lemma}
\label{lemma3}
$N_\mi$ and $N_\ma$ are constant functions.
\end{lemma}

\begin{proof}
Let $J_0$ and $J_1$ in $\Jreg(\ga^+_\mi,\ga^-_\mi;J_-,J_+)$ (resp. $\Jreg(\ga^+_\ma,\ga^-_\ma;J_-,J_+)$). It is well known that we can find a family of almost complex structures $J_t \in \J(\ga^+_\mi,\ga^-_\mi;J_-,J_+)$ (resp. $J_t \in \J(\ga^+_\ma,\ga^-_\ma;J_-,J_+)$)  joining $J_0$ and $J_1$ such that $J_t$ is regular for every $0\leq t \leq 1$ except at finitely many values $0<s_1<\dots<s_k<1$ (recall that regularity here means that the linearized Cauchy-Riemann operator is surjective only for somewhere injective curves). For these values, the somewhere injective $J_{s_j}$-holomorphic curves have index bigger than or equal to $-1$.

\noindent {\bf Case (H1).} Define the set
\[
\M_\mi := \{(t,u)\mid t \in [0,1]\text{ and }u \in \M(\ga^+_\mi,\ga^-_\mi;J_t)\}.
\]
This is a smooth 1-dimensional manifold admitting a suitable SFT-compactification. (Notice that every $u \in  \M(\ga^+_\mi,\ga^-_\mi;J_t)$ is somewhere injective.) Given a sequence $(t_i,u_i) \in \M_\mi$ such that $t_i \to s_j$ for some $1\leq j \leq k$, it follows from SFT-compactness (notice here that $E(u_i)$ is uniformly bounded because the asymptotes are fixed) that $u_i$ converges to a holomorphic building given by
\begin{itemize}
\item (possibly several) $J_+$-holomorphic curves $u^1,\dots,u^{l-1}$;
\item one $J_{s_j}$-holomorphic curve $u^l$;
\item (possibly several) $J_-$-holomorphic curves $u^{l+1},\dots,u^m$.
\end{itemize}
Since $\ga^+_\mi$ has minimal action, the $J_+$-holomorphic curves are only trivial vertical cylinders. Thus, $u^l=u^1$ is positively asymptotic to $\ga^+_\mi$ and consequently is somewhere injective. Let $\ga^-_1,\dots,\ga^-_l$ be the negative asymptotes of $u^l$. We have that $u^l$ cannot have index $-1$ because
\[
\text{index}(u^l) = \rcz(\ga^+_\mi) - \sum_{i=1}^l \rcz(\ga^-_i),
\]
and the reduced Conley-Zehnder index of every periodic orbit of $\bbeta^\pr_-$ with action less than $T$ is even. Consequently, arguing exactly as in the proof of Lemma \ref{lemma1}, we conclude that $u^l$ is a cylinder negatively asymptotic to $\ga^-_\mi$ (because $\text{index}(u^l)\geq 0$ and every contractible periodic orbit of $\bbeta^\pr_-$ with action less than $T$ has positive reduced Conley-Zehnder index). But $\ga^-_\mi$ has minimal action, and therefore the $J_-$-holomorphic curves are only trivial vertical cylinders. Consequently, $\M_\mi$ is a compact 1-dimensional manifold such that
\[
\partial \M_\mi =  \M(\ga^+_\mi,\ga^-_\mi;J_0) \cup \M(\ga^+_\mi,\ga^-_\mi;J_1).
\]
It follows from this that $N_\mi(J_0)=N_\mi(J_1)$.

\noindent {\bf Case (H2).} Define the set
\[
\M_\ma := \{(t,u)\mid t \in [0,1]\text{ and }u \in \M(\ga^+_\ma,\ga^-_\ma;J_t)\}
\]
As before, consider a sequence $(t_i,u_i) \in \M_\ma$ such that $t_i \to s_j$ for some $1\leq j \leq k$. From SFT-compactness, $u_i$ converges to a holomorphic building given by
\begin{itemize}
\item (possibly several) $J_+$-holomorphic curves $u^1,\dots,u^{l-1}$;
\item one $J_{s_j}$-holomorphic curve $u^l$;
\item (possibly several) $J_-$-holomorphic curves $u^{l+1},\dots,u^m$.
\end{itemize}
Since every periodic orbit of $\bbeta^\pr_\pm$ with free homotopy class $\ba$ and action less than $T$ is simple and there is no contractible periodic orbit with action less than $T$, every holomorphic curve in this building is a somewhere injective cylinder. Hence, every $J_\pm$-holomorphic curve $u^1,\dots,u^{l-1},\allowbreak u^{l+1},\dots,u^m$ has positive index and the index of $u^l$ is bigger than or equal to $-1$. We conclude then that if this building has more than one non-trivial cylinder then we have either
\begin{itemize}
\item  exactly one $J_+$-holomorphic cylinder (positively asymptotic to $\ga^+_\ma$) with index $1$ and one $J_{s_j}$-holomorphic cylinder (negatively asymptotic to $\ga^-_\ma$) with index $-1$;
\item  or exactly one $J_{s_j}$-holomorphic cylinder (positively asymptotic to $\ga^+_\ma$) with index $-1$ and one $J_-$-holomorphic cylinder (negatively asymptotic to $\ga^-_\ma$) with index $1$.
\end{itemize}
The last case is impossible because every periodic orbit of $\bbeta^\pr_-$ with free homotopy class $\ba$ and action less than $T$ has Conley-Zehnder index less than or equal to $\cz(\ga^+_\ma)$. Thus, by a gluing argument we conclude that
\[
\partial \M_\ma = \M(\ga^+_\ma,\ga^-_\ma;J_0) \cup \M(\ga^+_\ma,\ga^-_\ma;J_1) \cup \big( \bigcup_{j=1}^k \M_{\text{broken}}(\ga^+_\ma,\ga^-_\ma;J_+,J_{s_j}) \big)
\]
where $\M_{\text{broken}}(\ga^+_\ma,\ga^-_\ma;J_+,J_{s_j})$ is the moduli space of the broken cylinders given by one $J_+$-holomorphic cylinder with index $1$ positively asymptotic to $\ga^+_\ma$ and one $J_{s_j}$-holomorphic cylinder with index $-1$ negatively asymptotic to $\ga^-_\ma$. But by \eqref{eq:HC-2} we have that
\[
HC^{\ba,T}_n(\bbeta^\pr_+) \simeq \Q.
\]
Since $\ga^+_\ma$ is the unique periodic orbit $\bbeta^\pr_+$ with free homotopy class $\ba$, action less than $T$ and Conley-Zehnder index equal to $n$, it is necessarily a generator of $HC^{\ba,T}_n(\bbeta^\pr_+)$ and consequently the cardinality of the moduli space of $J_+$-holomorphic cylinders positively asymptotic to $\ga^+_\ma$ and negatively asymptotic to a periodic orbit with index $\cz(\ga^+_\ma)-1$ must be even. (Observe here that every periodic orbit of $\bbeta^\pr_+$ with free homotopy class $\ba$ and action less than $T$ is simple.) Thus, the cardinality of $\M_{\text{broken}}(\ga^+_\ma,\ga^-_\ma;J_{s_j},J_+)$ is even and we conclude that
\[
\#\M(\ga^+_\ma,\ga^-_\ma;J_0) = \#\M(\ga^+_\ma,\ga^-_\ma;J_1)\ \text{(mod 2)},
\]
as desired.
\end{proof}

\begin{proof}[Proof of Proposition \ref{cylinder}]
If $u_n$ does not exist, we would conclude that $J^\pr_{R_n} \in \Jreg(\ga^+_\mi,\ga^-_\mi;\allowbreak J_-,J_+)$ and $N_\mi(J^\pr_{R_n})=0$, contradicting Lemmas \ref{lemma2} and \ref{lemma3}. The same argument implies the existence of $v_n$.
\end{proof}

We will now finish the proof of Theorem \ref{Thm 1}. 

\noindent {\bf Case (H1).} Take a sequence $R_n \to \infty$ and a sequence of $J_{R_n}$-holomorphic cylinders $u_n: \R \times \R/\Z \to \bW$ positively asymptotic to $\ga^+_\mi$ and negatively asymptotic to $\ga^-_\mi$. The existence of such cylinders is ensured by Proposition \ref{cylinder}. Since the asymptotes are fixed, the energy of $u_n$ is uniformly bounded. Therefore, by SFT-compactness there are integer numbers $1\leq l < l^\pr \leq m$ and a subsequence $u_n$ converging in a suitable sense to a holomorphic building given by
\begin{itemize}
\item (possibly several) $J_+$-holomorphic curves $u^1,\dots,u^{l-1}$;
\item one $\bar J_2$-holomorphic curve $u^l$;
\item (possibly several) $J$-holomorphic curves $u^{l+1},\dots,u^{l^\pr-1}$;
\item one $\bar J_1$-holomorphic curve $u^{l^\pr}$;
\item (possibly several) $J_-$-holomorphic curves $u^{l^\pr+1},\dots,u^m$.
\end{itemize}
Moreover, these holomorphic curves satisfy the properties (a)-(e) mentioned in the proof of Lemma \ref{lemma1}. Since $\ga^+_\mi$ has minimal action, every $J_+$-holomorphic curve $u^1,\dots,u^{l-1}$ must contain only trivial vertical cylinders. Consequently, $u^l=u^1$ is positively asymptotic to $\ga^+_\mi$ which implies that it is somewhere injective. Let $\ga^-_1,\dots,\ga^-_k$ be the negative asymptotes of $u^l$. Reordering these orbits if necessary, we have that $\ga^\pr_\mi := \ga^-_1$ has free homotopy class $\ba$ and $\ga^-_2,\dots,\ga^-_k$ are contractible. Thus, $\balpha^\pr$ carries a periodic orbit $\ga^\pr_\mi$ with free homotopy class $\ba$; see Figure \ref{fig:SFT compactness}.

\begin{figure}[h]
\includegraphics[width=5.2in, height=3in]{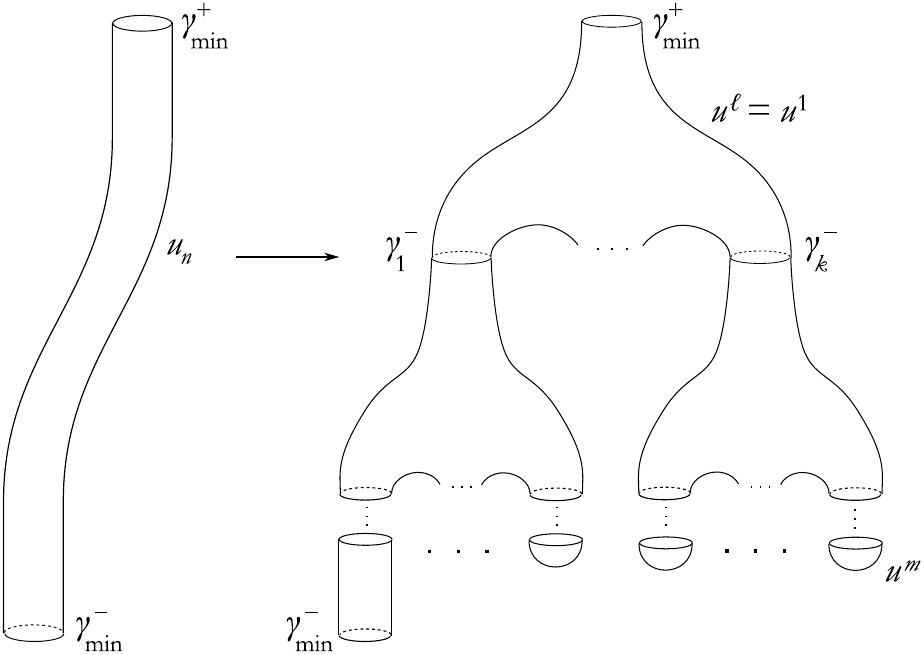}
\centering
\caption{SFT compactness applied to the sequence $u_n$.}
\label{fig:SFT compactness}
\end{figure}

Now, suppose that every contractible periodic orbit $\tilde\ga$ of $\alpha$ satisfies $\czl(\tilde\ga)\geq \rs(\vr)-n$. Then $\rcz(\ga^-_i)\geq 0$ for every $i\geq 2$, where the index is computed using a trivialization $\Upsilon$ given by a capping disk. Indeed, by our assumption $\cz(\ga^-_i,\Upsilon) \geq \rs(\vr,\Upsilon) - n$. Thus we arrive at
\[
\rcz(\ga^-_i) = \cz(\ga^-_i,\Upsilon) + n-2 \geq \rs(\vr,\Upsilon) - 2 \geq 2,
\]
where the last inequality follows from the assumption that $\lambda>1 \implies \rs(\vr,\Upsilon)>2$ which implies that $\rs(\vr,\Upsilon)\geq 4$ because $R_\beta$ defines a circle action and therefore $\rs(\vr,\Upsilon)$ is an even integer. Hence,
\[
\cz(\ga^+_\mi) - \cz(\ga^\pr_\mi) \geq \text{index}(u^l) \geq 0.
\]
This implies that $\cz(\ga^\pr_\mi) \leq -n$.

\begin{remark}
\label{rmk:transversality3}
Notice that to achieve the relation $\rcz(\ga_i)\geq 0$ we only need the weaker condition that $\lambda>0$. Thus, this is the only hypothesis we need in case (H1) of Theorem A if we assume that the relevant moduli space of holomorphic curves can be cut out transversally.
\end{remark}

\noindent {\bf Case (H2).} Consider a sequence of $J_{R_n}$-holomorphic cylinders $v_n: \R \times \R/\Z \to \bW$ positively asymptotic to $\ga^+_\ma$ and negatively asymptotic to $\ga^-_\ma$. There is a subsequence converging to a holomorphic building as before. Since every periodic orbit of $\bbeta^\pr_-$ with free homotopy class $\ba$ has action less than or equal to the action of $\ga^-_\ma$, the $J_-$-holomorphic curves $v^{l^\pr+1},\dots,v^m$ are trivial vertical cylinders over $\ga^-_\ma$. Thus, the first (non-trivial) level of this building $v^{l^\pr}$ must contain a $J_1$-holomorphic cylinder $v$ negatively asymptotic to $\ga^-_\ma$. Since $\ga^-_\ma$ is simple, this cylinder is somewhere injective. Let $\ga^\pr_\ma$ be the positive asymptote of $v$. We have that
\[
\text{index}(v) = \cz(\ga^\pr_\ma) - \cz(\ga^-_\ma) \geq 0
\]
so that $\cz(\ga^\pr_\ma) \geq n$.

Finally, as explained before, $\balpha^\pr$ is arbitrarily closed to $\balpha$ and the periodic orbits $\ga^\pr_\mi$ and $\ga^\pr_\ma$ have uniformly bounded action. Therefore, these periodic orbits must come from bifurcations of periodic orbits $\gamma_\mi$ and $\gamma_\ma$ of $\balpha$ such that $\czl(\ga_\mi) \leq -n$ and $\czu(\ga_\ma) \geq n$.

\subsection{Proof of Theorem \ref{Thm 2}}

Following \cite{Lon02,Lon99}, one can associate to $\ga$ Bott's index function $\Ga: S^1 \to \Z$ which satisfies the following properties:
\begin{itemize}
\item[(a)] $\czl(\ga^k;\Psi^k) = \sum_{z \in S^1;\ z^k=1} \Ga(z)$ for every $k \in \N$ (Bott's formula).
\item[(b)] The discontinuity points of $\Ga$ are contained in $\sigma(P_\ga) \cap S^1$, where $P_\ga$ is the linearized Poincar\'e map of $\ga$ and $\sigma(P_\ga)$ is its spectrum.
\item[(c)] $\Ga(z)=\Ga(\bar z)$ for every $z \in S^1$.
\item[(d)] The \emph{splitting numbers} $S^\pm(z) := \lim_{\ep\to 0^+} \Ga(e^{\pm i\ep}z)-\Ga(z)$ satisfy
\begin{itemize}
\item[(d1)] $0 \leq S^\pm(z) \leq \nu(z)$ for every $z \in \sigma(P_\ga) \cap S^1$, where $\nu(z)$ is the geometric multiplicity of $z$;
\item[(d2)] $S^+(z) + S^-(z) \leq m(z)$ for every $z \in \sigma(P_\ga) \cap S^1$, where $m(z)$ is the algebraic multiplicity of $z$;
\item[(d3)] $S^\pm(z)=S^\mp(\bar z)$ for every $z \in S^1$.
\end{itemize}
\end{itemize}
For a definition of $\Ga$ and a proof of these properties we refer to \cite{Lon02,Lon99}. Now, suppose that $\czl(\ga;\Psi) \leq -n$ and $\czl(\ga^j;\Psi^j) \geq -n$ for some $j>1$. By Bott's formula,
\begin{align*}
\czl(\ga^j;\Psi^j) & = \sum_{z^j=1} \Ga(z) = \Ga(1) + \sum_{z^j=1; z \neq 1} \Ga(z) \\
& = \czl(\ga;\Psi) + \sum_{z^j=1; z \neq 1} \Ga(z) \geq -n.
\end{align*}
But $\czl(\ga;\Psi) \leq -n$ and consequently $\sum_{z^j=1; z \neq 1} \Ga(z) \geq 0$. Together with property (c), this implies that there exists $z=e^{i\theta} \in S^1$,
with $0 < \theta \leq \pi$ such that
\begin{equation}
\label{index jump}
\Ga(z)-\Ga(1) \geq n.
\end{equation}

Let $\{e^{i\theta_1},\dots,e^{i\theta_l}\}$ be the eigenvalues of $P_\ga$ with modulus one and argument $0 < \theta_k < \theta$, 
such that $\theta_k < \theta_{k+1}$ for every $1\leq k \leq l-1$. The properties of $\Ga$ imply that
\[
\Gamma (z) - \Gamma(1) = (S^+ (1) - S^- (1)) + \sum_{k=1}^l (S^+ (e^{i\theta_k}) - S^- (e^{i\theta_k}))\,.
\]
Since the splitting numbers are non-negative (by property (d1)), it follows from \eqref{index jump} that
\[
S^+ (1) + \sum_{k=1}^l S^{+}(e^{i\theta_k}) \geq n,
\]
By property (d3), this implies that
\[
S^- (1) + \sum_{k=1}^l S^{-}(e^{-i\theta_k}) \geq n.
\]
Putting these two inequalities together we arrive at
\begin{align*}
2n \geq \sum_{z\in S^1} m(z) & \geq \sum_{z \in S^1} S^+(z)+S^-(z) \\
& \geq S^+ (1) + S^- (1) + \sum_{k=1}^l (S^{+}(e^{i\theta_k}) + S^{-}(e^{-i\theta_k})) \geq 2n,
\end{align*}
where the first inequality is trivial by a dimensional reason, the second inequality holds by property (d2) and the third inequality 
follows again from the fact that the splitting numbers are non-negative. This ensures that all of the above inequalities are in fact
equalities. Hence, $\ga$ is elliptic. Moreover, $\sum_{z^j=1; z \neq 1} \Ga(z)$ must vanish and therefore 
$\czl(\ga;\Psi)=\czl(\ga^j;\Psi^j)=-n$.

\begin{remark}
\label{rmk:strictly elliptic - proof}
The previous argument actually shows that $\ga$ is strictly elliptic. More precisely, let $J_0$ be the multiplication by $i$ in $\R^{2n} \simeq \C^n$ and $\om_0$ the canonical symplectic form on $\R^{2n}$. Consider their extensions $J$ and $\om$ to $\C^{2n} = \R^{2n} \otimes \C$ by complex linearity. The Krein form is defined as $\beta(v,w)=-i\om(v,\bar w)$. It turns out that $\beta$ is a non-degenerate Hermitian symmetric form; see \cite{Eke,Lon02}. Given $z \in \sigma(P_\ga) \cap S^1$ denote by $E_z \subset \C^{2n}$ the generalized eigenspace of $z$. The Krein type numbers $p_z$ and $q_z$ of $z$ are defined as the number of positive and negative eigenvalues of $\beta|_{E_z}$ respectively. We say that $\ga$ is strictly elliptic if $\beta|_{E_z}$ is definite for every $z \in \sigma(P_\ga) \cap (S^1\setminus\{1\})$. It is well known (c.f. \cite[Theorem 9.1.7]{Lon02}) that $p_z \geq S^+(z)$ and $q_z \geq S^-(z)$ for every $z \in \sigma(P_\ga) \cap S^1$. Thus, the previous analysis shows that
\[
2n \geq \sum_{z \in S^1} p_z + q_z \geq \sum_{z \in S^1} S^+(z)+S^-(z) \geq S^+ (1) + S^-(1) + \sum_{k=1}^l (S^{+}(e^{i\theta_k}) + S^{-}(e^{-i\theta_k})) \geq  2n.
\]
Since all these numbers are non-negative, we conclude that $p_z = S^+(z)$ and $q_z = S^-(z)$ for every $z \in \sigma(P_\ga) \cap S^1$. Moreover, $S^+(z)=0$ for every $z \notin \{1,e^{i\theta_1},\dots,e^{i\theta_l}\}$ and $S^-(z)=0$ for every $z \notin \{1,e^{-i\theta_1},\dots,e^{-i\theta_l}\}$. Therefore, we conclude that $\ga$ must be strictly elliptic.
\end{remark}

In the case that $\czu(\ga;\Psi) \geq n$ and $\czu(\ga^j;\Psi^j) \leq n$, consider the inverted Reeb flow of $\alpha$,  $-\ga(t):=\ga(-t)$ and the trivialization of $(\xi,d\alpha)$ along $-\ga$ given by $\Psi_{-t}$. By \eqref{eq:czl x czu},
\[
\czl(-\ga;\Psi) = -\czu(\ga;\Psi).
\]
Then the previous argument applies {\it mutatis mutandis} and shows that $-\ga$ is (strictly) elliptic, which is equivalent to saying that $\ga$ is (strictly) elliptic, 
and that $\czl(-\ga;\Psi)=\czl(-\ga^j;\Psi^j) = -n$, which is equivalent to saying that $\czu(\ga;\Psi)=\czu(\ga^j;\Psi^j)=n$.

\section{Proofs of the applications} 
\label{sec:proof applications}

Before we present the proofs of the applications, let us discuss some useful facts about trivializations of the contact structure along closed geodesics in Zoll manifolds that will be important in the next subsections.

\subsection{Trivializations of closed geodesics in Zoll manifolds}
\label{sec:trivializations}

Let $N^{n+1}$ be a closed Zoll manifold, that is, a closed Riemannian manifold all of whose geodesics are closed with the same minimal period. (Equivalently, the geodesic flow generates a free circle action in the unit sphere bundle $SN$.) Let $\pi: TN \to N$ and $\tau: SN \to SN/S^1$ be the corresponding projections. Let $\bar g$ be the metric on $N$ all of whose geodesics are closed and $g$ any metric on $N$. Given a closed geodesic $\ga$ of $g$ freely homotopic to a fiber of $S^1 \to SN \to SN/S^1$ (given by a simple closed geodesic of $\bar g$) we have two natural trivializations of the contact structure $\xi$ along $\ga$. The first one is the trivialization $\Psi$ induced by a trivialization on a fiber $\bar\ga$ over $p:=\tau(\bar\gamma) \in SN/S^1$ given by a constant symplectic frame of $T_p(SN/S^1)$. It is easy to see that $\czl(\bar\ga;\Psi)=-n$.

To define the second trivialization, suppose initially that $N$ is orientable. Consider the trivialization $\Phi$ that sends the intersection of the vertical distribution of $TTN$ with $\xi$ to a fixed Lagrangian subspace of $\R^{2n}$ (note that, since $N$ is orientable, the normal bundle of $x:=\pi\circ\ga$ is trivial). This trivialization is unique up to homotopy, see \cite[Lemma 1.2]{AS} (more precisely, the argument in \cite{AS} is for a trivialization of $\ga^*TT^*N$ but it can be readily adapted to a trivialization of the contact structure $\xi$). It turns out that $\czl(\ga;\Phi)$ coincides with the Morse index of $\ga$, see \cite{CF, Lon02, Web}. Since $\czl(\bar\ga;\Psi)=-n$, this implies that the difference of the indexes with respect to the trivializations $\Phi$ and $\Psi$ is given by
\begin{equation}
\label{eq:trivializations1}
\czl(\ga;\Phi) - \czl(\ga;\Psi) = \morse(\bar\ga) + n.
\end{equation}
If $N$ is not orientable, the normal bundle of $x$ does not need to be trivial and, following \cite{Web}, we have to suitably modify the construction of $\Phi$ in order to produce a trivialization $\Phi^\pr$ of $\xi$ over $\ga$ so that
\[
\czl(\ga;\Phi^\pr) = \morse(\ga) + 1,
\]
see \cite{Web} for details. (The relation above was proved in \cite{Web} for non-degenerate closed geodesics but it can be extended to degenerate closed geodesics, see \cite{CF}.) Thus,
\begin{equation}
\label{eq:trivializations2}
\czl(\ga;\Phi^\pr) - \czl(\ga;\Psi) = \morse(\bar\ga) + n + 1.
\end{equation}

\subsection{Proof of Theorem \ref{thm:elliptic geodesic S^2}}
\label{sec:proof elliptic geodesic S^2}

Recall that $F$ is a Finsler metric on $S^2$, with reversibility $r$ and flag curvature $K$ satisfying $(r/(r+1))^2 \leq K \leq 1$. As explained in Section \ref{sec:results}, if we have $(r/(r+1))^2 < K \leq 1$, then the lift of the geodesic flow of $F$ to $S^3$ is dynamically convex and $\Z_2$-invariant. Thus, we conclude from Theorems \ref{Thm 1} and \ref{Thm 2} that $F$ carries an elliptic closed geodesic $\ga$ such that $\ga$ is not contractible in $SS^2$ and $\czl(\ga;\Psi)=-1$, where the index is computed using the trivialization $\Psi$ discussed in Section \ref{sec:trivializations}. Since a prime closed geodesic of the round metric on $S^2$ has Morse index equal to one, we conclude from \eqref{eq:trivializations1} that $\morse(\ga)=\czl(\ga;\Phi)=1$.

Thus, it remains only to deal with the case where the inequality $(r/(r+1))^2 \leq K$ is not strict. In order to do it, we will need the following general result. Let $F$ be a Finsler metric on a closed manifold $N$ and denote by $N_0$ the zero section in $TN$. Given a point $(x,v) \in TN\setminus N_0$ and a plane $\sigma \subset T_xM$ denote by $K(\sigma,v)$ the corresponding flag curvature. Let $SN := F^{-1}(1)$ be the unit sphere bundle and consider the continuous functions $K_-: SN \to \R$ and $K_+: SN \to \R$ defined as $K_-(x,v) = \min_{\sigma \subset T_xN} K(\sigma,v)$ and $K_+(x,v) = \max_{\sigma \subset T_xN} K(\sigma,v)$, where $\sigma$ runs over all the planes in $T_xN$. Define $K_\mi = \min_{(x,v) \in SN} K_-(x,v)$ and  $K_\ma = \max_{(x,v) \in SN} K_+(x,v)$.

\begin{proposition}
\label{prop:perturbation}
Suppose that $K_\mi < K_\ma$ and $K_-(x,v)/K_+(x,v) > K_\mi/K_\ma$ for every $(x,v) \in SN$. Then $F$ can be $C^\infty$-perturbed to a Finsler metric $F^\pr$ whose flag curvature $K^\pr$ satisfies $K_\mi < K^\pr(\sigma,v) \leq K_\ma$ for every $(x,v) \in SN$ and $\sigma \subset T_xN$. In particular, such perturbation is always possible if $N$ has dimension two.
\end{proposition}

\begin{proof}
Let $f$ be a smooth function $C^0$-arbitrarily close to $K_-$ such that $\min f = K_\mi$ and define $S=\{(x,v) \in SN;\ K_-(x,v)=K_\mi\}$. Fix $\ep>0$ such that $K_\mi+\ep$ is a regular value of $f$ and $U_\ep:=\{(x,v) \in SN;\ f(x,v) < K_\mi+\ep\}$ is a neighborhood of $S$ satisfying $K_-(x,v) > K_\mi+2\delta$ for every $(x,v) \notin U$ and some $\delta>0$. Taking $\ep$ sufficiently small, we can assume, by our hypothesis, that $K_+(x,v) < K_\ma-\delta'$ for every $(x,v) \in U$ and some $\delta'>0$. Take $c<1$ such that $K_\mi < (1/c^2)K_-(x,v) \leq (1/c^2)K_+(x,v) < K_\ma$ for every $(x,v) \in U$. Let $\ep'<\ep$ such that $f$ has no critical values in $[K_\mi+\ep',K_\mi+\ep]$ and $K_-(x,v) > K_\mi + \delta$ for every $(x,v) \notin U_{\ep'}$. Consider a bump function $\rho: \R \to \R$ such that $\rho(r)=c$ for every $r\leq K_\mi+\ep'$ and $\rho(r)=1$ for every $r\geq K_\mi+\ep$. Notice that $c$ can be taken arbitrarily close to $1$ independently of the choice of $\ep'$. Consequently, $\rho$ can be chosen $C^\infty$-close to the constant function equal to one. Now it is clear from the choices of the constants that $F^\pr:=\rho(f)F$ is the desired metric.
\end{proof}

By this proposition, if our metric $F$ on $S^2$ does not satisfy the strict inequality $(r/(r+1))^2 < K \leq 1$ then we can perturb this to a metric $F^\pr$ whose flag curvature $K^\pr$ satisfies $(r/(r+1))^2 < K^\pr \leq 1$. Now, it is clear from the proof of Theorem \ref{Thm 1} that the period of the elliptic closed geodesic $\ga^\pr$ of $F^\pr$ is uniformly bounded from above. Passing to the limit we get an elliptic closed geodesic $\ga$ for $F$ such that $\czl(\ga;\Phi) \leq -1$.

Finally, we claim that if $(r/(r+1))^2 < K \leq 1$ then $\ga$ must be prime. As a matter of fact, arguing indirectly, suppose that $\ga=\psi^k$ for some prime closed geodesic $\psi$ and $k\geq 2$. Since $\ga$ is not contractible, we have that $k\geq 3$. Note that $\morse(\psi) \leq \morse(\ga) = 1$ (given a closed geodesic $c$ we have the relation $\morse(c^{k})\geq \morse(c)$ for every $k \in \N$). By \eqref{eq:mean index} we conclude that $k\Delta(\psi;\Phi)=\Delta(\ga;\Phi)\leq 2 \implies \Delta(\psi^2;\Phi^2) \leq 4/3 \implies \morse(\psi^2)<3$. But $\psi^2$ is contractible and therefore admits a lift to $S^3$ with Conley-Zehnder index less than $3$, contradicting the dynamical convexity. (Note here that a trivialization given by a capping disk coincides with $\Phi^2$ up to homotopy.)

\subsection{Proof of Theorem \ref{thm:elliptic geodesic RP^n}}

First of all, note that both $SS^{n+1}$ and $S\RP^{n+1}$ are Boothby-Wang contact manifolds with the circle action given by the geodesic flow of the Riemannian metric with constant curvature. Moreover, the induced $\Z_2$-action on $SS^{n+1}$ is generated by the lift of the antipodal map $\psi: S^{n+1} \to S^{n+1}$ to $SS^{n+1}$ given by $(x,v) \mapsto (\psi(x),d\psi(x)v)$. Consequently, $SS^{n+1}/\Z_2$ coincides with $S\RP^{n+1}$. It is easy to see that $SS^{n+1}/S^1=S\RP^{n+1}/S^1$ is diffeomorphic to the Grassmannian of oriented two-planes $G^+_2(\R^{n+2})$ (each oriented great circle is identified with the oriented two-plane in $\R^{n+2}$ that contains it; see \cite{Bes}).

The Morse index of the simple closed geodesic of the metric with constant curvature in $\RP^{n+1}$ vanishes. Thus, by the discussion in Section \ref{sec:trivializations}, it is enough to show that if $(r/(r+1))^2 < K \leq 1$ then $F$ has an elliptic closed geodesic $\ga$ such that $\czl(\ga;\Psi)=-n$. Indeed, by our assumption and Proposition \ref{prop:perturbation}, we can argue as in the proof of Theorem \ref{thm:elliptic geodesic S^2} and conclude that if $(r/(r+1))^2 \leq K \leq 1$ then $F$ has an elliptic closed geodesic $\ga$ such that $\czl(\ga;\Psi)\leq -n$. Clearly, $\ga$ can be chosen prime since $\morse(\ga)=0$ and $\morse(c^{k})\geq \morse(c)$ for every closed geodesic $c$ and $k \in \N$.

To prove the existence of $\ga$ when $(r/(r+1))^2 < K \leq 1$, suppose firstly that $n>1$ and consider the lift $\widetilde F$ of $F$ to $S^{n+1}$. By \cite[Theorem 1]{Rad04} the pinching condition $(r/(r+1))^2 < K \leq 1$ implies that the length of every closed geodesic $c$ of $\widetilde F$ satisfies $L(c) \geq \pi(1+1/r)$.  Using this and \cite[Lemma 3]{Rad04} we conclude that the Morse index of $c$ is bigger than or equal to $n$. By the discussion in Section \ref{sec:trivializations}, this means that $\czl(c;\Phi) \geq n$. On the other hand, it is easy to see that the (prime) geodesics of the round metric on $S^{n+1}$ have Robbin-Salamon index equal to $2n$ using the trivialization $\Phi$. Therefore, the pinching condition implies that the geodesic flow of $\widetilde F$ is dynamically convex (note here that, since $n>1$, $SS^{n+1}$ is simply connected). When $n=1$, consider the lift of the geodesic flow of $F$ to a $\Z_4$-invariant Reeb flow of a contact form $\alpha$ on $S^3$. By the discussion in the proof of Theorem \ref{thm:elliptic geodesic S^2}, $\alpha$ is dynamically convex.

Thus, in order to apply Theorems \ref{Thm 1} and \ref{Thm 2}  to obtain the desired closed geodesic $\ga$ of $F$, it remains only to show that $G^+_2(\R^{n+2})$ satisfies the hypothesis (H1) in Theorem A, namely, that $G^+_2(\R^{n+2})$ admits an even Morse function and is monotone with constant of monotonicity $\lambda>1$. This Grassmannian is diffeomorphic to the quadric
\[
Q_n = \{[z_0:\cdots:z_{n+1}] \in \CP^{n+1} \mid \sum_{j=0}^{n+1} z_j^2 = 0\}.
\]
The identification goes as follows: an oriented two-plane in $\R^{n+2}$ can be described by an orthonormal basis $(x,y)$ which can be mapped into $\CP^{n+1}$ via the map $(x,y) \mapsto [x+iy]$ defined on the Stiefel manifold $V_2(\R^{n+2})$, where $[\cdot]$ denotes the classes in $\CP^{n+1}$ of non-zero vectors in $\C^{n+2}$. The equations $\|x\|^2=\|y\|^2$ and $x\cdot y=0$ read as $\sum_{j=0}^{n+1} (x_j + iy_j)^2=0$. One can check that this map descends to the quotient and defines a diffeomorphism between $G^+_2(\R^{n+2})$ and $Q_n$. Moreover, the pullback $\om$ of the canonical symplectic form on $\CP^{n+1}$ to $Q_n$ is a representative of the Euler class of the circle bundle $S^1 \to SS^{n+1} \to Q_n$. Thus, $SS^{n+1}$ is the prequantization of $Q_n$ with the symplectic form $\om$.

Let us show that $Q_n$ is monotone. Consider first the cases $n=1$ and $n=2$. We have that $Q_1 \simeq S^2$ and $Q_2 \simeq S^2 \times S^2$ which one can check that are monotone with constant of monotonicity $\lambda=2$. Thus, suppose that $n\geq 3$. A computation shows that $H^2(Q_n;\Z) \simeq H_2(Q_n;\Z) \simeq \pi_2(Q_n) \simeq \Z$. Moreover, $[\om]$ is a generator of $H^2(Q_n;\Z)$. The first Chern class is given by
\[
c_1(TQ_n) = n[\om],
\]
see \cite[Theorem 1.5]{Li}, \cite{Aud} and \cite[pages 429-430]{MS}. Consequently, $Q_n$ is monotone with constant of monotonicity bigger than one.

To prove that $G^+_2(\R^{n+2})$ admits an even Morse function, we proceed by induction on $n$. Suppose that $G^+_2(\R^{k+2})$ admits an even Morse function for every $k\leq n$. It is clearly true for $n=2$ since $G^+_2(\R^{3}) \simeq S^2$ and $G^+_2(\R^{4}) \simeq S^2 \times S^2$. So suppose that $n\geq 3$. It is well known that $G^+_2(\R^{n+2})$ admits a Hamiltonian circle action whose momentum map $f$ is a perfect Morse-Bott function such that its critical set has three connected components given by
\[
X = \PP(z,0,\dots,0) = \text{pt},\ \ Y = G^+_2(0\times \R^{n}),\ \ Z = \PP(z,0,\dots,0) = \text{pt},
\]
where $X$ and $Z$ correspond to the two orientations on the real two-plane $(z,0,\dots,0)$, see \cite[Example 1.2]{Li}. The points $X$ and $Z$ correspond to the minimum and maximum of $f$ respectively. Moreover, since $H_1(G^+_2(\R^{n+2});\Z) = 0$, $H_2(G^+_2(\R^{n+2});\Z) \neq 0$ and $f$ is perfect, we have that the index of $Y$ (given by the index of the Hessian of $f$ restricted to the normal fibers of $Y$) is equal to two.

By our induction hypothesis, $Y$ admits an even Morse function $g$. Fix a Riemannian metric on $G^+_2(\R^{n+2})$ and consider a tubular neighborhood $V=\exp(N_\ep Y)$ of $Y$, where $N_\ep Y$ is the normal bundle of $Y$ with radius $\ep$. Let $\pi: N_\ep Y \to Y$ be the bundle projection and $\beta: [0,\ep] \to \R$ be a bump function such that $\beta(r)=1$ if $0 \leq r \leq \ep/3$, $\beta(r)=0$ if $2\ep/3 \leq r \leq \ep$ and $\beta^\pr(r)\leq 0$ for every $r$. Consider the function $h: V \to \R$ given by
\[
h(x) = \beta(r(\exp^{-1}(x)))g(\pi(\exp^{-1}(x))),
\]
where $r(\exp^{-1}(x))$ is the radius of $\exp^{-1}(x)$. By construction, $h$ extends to a smooth function defined on $G^+_2(\R^{n+2})$ which vanishes identically outside $V$. Define $\tilde f: G^+_2(\R^{n+2}) \to \R$ as
\[
\tilde f(x) = f(x) + \delta h(x),
\]
where $\delta>0$ is a constant. One can check that choosing $\delta$ sufficiently small then $\tilde f$ is Morse and its critical points are given by $X$, $Z$ and the critical points of $g$ on $Y$. Moreover, the indexes of these last points equal the indexes of the critical points of $g$ plus the index of $Y$, which is equal to two. Consequently, $\tilde f$ is an even Morse function, as desired.

\subsection{Proof of Theorem \ref{thm:elliptic geodesic S^n}}

By Theorem \ref{thm:elliptic geodesic S^2}, we can suppose that $n>1$. The Morse index of the simple closed geodesic of the round metric on $S^{n+1}$ is $n$. Therefore, by the discussion in Section \ref{sec:trivializations}, we have to show that if $\frac{9}{4}(r/(r+1))^2 < K \leq 1$ then $F$ has an elliptic closed geodesic $\ga$ such that $\czl(\ga;\Psi) = -n$. The case that the pinching condition is not strict can be dealt with using a perturbation argument as in the previous sections.

By \cite[Theorem 1 and Lemma 3]{Rad04}, the condition $\frac{9}{4}(r/(r+1))^2 < K \leq 1$ implies that every closed geodesic $c$ of $F$ satisfies $\czl(c;\Phi) \geq n$ and $\czl(c^2;\Phi) \geq 3n$, where $\Phi$ is the trivialization discussed in Section \ref{sec:trivializations}. The Robbin-Salamon index of the (prime) geodesics of the round metric on $S^{n+1}$ is equal to $2n$ with respect to this trivialization. Moreover, as explained in the previous section, $SS^{n+1}$ satisfies the hypothesis (H1) of Theorem A. Thus, we can apply Theorems \ref{Thm 1} and \ref{Thm 2} and conclude that $F$ carries an elliptic closed geodesic $\ga$ satisfying $\czl(\ga;\Psi) = -n$. Moreover, $\ga$ can be chosen prime: if $\ga=\psi^k$ for some prime closed geodesic $\psi$ then $\psi$ is elliptic and $\morse(\psi)=n$. Indeed, since $\ga=\psi^k$ we have that $\morse(\psi) \leq \morse(\ga)=n$. On the other hand, $\morse(\psi)\geq n$ by dynamical convexity.

\subsection{Proof of Theorem \ref{thm:magnetic}}
\label{sec:proof magnetic}

Consider pairs $(g,\Om)$ given by a Riemannian metric $g$ and a non-degenerate magnetic field $\Om$ on $N$. As explained in Section \ref{sec:results}, the magnetic flow generated by $(g,\Om)$ is the Hamiltonian flow of $H_g(x,p):=(1/2)\|p\|^2$ with respect to the twisted symplectic form $\om=\om_0+\pi^*\Om$. One can check that if $g$ has constant sectional curvature and $\Om$ is the corresponding area form (normalized to have total area equal to one) then $H_g^{-1}(k)$ is of contact type and the magnetic flow restricted to $H_g^{-1}(k)$ generates a free circle action such that $H_g^{-1}(k)/S^1$ is diffeomorphic to $N$ for every $0<k<1/2$ (see, for instance, \cite{Mac}). Denote by $\xi_0$ the corresponding contact structure. 

It is well known that given a pair $(g,\Om)$ there exists $\ep>0$ such that the energy level $H_g^{-1}(k)$ is of contact type for every $k<\ep$ \cite{Ben,GGM}. Fix $S:=H^{-1}(k)$ and let $\xi$ be the contact structure on $S$. We claim that $\xi$ is equivalent to $\xi_0$. As a matter of fact, let $(g_s,\Om_s)$ be a smooth family of pairs given by Riemannian metrics and symplectic forms on $N$, with $0\leq s\leq 1$, such that $g_0$ has constant sectional curvature, $\Om_0$ is the area form (corresponding to $g_0$) and $(g_1,\Om_1)=(g,\Om)$. (The existence of this family follows from the fact that every Riemannian metric on a connected surface is conformally equivalent to a metric with constant sectional curvature.) Define
\[
\ep(s)=\sup\{k_0 \in \R \mid H_{g_s}^{-1}(k)\text{ is of contact type for every }k \in (0,k_0)\}.
\]
By the aforementioned result, $\ep(s)>0$ for every $s \in [0,1]$ and one can check that  $\ep_0:=(1/2)\inf_{s\in [0,1]} \ep(s)$ is positive. Thus, it follows from Gray's stability theorem that the contact structure on $H_g^{-1}(\ep_0)$ is equivalent to the contact structure on $H_{g_0}^{-1}(\ep_0)$. Applying Gray's theorem again we conclude that the contact structure on $H_g^{-1}(\ep_0)$ is equivalent to the contact structure on $S$. Therefore, $S$ is Boothby-Wang and we will use this circle bundle structure in what follows.

Let us split the proof in the cases that $\g=0$ and $\g>1$. In the first case, it was proved in \cite{Ben} that if $k$ is sufficiently small then every periodic orbit $\ga$ of the lift of the magnetic flow to $S^3$ satisfies $\czl(\ga) \geq 3$ (with a trivialization given by a capping disk). Moreover, the lifted contact structure is tight. Thus, the lifted contact form is positively dynamically convex and we can apply Theorems \ref{Thm 1} and \ref{Thm 2} to conclude that the magnetic flow carries an elliptic closed orbit $\ga$ such that $\ga$ is not contractible in $S$ and $\czl(\ga;\Psi)=-1$, where the index is computed using a trivialization $\Psi$ as in Section \ref{sec:trivializations}. The contact structure $\xi$ on $S$ admits a global trivialization $\Phi$ (see \cite[Lemma 6.7]{Ben}) and an easy computation shows that
\[
\czl(\ga;\Phi) = \czl(\ga;\Psi) + 2 = 1.
\]
We claim that $\ga$ is prime. As a matter of fact, arguing by contradiction, suppose that $\ga=\psi^k$ for some $k\geq 2$, where $\psi$ is a prime periodic orbit. Since $\ga$ is not contractible, we have that $k\geq 3$. The fact that $\czl(\psi^2) \geq 3$ implies, by \eqref{eq:mean index}, that the mean index satisfies $\Delta(\psi^2) \geq 2 \implies \Delta(\psi) \geq 1 \implies \Delta(\ga) \geq 3$. But this contradicts the fact that $\czl(\ga)=1$.

Now, let us consider the case that $\g>1$. Denote by $\beta$ the contact form on $S$ that generates our free circle action and let $\vr$ be a simple orbit of $\beta$. As before, $\xi$ admits a global trivialization $\Phi$ and a simple computation shows that $\rs(\vr;\Phi)=-2$. One can check that $[\vr^l] \neq 0$ for every $l \in \N$, $(\vr^l) \neq 0$ for every $1 \leq l < |\chi(N)|$ and $(\vr^{|\chi(N)|})=0$, where $[\vr] \in \pi_1(S)$ and $(\vr) \in H_1(S,\Z)$ denote the homotopy and homology classes respectively. Moreover, there is a $|\chi(N)|$-covering $\tau: \wtl S \to S$ such that $\wtl\beta = \tau^*\beta$ generates a free circle action and the deck transformations are given by the induced action of $\Z_{|\chi(N)|} \subset S^1$. Let $\wtl\vr$ be a (simple) orbit of the Reeb flow of  $\wtl\beta$ and $\wtl a$ its free homotopy class. The contact structure $\wtl\xi$ of $\wtl\beta$ admits a global trivialization $\wtl\Phi$ (given by the pullback of $\Phi$) and $\rs(\wtl\vr;\Phi)=2\chi(N)$.

Let $\alpha$ be the contact form on $S$ whose Reeb flow is the magnetic flow. It follows from the computations in \cite{Ben} that
\begin{equation}
\label{eq:gabriele1}
\czu(\ga;\Phi) \leq 2\chi(N)+1
\end{equation}
for every periodic orbit $\ga$ of $\alpha$ homologous to zero. Indeed, it is proved in \cite[Equation (1) and Remark 6.10]{Ben} that every periodic orbit $\ga$ of the magnetic flow homologous to zero satisfies 
\begin{equation}
\label{eq:gabriele2}
\czl(\ga;\Phi) \leq 2\chi(N)+1.
\end{equation}
It turns out that if we take the inverted magnetic flow $-R_\alpha$ then the same argument yields the inequality $\czl(-\ga;\Phi) \geq 2|\chi(N)|-1$. Indeed, inequality \eqref{eq:gabriele2} is proved showing that the winding interval $I(\ga)$ associated to the linearized flow along $\ga$ satisfies $I(\ga) \subset (-\infty,\chi(N)+1)$; see \cite{Ben} for details. On the other hand, when we take the inverted magnetic flow (with the same trivialization of the contact structure) the corresponding winding interval $I(-\ga)$ satisfies $I(-\ga)=-I(\ga) \subset (|\chi(N)|-1,+\infty)$ implying that $\czl(-\ga;\Phi) \geq 2|\chi(N)|-1$. Consequently, inequality \eqref{eq:gabriele1} follows from \eqref{eq:czl x czu}.

Define $\wtl\alpha=\tau^*\alpha$. By the previous discussion, every periodic orbit $\wtl\ga$ of $\wtl\alpha$ freely homotopic to $\wtl\vr$ has a projection $\ga$ on $S$ homologous to zero. Therefore, $\czu(\wtl\ga)=\czu(\ga) \leq 2\chi(N)+1$ and we conclude that $\wtl\alpha$ is negatively $\wtl a$-dynamically convex. Thus, we can apply Theorems \ref{Thm 1} and \ref{Thm 2} to conclude that $\alpha$ carries an elliptic closed geodesic $\ga$ with free homotopy class $a$ such that $\czu(\ga;\Psi)=1$, where the index is computed using the trivialization $\Psi$ of the contact structure $\xi$ over $\ga$ induced by a trivialization on a fiber of $S$ over $p \in N$ given by a constant symplectic frame of $T_pN$, as discussed in Section \ref{sec:trivializations}. Now, a computation gives the relation
\[
\czu(\ga;\Phi) = \czu(\ga;\Psi) - 2 = -1.
\]
Finally, notice that $\ga$ must be prime because every closed curve with free homotopy class $a$ is simple (observe that $\pi_1(N)$ is torsion free).

\section{Good toric contact manifolds} 
\label{sec:toric}

In this section we provide the necessary information on toric contact manifolds, briefly introduced
in Section~\ref{sec:results} for the statement of Theorem~\ref{thm:convex}. As in that theorem,
we will restrict ourselves to good toric contact manifolds, i.e. those that are determined by a
strictly convex moment cone. For further details we refer the interested reader to~\cite{Le1}
and~\cite{AM}.

\subsection{Toric symplectic cones}

Via symplectization, there is a $1$-$1$ correspondence between co-oriented contact manifolds
and symplectic cones, i.e. triples $(W,\om,X)$ where $(W,\om)$ is a connected symplectic manifold
and $X$ is a vector field, the Liouville vector field, generating a proper $\R$-action
$\rho_t:W\to W$, $t\in\R$, such that $\rho_t^\ast (\om) = e^{t} \om$. A closed symplectic cone is a 
symplectic cone $(W,\om,X)$ for which the corresponding contact manifold $M = W/\R$ is closed.

A toric contact manifold is a contact manifold of dimension $2n+1$ equipped with an effective Hamiltonian
action of the standard torus of dimension $n+1$: $\T^{n+1} = \R^{n+1} / 2\pi\Z^{n+1}$. Also via symplectization,
toric contact manifolds are in $1$-$1$ correspondence with toric symplectic cones, i.e. symplectic cones
$(W,\om,X)$ of dimension $2(n+1)$ equipped with an effective $X$-preserving Hamiltonian $\T^{n+1}$-action,
with moment map $\mu : W \to \R^{n+1}$ such that $\mu (\rho_t (w)) = e^{t} \mu (w)$, for all $w\in W$ and $t\in\R$.
Its moment cone is defined to be $C:= \mu(W) \cup \{ 0\} \subset \R^{n+1}$.

A toric contact manifold is {\it good} if its toric symplectic cone has a moment cone with the following properties.
\begin{definition} \label{def:good}
A cone $C\subset\R^{n+1}$ is \emph{good} if it is strictly convex and there exists a minimal set 
of primitive vectors $\nu_1, \ldots, \nu_d \in \Z^{n+1}$, with 
$d\geq n+1$, such that
\begin{itemize}
\item[(i)] $C = \bigcap_{j=1}^d \{x\in\R^{n+1}\mid 
\ell_j (x) := \langle x, \nu_j \rangle \geq 0\}$.
\item[(ii)] Any codimension-$k$ face of $C$, $1\leq k\leq n$, 
is the intersection of exactly $k$ facets whose set of normals can be 
completed to an integral base of $\Z^{n+1}$.
\end{itemize}
\end{definition}
The analogue for good toric contact manifolds of Delzant's classification theorem for closed toric
symplectic manifolds is the following (see~\cite{Le1})
\begin{theorem} \label{thm:good}
For each good cone $C\subset\R^{n+1}$ there exists a unique closed toric symplectic cone
$(W_C, \om_C, X_C, \mu_C)$ with moment cone $C$.
\end{theorem}

One source for examples of good toric contact manifolds is the Boothby-Wang (prequantization)
construction over integral closed toric symplectic manifolds. The corresponding good cones have 
the form
\[
C:= \left\{z(x,1)\in\R^{n}\times\R\mid x\in P\,,\ z\geq 0\right\}
\subset\R^{n+1}
\]
where $P\subset\R^n$ is a Delzant polytope with vertices in the integer lattice $\Z^n\subset\R^n$.

\subsection{Toric contact forms and Reeb vectors}

Let $(W,\omega, X)$ be a good toric symplectic cone of dimension $2(n+1)$, with
corresponding closed toric contact manifold $(M,\xi)$. Denote by $\Xx_X (W, \omega)$
the set of $X$-preserving symplectic vector fields on $W$ and by $\Xx (M,\xi)$
the corresponding set of contact vector fields on $M$. The $\T^{n+1}$-action
associates to every vector $\nu \in \R^{n+1}$ a vector field
$R_\nu \in \Xx_X (W,\omega) \cong \Xx (M, \xi)$. We will say that a
contact form $\alpha_\nu \in \Omega^1 (M,\xi)$ is \emph{toric} if
its Reeb vector field $R_{\alpha_\nu}$ satisfies
\[
R_{\alpha_\nu} = R_\nu \quad\text{for some $\nu\in\R^{n+1}$.}
\]
In this case we will say that $\nu\in\R^{n+1}$ is a \emph{Reeb vector}
and that $R_\nu$ is a \emph{toric Reeb vector field}.
The following proposition characterizes which $\nu\in\R^{n+1}$ are Reeb 
vectors of a toric contact form on $(M,\xi)$.

\begin{prop}[{\cite{MSY} or \cite[Proposition 2.19]{AM}}] \label{prop:sasaki}
Let $\nu_1, \ldots, \nu_d \in \R^{n+1}$ be the defining integral normals
of the moment cone $C\in\R^{n+1}$ associated with $(W,\omega,X)$ and 
$(M,\xi)$. The vector field $R_\nu \in \Xx_X (W,\omega) \cong \Xx(M,\xi)$
is the Reeb vector field of a toric contact form 
$\alpha_\nu \in \Omega^1 (M,\xi)$ if and only if
\[
\nu = \sum_{j=1}^d a_j \nu_j \quad\text{with $a_j\in\R^+$ for all
$j=1, \ldots, d$.}
\]
\end{prop}

\subsection{Explicit models and fundamental group}
\label{ssection:models}

The existence part of Theorem~\ref{thm:good} is given by an explicit symplectic reduction
construction, which we now briefly describe (as we did in~\cite{AM}). For complete details
the interested reader may look, for example, at~\cite{Le1}.

Let $C\subset(\R^{n+1})^\ast$ be a good cone defined by
\begin{equation} \label{eq:cone}
C = \bigcap_{j=1}^d \{x\in(\R^{n+1})^\ast\mid
\ell_j (x) := \langle x, \nu_j \rangle \geq 0\}\,
\end{equation}
where $d\geq n+1$ is the number of facets and each $\nu_j$ is a primitive 
element of the lattice $\Z^{n+1} \subset \R^{n+1}$ (the inward-pointing 
normal to the $j$-th facet of $C$).

Let $(e_1, \ldots, e_d)$ denote the standard basis of $\R^d$, and define
a linear map $\beta : \R^d \to \R^{n+1}$ by 
\begin{equation} \label{def:beta}
\beta(e_j) = \nu_j\,,\ j=1,\ldots,d\,. 
\end{equation}
The conditions of Definition~\ref{def:good} imply that
$\beta$ is surjective. Denoting by $\fk$ its kernel, we have short
exact sequences
\[
0 \to \fk \stackrel{\iota}{\to} \R^d \stackrel{\beta}{\to}
\R^{n+1} \to 0
\ \ \ \mbox{and its dual}\ \ \ 
0 \to (\R^{n+1})^\ast \stackrel{\beta^\ast}{\to} (\R^d)^\ast 
\stackrel{\iota^\ast}{\to}\fk^\ast \to 0\ .
\]
Let $K$ denote the kernel of the map from $\T^d = \R^d/2\pi\Z^d$ to
$\T^{n+1} = \R^{n+1}/2\pi\Z^{n+1}$ induced by $\beta$. More precisely,
\begin{equation} \label{eq:K}
K = \left\{ [y]\in \T^d\mid \sum_{j=1}^{d} y_j \nu_j
\in 2\pi\Z^n\right\}\,.
\end{equation}
It is a compact abelian subgroup of $\T^d$ with Lie algebra 
$\fk = \ker (\beta)$. 

Consider $\R^{2d}$ with its standard symplectic form
\[
\om_{\rm st} = du\wedge dv = \sum_{j=1}^d du_j\wedge dv_j
\]
and identify $\R^{2d}$ with $\C^d$ via $z_j = u_j + i v_j\,,\ 
j=1,\ldots,d$. The standard action of $\T^d$ on $\R^{2d}\cong
\C^d$ is given by
\[
y \cdot z = \left( e^{i y_1} z_1, \ldots, e^{i y_d} z_d\right)
\]
and has a moment map given by
\[
\phi_{\T^d} (z_1,\ldots,z_d) = \sum_{j=1}^d \frac{|z_j|^2}{2}\, e_j^\ast 
\in (\R^d)^\ast\,.
\]
Since $K$ is a subgroup of $\T^d$, $K$ acts on $\C^d$ with moment map
\begin{equation}\label{def:phiK}
\phi_K = \iota^\ast \circ \phi_{\T^d} =
\sum_{j=1}^d \frac{|z_j|^2}{2} \iota^\ast(e_j^\ast)\in \fk^\ast\ .
\end{equation}

The toric symplectic cone $(W_C,\om_C, X_C)$ associated to the good
cone $C$ is the symplectic reduction of 
$(\R^{2d}\setminus\{0\}, \om_{\rm st} = du \wedge dv, X_{\rm st} =  u\, \partial / \partial u 
+ v\, \partial / \partial v)$ with respect to the $K$-action, i.e.
\[
W_C = Z / K\ \ \mbox{where}\ \ Z=\phi_K^{-1}(0) \setminus\{0\}
\equiv\ \mbox{zero level set of the moment map in $\R^{2d}\setminus\{0\}$,}
\]
the symplectic form $\om_C$ comes from $\om_{\rm st}$ via symplectic 
reduction, while the $\R$-action of the Liouville vector field $X_C$ and 
the action of $\T^{n+1} \cong \T^d/K$ are induced by the actions of 
$X_{\rm st}$ and $\T^d$ on $Z$.

Lerman showed in~\cite{Le2} how to compute the fundamental group of
a good toric symplectic cone, which is canonically isomorphic to the
fundamental group of the associated good toric contact manifold.
\begin{prop} \label{prop:pi1} 
Let $W_C$ be the good toric symplectic cone determined by a good
cone $C\subset\R^{n+1}$. Let $\Nn := \Nn\{\nu_1, \ldots, \nu_d\}$
denote the sublattice of $\Z^{n+1}$ generated by the primitive integral 
normal vectors to the facets of $C$. The fundamental group of $W_C$ 
is the finite abelian group
\[
\Z^{n+1}/\Nn\,.
\]
\end{prop}
Hence, for simply connected good toric contact manifolds we have that
the $\Z$-span of the set of integral normals $\{\nu_1, \ldots, \nu_d\}$ is
the full integer lattice $\Z^{n+1}\subset\R^{n+1}$

\subsection{Cylindrical contact homology of good toric contact manifolds}
\label{ssection:homology}

In this section we provide the necessary information, taken from~\cite{AM}
and needed for the proof of Theorem \ref{thm:convex}, on how to compute the 
cylindrical contact homology of a good toric contact manifold $(M,\xi)$. We will
also show that $k_-$ is always finite and can be used to provide a lower bound 
for the Robbin-Salamon index of the generic orbit of any periodic Reeb flow on 
$(M,\xi)$ that comes from an $S^1$-subaction of the toric $\T^{n+1}$-action.

Let $(W,\omega,X)$ be the good toric symplectic cone associated with $(M,\xi)$
and $C\subset (\R^{n+1})^\ast$ its good moment cone defined by~(\ref{eq:cone}).
Consider a toric Reeb vector field $R_\nu \in \Xx_X (W,\omega) \cong \Xx (M,\xi)$
determined by
\[
\nu = \sum_{j=1}^d a_j \nu_j \quad\text{with $a_j\in\R^+$ for all
$j=1, \ldots, d$,}
\]
as in Proposition~\ref{prop:sasaki}. By a small abuse of notation, we will also write
\[
R_\nu = \sum_{j=1}^d a_j \nu_j \,.
\]
Moreover, we will assume that
\[
\text{the $1$-parameter subgroup generated by $R_\nu$ is dense in $\T^{n+1}$,}
\]
which is equivalent to the corresponding toric contact form being non-degenerate. 
In fact, the toric Reeb flow of $R_\nu$ on $(M,\xi)$ has exactly $m$ simple closed
orbits $\gamma_1, \ldots,\gamma_m$, all non-degenerate, corresponding to the $m$ edges
$E_1,\ldots,E_m$ of the cone $C$. For each $\ell =1, \ldots, m$, we can use the
symplectic reduction construction of $(W, \omega, X)$ to lift the $X$-invariant
Hamiltonian flow of $R_\nu$ to a linear flow on $\C^d = \R^{2d}$ that has a periodic
orbit $\tgamma_\ell$ as a lift of $\gamma_\ell$. It then follows from Lemma 3.4 
in~\cite{AM} that
\[
\rs (\gamma_\ell^N) = \rs (\tgamma_\ell^N)\,,\ 
\text{for all $\ell =1,\ldots,m$, and all iterates $N\in\N$,}
\]
where $\rs$ is the Robbin-Salamon extension of the Conley-Zehnder index.

To compute $\rs (\tgamma_\ell^N)$ one can use the global trivialization of 
$T\R^{2d}$ and the fact that, since the lifted flow is given by the standard action
on $\R^{2d}$ of a $1$-parameter subgroup of $\T^d$, the index can be directly
computed from the corresponding Lie algebra vector $\tilde{R}_\nu^\ell \in \R^d$,
which necessarily satisfies
\[
\beta (\tilde{R}_\nu^\ell) = R_\nu\,,\ 
\text{where $\beta:\Z^d \to \Z^{n+1}$ is the linear surjection defined by~(\ref{def:beta}).}
\]
In fact, denoting by $F_{\ell_1},\ldots, F_{\ell_n}$, the $n$ facets of $C$ that meet at
the edge $E_\ell$, we have that
\[
\tilde{R}_\nu^\ell = \sum_{j=1}^n b_j^\ell e_{\ell_j}  + b^\ell \tilde{\eta}_\ell\,,\ 
\text{with $\tilde{\eta}_\ell\in\Z^d$ and $b_1^\ell, \ldots, b_n^\ell, b^\ell \in \R$.}
\]
Since we are assuming that $R_\nu$ generates a dense $1$-parameter subgroup of
$\T^{n+1}$, we have that the $n+1$ real numbers $\{b_1^\ell , \ldots, b_n^\ell, b^\ell\}$
are $\Q$-independent, which implies that
\[
\rs (\tgamma_\ell^N) = \text{even}\, + \, n\,, \ \forall N\in\N\, .
\]
Hence, the parity of the Robbin-Salamon index is the same for all periodic orbits
of $R_\nu$.

We conclude that the cylindrical contact homology of every toric non-degenerate contact form $\alpha_\nu$ is well defined since its differential vanishes identically. Consequently, $HC_\ast (\alpha_\nu)$ is isomorphic to the chain complex generated by the closed orbits of $\alpha_\nu$ (note that every periodic orbit of $\alpha_\nu$ is good). In particular, all of its periodic orbits are homologically essential, i.e. non-zero in $HC_\ast (\alpha_\nu)$.

\begin{proposition} \label{prop:finite}
If $(M,\xi)$ is a good toric contact manifold, then there exists a toric non-degenerate contact form $\alpha_\nu$ such that
\[
k_- := \inf\{k \in \Z\mid HC_k(\alpha_\nu) \neq 0\} 
\]
is finite.
\end{proposition}
\begin{proof}
It is enough to show that there exists a toric non-degenerate contact form $\alpha_\nu$ such that every periodic of its Reeb flow $R_\nu$ has positive mean index (cf. Section~\ref{sec:index_orbits} for the definition and continuity property of the mean index).

Let us consider the (degenerate) toric Reeb vector field $R_1$ given by
\[
R_1 = \sum_{j=1}^d \nu_j \in \R^{n+1}\,.
\]
A possible lift to $\R^d$ is given by
\[
\tilde{R}_1 = \sum_{j=1}^d e_j \,,
\]
which generates the standard diagonal $2\pi$-periodic $S^1$-action on $\C^d$. Its
$2\pi$-periodic orbits have Robbin-Salamon index $2d$, which implies that any
$2\pi$-periodic orbit of $R_1$ has Robbin-Salamon index $2d$.
Since $R_1$ generates an $S^1$-action on $(M,\xi)$, we conclude that any such
$2\pi$-periodic orbit has mean index $2d$. Moreover,
\[
\forall\  \text{simple periodic orbit}\ \gamma\text{ of } R_1\,,\  \ \exists\  N\in\N \ \text{such that}\ \gamma^N \ \text{is $2\pi$-periodic.} 
\]
The mean index of such $\gamma$ is then equal to $2d/N$. Hence, we conclude
that all periodic orbits of $R_1$ have positive mean index. 

By continuity of mean index, the same is true
for any sufficiently small perturbation of $R_1$, which can then be chosen to give
a non-degenerate $R_\nu$ with the required property.
\end{proof}

In the proof of Theorem \ref{thm:convex} we will need the following estimate.

\begin{proposition} \label{prop:estimate2}
Let $\nu\in\Z^{n+1}$ be a primitive integral Reeb vector and $R_\nu$ the corresponding toric Reeb
vector field. Denote by $\gamma_\nu$ any $2\pi$-periodic Reeb orbit of the $2\pi$-periodic Reeb
flow on $(M,\xi)$ generated by $R_\nu$. Then there exists a toric non-degenerate contact form $\alpha_{\nu_\ep}$,
with $\nu_\ep \in \R^{n+1}$ arbitrarily close to $\nu$, such that
\[
\rs (\gamma_\nu) - n \geq k_-\,,
\] 
where $k_- :=  \inf\{k \in \Z\mid HC_k(\alpha_{\nu_\ep}) \neq 0\}$.
\end{proposition}
\begin{proof}
Since $\nu\in\Z^{n+1}$ we can express $R_\nu$ as
\[
R_\nu = \sum_{j=1}^d b_j \nu_j \ \text{with}\ b_j \in \Z^{n+1}\,,
\]
and a possible lift to $\R^d$ is given by
\[
\tilde{R}_\nu = \sum_{j=1}^d b_j e_j \,.
\]
Such $\tilde{R}_\nu$ generates a $2\pi$-periodic flow on $\C^d$ with
\[
\rs (\tgamma_\nu) = \sum_{j=1}^d 2b_j\,,\ \text{where $\tgamma_\nu$ is any $2\pi$-periodic orbit of $\tilde{R}_\nu$.}
\]

Choose any edge $E$ of the good moment cone $C$ of the symplectization
$(W,\omega,X)$ of $(M,\xi)$. As pointed out above, the edge $E$ determines a simple closed orbit 
$\gamma_\nu^E$ of the toric Reeb vector $R_\nu$ on $(M,\xi)$. Since the flow of $R_\nu$ is $2\pi$-periodic, 
by taking a multiple of $\gamma_\nu^E$ if necessary, we can assume that $\gamma_\nu^E$ is a
(possibly non-simple) $2\pi$-periodic orbit of $R_\nu$. Denote by $\tgamma_\nu^E$ its lift as a $2\pi$-periodic
orbit of $\tilde{R}_\nu$ on $\C^d$. We then have that
\[
\rs (\gamma_\nu) = \rs (\gamma_\nu^E) = \rs (\tgamma_\nu^E)  = \sum_{j=1}^d 2b_j\,.
\]

Assume, without any loss of generality, that the normals to the $n$ facets of $C$ that meet 
at $E$ are  $\nu_1, \ldots,\nu_n \in \Z^{n+1}$. It then follows from the explicit form of the
symplectic reduction construction described above that
\[
\tgamma_\nu^E \subset Z \cap \{ z\in\C^d\mid \ z_1 = \cdots = z_n = 0 \}\,.
\]
This implies that any small perturbation of $\tilde{R}_\nu$ of the form
\[
\tilde{R}_{\nu_\ep} = \sum_{j=1}^n (b_j -\ep_j) e_j + \sum_{j=n+1}^d b_j e_j\,,\ 
\ep = (\ep_1, \ldots, \ep_n)\in (\R^+)^n\,,
\]
has $\tgamma_{\nu_\ep}^E := \tgamma_{\nu}^E$ as one of its $2\pi$-periodic orbits.
For sufficiently small $\ep$, the Robbin-Salamon index of $\tgamma_{\nu_\ep}^E$ is
given by
\[
\rs (\tgamma_{\nu_\ep}^E) = \sum_{j=1}^n (2 \lfloor b_j -\ep_j\rfloor +1)
+ \sum_{j=n+1}^d 2 b_j = \sum_{j=1}^n (2(b_j -1) +1) + \sum_{j=n+1}^d 2b_j =
\rs (\tgamma_{\nu}^E) - n\,,
\]
where $\lfloor b \rfloor := \max \{m\in\Z\mid  m\leq b \}$. For sufficiently small and generic
$\ep = (\ep_1, \ldots, \ep_n)\in (\R^+)^n$, where here generic means 
$\ep_1,\ldots,\ep_n \in \R^+ \setminus \Q^+$ are rationally independent, the vector
\[
R_{\nu_\ep} = 
\sum_{j=1}^n (b_j - \ep_j) \nu_j + \sum_{j=n+1}^d b_j \nu_j
\]
is still a toric Reeb vector, which is non-degenerate and has a $2\pi$-periodic orbit
$\gamma_{\nu_\ep}^E$ with a lift to $\C^d$ given by $\tgamma_{\nu_\ep}^E$.
We then have that
\[
\rs (\tgamma_\nu^E) - n = \rs (\tgamma_{\nu_\ep}^E) =
\rs (\gamma_{\nu_\ep}^E) \geq k_-\,,
\]
where the last inequality follows from the already pointed out fact that any closed orbit of a 
non-degenerate toric Reeb vector, such as $R_{\nu_\ep}$, is non-zero in $HC_\ast (\alpha_{\nu_\ep})$.
\end{proof}


Note that $k_-$ is an invariant of the contact structure once the foundational transversality issues in contact homology are resolved. However, this is not essential here and under suitable conditions one can use, for instance, equivariant symplectic homology to bypass these transversality issues.

\subsection{Examples}

The standard contact sphere $(S^{2n+1}, \xi_{\rm st})$ is the most basic example of
a good toric contact manifold. Its symplectization is 
$(\R^{2(n+1)}\setminus \{ 0\}, \omega_{\rm st}, X_{\rm st})$ and its moment cone has
$d=n+1$ facets determined by the normals
\[
\nu_j = e_j\,,\ j=1,\ldots,n\,,\ \text{and}\ \nu_{n+1} = e_{n+1} - \sum_{j=1}^n e_j \,,
\]
where $\{e_1, \ldots, e_{n+1}\}$ is the standard basis of $\R^{n+1}$. The corresponding
symplectic reduction construction is trivial, i.e. the map $\beta$ defined by~(\ref{def:beta})
is an isomorphism and $K\subset\T^d$ is trivial.

In~\cite{AM} we discussed a family of good toric contact structures
$\xi_k$ on $S^2 \times S^3$, originally considered in~\cite{GW1}. Their good moment 
cones $C(k)\subset\R^3$, $k \in \{0, 1, 2, \ldots\}$, have four facets determined by the normals
\[
\nu_1 = (1,0,1)\,,\ \nu_2 = (0,-1,1)\,,\ ,\nu_3 = (0,k,1)\ \text{and}\ \nu_4 = (-1,2k-1,1)\,.
\]
The corresponding  symplectic reduction is non-trivial, with $K$ a circle subgroup 
of $\T^4$. Their cylindrical contact homology can be explicitly computed, using the 
method described in the previous section, with the following result:
\[
\rank HC_\ast (S^2  \times S^3, \xi_k ) =
\begin{cases}
k & \text{if $\ast = 0$;} \\
2k+1  & \text{if $\ast = 2$;} \\
2k+2 & \text{if $\ast > 2$ and even;} \\
0 & \text{otherwise.}
\end{cases} 
\]
This shows that
\[
k_- (S^2 \times S^3, \xi_0) = 2 \quad\text{while}\quad k_- (S^2 \times S^3, \xi_k) = 0\,,
\ \forall\, k>0\,.
\]
Note that $(S^2 \times S^3, \xi_0)$ can be identified with the standard contact structure
on the unit cosphere bundle of $S^3$, which is also the Boothby-Wang contact manifold
over the monotone symplectic manifold $(S^2 \times S^2, \omega =\sigma \times \sigma)$,
where $\sigma (S^2) = 2\pi$.

\section{Proof of Theorem \ref{thm:convex}} 
\label{sec:proof convex}

As before, we let $(M,\xi)$ denote a good closed toric contact manifold and
$(W,\om,X)$ its associated good toric symplectic cone, obtained via symplectic
reduction of $(\C^d\setminus\{0\}\cong \R^{2d}\setminus\{0\}, \om_{\rm st}, X_{\rm st})$ by the linear
action of a subtorus $K\subset\T^d$ with moment map 
$\phi_K: \C^{d}\setminus\{0\} \to \fk^\ast$. To simplify notation, let us define
$F:= \phi_K$. Then
\[
W = Z/K \quad\text{where}\quad Z := F^{-1} (0)\,.
\]
If $K =\{1\}$ is the trivial subgroup of $\T^d$, we will take $F\equiv 0$ so that
$(W,\om, X) = (\R^{2d}\setminus\{0\}, \om_{\rm st}, X_{\rm st})$ and
$(M,\xi) = (S^{2d-1}, \xi_{\rm st})$.

Let $\alpha$ be an arbitrary, i.e. not necessarily toric, contact form on $(M, \xi)$,
and denote by $R_\alpha$ its Reeb vector field. We can always find a $K$-invariant
function $H_\alpha: \C^{d}\setminus\{0\} \to \R$ such that its reduced Hamiltonian flow 
is the flow of $R_\alpha$. Notice that $H_\alpha$ is not unique, since it is only determined
on $Z$ (up to a constant). In Theorem \ref{thm:convex} we assume that $M$ is simply connected and that 
$\alpha$ is a convex contact form, i.e. $H_\alpha$ can be chosen to be convex on $Z$:
\[
d^2_p (H_\alpha) > 0\,,\ \forall\ p\in Z\,.
\]
Our goal in this section is to prove that any such contact form $\alpha$ is positively dynamically 
convex. Based on the discussion in the previous section, it suffices to prove that every periodic orbit 
$\ga$ of $\alpha$ satisfies $\czl(\ga) \geq k_- :=  \inf\{k \in \Z\mid HC_k(\alpha_\nu) \neq 0\}$ for some 
toric non-degenerate contact form $\alpha_\nu$. We will do that according to the following main steps and
using the following notational convention: a subscript $\alpha$ indicates something not necessarily
torus invariant, like the contact form $\alpha$ itself, while a subscript $\lambda$ indicates something
torus invariant.
\begin{itemize}
\item[(1)] Given a periodic orbit $\gamma$ of $\alpha$ we change the convex Hamiltonian $H_\alpha$ by an appropriate component $F_\lambda$ of
the moment map $F$ for the linear $K$-action on $\C^d$, so that $\ga$ has a lift $\tgamma$ in $Z\subset \C^d$ which is a periodic orbit of the Hamiltonian flow of
\[
H_{\alpha,\lambda} := H_\alpha - F_\lambda.
\]
Note that such an $F_\lambda : \C^d \to \R$ is quadratic, being a component of the moment map for
the standard action on $\C^d$ of the subtorus $K\subset \T^d$. Moreover, it follows from  Lemma 3.4 
in~\cite{AM} that
\[
\rs (\gamma) = \rs (\tgamma)\,.
\]
(Here we are using the assumption that the linearized Hamiltonian flow of $H_\alpha$ along $\gamma$ satisfies $d\Phi^t_{H_\alpha}(\nabla F_\kappa)=\nabla F_\kappa$ for all $t \in \R$ and every component $F_\kappa$ of the moment map: we need this hypothesis to apply \cite[Lemma 3.4]{AM}.)
\item[(2)] Find an appropriate convex quadratic Hamiltonian $H_\lambda$ on $\C^d$ such that, for any
simple periodic orbit $\tgamma_\lambda$ of the linear flow on $\C^d$ generated by the quadratic Hamiltonian
\[
Q_\lambda := H_\lambda - F_\lambda\,,
\]
we have 
\[
\rs (\tgamma) \geq \rs (\tgamma_\lambda) - n
\]
whenever $\gamma$ is non-degenerate.
\item[(3)] Combining the previous two steps we have that (cf. Proposition~\ref{prop:estimate} below)
\[
\czl (\gamma) \geq \rs (\tgamma_\lambda) - n
\]
and Theorem \ref{thm:convex} follows by applying Proposition~\ref{prop:estimate2} to the periodic flow induced
on $(M,\xi)$ by reduction of $Q_\lambda$.
\end{itemize}

So fix from now on a periodic orbit $\ga$ of $\alpha$. Suppose, without loss of generality, that the period of $\ga$ is one. Denote by $\gamma'$ one of the possible lifts of $\ga$ to $Z\subset \C^{d}\setminus\{0\}$ as an orbit of the
Hamiltonian flow of $H_\alpha$. Such a lift $\gamma'$ is not necessarily a closed orbit but we know
that there exists $\lambda\in K$ such that
\[
\gamma' (1) = \lambda\cdot\gamma' (0)\,.
\]
Moreover, for a suitably normalized $\eta_\lambda\in\fk$, the Hamiltonian $F_\lambda :\C^d \to \R$ 
defined by
\[
F_\lambda (z) = \langle F(z) , \eta_\lambda \rangle\,,
\]
where $\langle \cdot,\cdot \rangle$ denotes the natural pairing between $\fk$ and $\fk^\ast$,
has a time-$1$ Hamiltonian flow $\Phi_{F_\lambda}:\C^d \to \C^d$ given by
$\Phi_{F_\lambda} (z ) = \lambda \cdot z\,, \ \forall\ z\in\C^d$. We can then consider the
Hamiltonian
\[
H_{\alpha,\lambda} := H_\alpha - F_\lambda \,.
\]
Since $H_\alpha$ is $K$-invariant, we have that $H_\alpha$ and $F_\lambda$ Poisson
commute, i.e. $\{H_\alpha, F_\lambda\} = 0$, and the time-$1$ flow of $H_{\alpha,\lambda}$
is given by
\[
\Phi _{H_{\alpha,\lambda}} = \left(\Phi_{F_\lambda}\right)^{-1} \circ  \Phi_{H_\alpha}\,.
\]
This means that
\[
\Phi _{H_{\alpha,\lambda}} (\gamma'(0)) = 
\left(\Phi_{F_\lambda}\right)^{-1} (\Phi_{H_\alpha} (\gamma'(0))) =
\left(\Phi_{F_\lambda}\right)^{-1} (\lambda \cdot \gamma'(0)) =
\lambda^{-1} \cdot \lambda \gamma'(0) = \gamma'(0) \,.
\]
Hence, the Hamiltonian flow $\Phi _{H_{\alpha,\lambda}}^t$ of $H_{\alpha,\lambda}$ has a 
periodic orbit which is a lift of $\gamma\in\Pp_\alpha$. We will denote by $\tgamma$ this closed 
lift of $\gamma$. 

To prove Theorem \ref{thm:convex} we need to obtain an appropriate estimate
for the Robbin-Salamon index of $\tgamma$, which is equal to the index of the path of
linear symplectic matrices given by
\[
d (\Phi_{-F_\lambda}^t)\circ d (\Phi_{H_\alpha}^t )\,. 
\]
This path is homotopic with fixed end points to the juxtaposition $d (\Phi_{H_\alpha}^t)
\ast d (\Phi_{-F_\lambda}^t)$, i.e. $d (\Phi_{-F_\lambda}^t)$ followed by $d (\Phi_{H_\alpha}^t)$.
An explicit homotopy, parametrized by $s\in [0,1]$, is given by the following formula:
\[
\begin{cases}
d (\Phi_{-F_\lambda}^{(1+s)t})\circ d (\Phi_{H_\alpha}^{(1-s)t} )\,,\ t\in \left[0,1/2\right] \\
d (\Phi_{H_\alpha}^{(1+s)t-s})  \circ d (\Phi_{-F_\lambda}^{(1-s)t + s}) \,,\ t\in \left[1/2, 1\right]\,.
\end{cases}
\]
This means that the Robbin-Salamon index of $\tgamma$ is given by
\begin{equation} \label{eq:index1}
\begin{split}
\rs(\tgamma) & = 
\rs (d (\Phi_{H_\alpha}^t) \ast d (\Phi_{-F_\lambda}^t)) \\
& = \czl(d (\Phi_{H_\alpha}^t) \ast d (\Phi_{-F_\lambda}^t)) + (\dim K) + 1 + \frac{\nu(\gamma)}{2}
\end{split}
\end{equation}
where the second equality follows from convexity of $H_\alpha$ and
\begin{itemize}
\item $ \czl$ is defined and discussed in Section~\ref{sec:CZ}. It is the lower 
semicontinuous extension of the Conley-Zehnder index and coincides with the index 
defined by Hofer-Wysocki-Zehnder~\cite{HWZ} and Long~\cite{Lon02} (cf.~\eqref{eq:hwz=czl}).
Since $H_\alpha$ is convex, it follows from~\eqref{eq:index_monotonicity2} 
that $\czl (d (\Phi_{H_\alpha}^t) \ast d (\Phi_{-F_\lambda}^t))$
is equal to the Robbin-Salamon index of this juxtaposed path of linear symplectic matrices right before
it reaches its endpoint.
\item $(\dim K) + 1 =$ half the dimension of the $1$-eigenspace due to $K$-equivariance and the
autonomous Hamiltonian vector field.
\item $\nu(\gamma) =$ dimension of the $1$-eigenspace due to possible degeneracy of $\gamma$.
\end{itemize}

To obtain the appropriate index estimate, we will compare the above juxtaposed path with the following
one, given by quadratic Hamiltonians. Write $\lambda\in K \subset\T^d$ as $\lambda = (\lambda_1, \ldots,\lambda_d)$, with
$0<\lambda_i\leq 2\pi\,,\ i=1,\ldots,d$, and consider the \emph{convex} quadratic Hamiltonian
$H_\lambda :\C^d \to \R$ given by
\[
H_\lambda (z_1, \ldots, z_d) = \frac{1}{2} \sum_{i=1}^d \lambda_i  |z_i|^2\,.
\]
Denote its time-$1$ Hamiltonian flow by $\Phi_{H_\lambda}:\C^d \to \C^d$.  We have
that $\Phi_{H_\lambda} (z ) = \lambda \cdot z\,, \ \forall\ z\in\C^d$. The Hamiltonian flow of 
$H_\lambda$ obviously commutes with the action of $K$ and induces a flow on the contact 
manifold $M$, which is the Reeb flow of the toric Reeb vector field
\[
R_\lambda := \sum_{i=1}^d \lambda_i \nu_i \,.
\]
Note that, since $\lambda\in K$, it follows from~(\ref{eq:K}) that $R_\lambda/2\pi$ is an integral
toric Reeb vector, i.e. $R_\lambda/2\pi \in\Z^{n+1}$. We can now consider the quadratic Hamiltonian
\[
Q_{\lambda} := H_\lambda - F_\lambda \,.
\]
Its flow $\Phi_{Q_\lambda}^t = \Phi_{H_\lambda}^t \circ \Phi_{-F_\lambda}^t$ is a lift to $\C^d$ of 
the Reeb flow of $R_\lambda$ and satisfies
\[
\Phi_{Q_\lambda}^1 (z) = \Phi_{H_\lambda}^1 \left( \Phi_{-F_\lambda}^{1} (z) \right)
= \lambda \cdot \frac{1}{\lambda} \cdot z = z \,.
\]
The associated path $d (\Phi_{-F_\lambda}^t)\circ d (\Phi_{H_\lambda}^t )$ of linear symplectic
matrices is also homotopic with fixed end points to the juxtaposition $d (\Phi_{H_\lambda}^t) \ast 
d (\Phi_{-F_\lambda}^t)$ and its Robbin-Salamon index, which is also the Robbin-Salamon index
of any $1$-periodic orbit $\tgamma_\lambda$ of $\Phi_{Q_\lambda}^t$, is given by
\begin{equation} \label{eq:index2}
\rs (\tgamma_\lambda) = \czl (d (\Phi_{H_\lambda}^t) \ast d (\Phi_{-F_\lambda}^t)) +
(\dim K) + 1 + n \,.
\end{equation}

\begin{lemma} \label{lem:index3}
\[
\czl (d (\Phi_{H_\lambda}^t) \ast d (\Phi_{-F_\lambda}^t)) =
\czl (d (\Phi_{\delta H_\lambda}^t) \ast d (\Phi_{-F_\lambda}^t))\,, \ 
\text{for any $0<\delta\leq 1$}.
\]
\end{lemma}

\begin{proof}
The fact that
\[
\Phi_ {H_\lambda}^t (z) = d (\Phi_{H_\lambda}^t) (z) =
(e^{it\lambda_1 /2\pi} z_1, \ldots, e^{it\lambda_d /2\pi} z_d)\,,\
\text{with $0 < \lambda_i \leq 2\pi$, $i=1,\ldots,d$,}
\]
implies that the path
\[
d (\Phi_{H_\lambda}^{2t-1}) \circ d (\Phi_{-F_\lambda}^1) =
(e^{i(t-1) \lambda_1 /\pi} z_1, \ldots, e^{i(t-1) \lambda_d /\pi} z_d)
\]
has no crossing, i.e. no eigenvalue $1$, when $1/2 < t < 1$. This, together with
the convexity of $H_\lambda$ (which implies that all crossings at $t=1$ are positive)
and the lower semicontinuity property of $\mu_{CZ}^-$, implies the result.
\end{proof}

\begin{proposition} \label{prop:estimate}
We have that
\[
\czl (\gamma) \geq \rs (\tgamma_\lambda) - n \,,
\]
where $\tgamma_\lambda$ denotes any $1$-periodic orbit  of the
linear flow $\Phi_{Q_\lambda}^t$ on $\C^d$ generated by the
quadratic Hamiltonian $Q_\lambda = H_\lambda - F_\lambda$.
\end{proposition}

\begin{proof}
Suppose initially that $\gamma$ is non-degenerate. We have from~(\ref{eq:index1}) that
\[
\czl(\ga) = \rs(\ga) = \rs (\tgamma) = \czl (d (\Phi_{H_\alpha}^t) \ast d (\Phi_{-F_\lambda}^t)) +
(\dim K) + 1\,.
\]
Convexity of $H_\alpha$ implies that there exists a sufficiently small $\delta >0$ such that
\[
d^2_p (H_\alpha) > d^2_p (\delta H_\lambda)\,,\ \text{for any $p\in\tgamma$.}
\]
It then follows from~(\ref{eq:index_monotonicity2}) that
\[
\czl (d (\Phi_{H_\alpha}^t) \ast d (\Phi_{-F_\lambda}^t)) \geq 
\czl (d (\Phi_{\delta H_\lambda}^t) \ast d (\Phi_{-F_\lambda}^t)) \,.
\]
The result now follows from Lemma~\ref{lem:index3} and~(\ref{eq:index2}).

Finally, let us consider the case that $\ga$ is degenerate. Given a symplectic path $\Gamma \in \P(2n)$ one can always make a small $C^\infty$-perturbation of $\Gamma$ such that the new path $\Gamma'$ is non-degenerate and satisfies $\rs(\Gamma') = \czl(\Gamma)$; see \cite{Lon02}. Therefore, one can make a small perturbation of the path $d (\Phi_{H_\alpha}^t)$ to a path $\tilde\Gamma$ (commuting with the action of $K$) such that $\rs(\tilde\Gamma \ast d (\Phi_{-F_\lambda}^t)) = \czl(\gamma)$ and $\nu(\tilde\Gamma(1) \circ d (\Phi_{-F_\lambda}^1)) =  \dim K + 1$ (which means that the reduced symplectic path in $\R^{2n}$ is non-degenerate). Then it follows from \eqref{eq:index1} and the previous argument that
\[
\czl(\gamma) = \rs(\tilde\Gamma \ast d (\Phi_{-F_\lambda}^t)) = \czl(\tilde\Gamma \ast d (\Phi_{-F_\lambda}^t)) + (\dim K) + 1
\]
and
\[
\czl (\tilde\Gamma \ast d (\Phi_{-F_\lambda}^t)) \geq 
\czl (d (\Phi_{\delta H_\lambda}^t) \ast d (\Phi_{-F_\lambda}^t)) \,.
\]
This, together with Lemma~\ref{lem:index3} and~(\ref{eq:index2}), proves the result in the degenerate case.
\end{proof}

To complete the proof of Theorem \ref{thm:convex}, we can now apply the Robbin-Salamon index estimate 
given in Proposition~\ref{prop:estimate2} to the periodic flow induced on $(M,\xi)$ by the reduction of $\Phi_{Q_\lambda}^t$.

\end{document}